\documentclass[10pt, reqno]{amsart}
\numberwithin{equation}{section}

\usepackage{amssymb,latexsym}
\usepackage{eucal}
\usepackage{tikz}

\newtheorem{thm}{Theorem}[section]

\newtheorem{lem}[thm]{Lemma}

\newtheorem*{TA}{Theorem A}
\newtheorem*{TB}{Theorem B}

\newtheorem*{T1}{Theorem 1.1}
\newtheorem*{T2}{Theorem 1.2}
\newtheorem*{T3}{Theorem 1.3}

\newcommand{\al}{\alpha}
\newcommand{\btt}{\theta}

\newcommand{\om}{\Omega}

\newcommand{\omd}{\omega_{D}}

\newcommand{\A}{\mathcal{A}}
\newcommand{\ph}{\varphi}
\newcommand{\Om}{{\Omega}}

\newcommand{\FF}{\mathcal{F}}

\newcommand{\PP}{\mathcal{P}}

\newcommand{\Ss}{\mathcal{S}}

\newcommand{\wtl}{\widetilde}
\newcommand{\wht}{\widehat}
\newcommand{\KK}{ \mathcal{K}_{\textsf{ff}}(\mathcal{F})}
\newcommand{\KKd}{ \mathcal{K}_{\textsf{ff}}(\mathcal{F}, d)}

\DeclareMathOperator{\core}{\textnormal{core}}

\DeclareMathOperator{\SLI}{\textnormal{SLI}}
\DeclareMathOperator{\inq}{\textnormal{{inq}}}
\DeclareMathOperator{\sol}{\textnormal{{sol}}}

\newcommand{\brr}{\bar {\mathrm{r}}}
\newcommand{\rr}{\mathrm{r}}

\begin{document}
\title
[Linear programming and the intersection of free subgroups]
{Linear programming and the intersection of free subgroups in free products of groups}
\author{S. V. Ivanov }
 \address{  Department of Mathematics\\
 University of Illinois \\
 Urbana\\  IL 61801\\ U.S.A. } \email{ivanov@illinois.edu}
\thanks{Supported in part by  the NSF under  grant  DMS 09-01782.}

\begin{abstract}
We study the intersection of finitely generated factor-free subgroups of free products of  groups by utilizing the method of linear programming.  For example, we prove that if $H_1$ is a finitely generated
factor-free noncyclic subgroup of the free product $G_1 * G_2$   of two finite groups $G_1$, $G_2$, then the WN-coefficient $\sigma(H_1)$ of $H_1$ is rational and can be computed in exponential time {}in the size of $H_1$. This coefficient $\sigma(H_1)$ is the minimal positive real number such that, for every  finitely generated
factor-free subgroup $H_2$ of  $G_1 * G_2$, it is true that
$\bar \rr(H_1, H_2)  \le  \sigma(H_1) \bar \rr(H_1) \bar \rr(H_2)$, where $\bar{ {\rm r}} (H) = \max ( {\rm r} (H)-1,0)$ is the reduced rank of $H$, $\rr (H)$ is the rank of $H$, and  $\bar \rr(H_1, H_2)$ is the reduced rank of the generalized intersection of $H_1$ and $H_2$. In the case of the free product $G_1 * G_2$   of two finite groups $G_1$, $G_2$, it is also proved that there exists a  factor-free subgroup $H_2^* = H_2^*(H_1)$ such that
 $\bar \rr(H_1, H_2^*)  =  \sigma(H_1) \bar \rr(H_1) \bar \rr(H_2^*)$, $H_2^*$ has at most doubly exponential size {}in the size of $H_1$, and $H_2^*$ can be constructed in exponential time {}in the size of $H_1$.
\end{abstract}
\maketitle
\tableofcontents

\section{Introduction}

Let $G_\al$, $\al \in I$, be some nontrivial groups and let
 $\FF = \prod_{\alpha \in I}^* G_\alpha$ denote the free product of these groups. According to the classic Kurosh subgroup
theorem  \cite{K}, \cite{LS}, every subgroup $H$ of $\FF$ is a free product $ F(H) * \prod^*  t_{{\al,\gamma}} H_{\al,\gamma} t_{{\al,\gamma}}^{-1}$, where $H_{{\al,\gamma}}$ is a subgroup of $G_\al$, $t_{{\al,\gamma}} \in
\FF$, and $F(H)$ is a free subgroup of
$\FF$  such that, for every $s \in \FF$
and $\gamma \in I$,  it is true that $F(H) \cap
s G_\gamma s^{-1} =\{ 1 \}$. We say that $H$
is a {\em factor-free} subgroup of  $\FF$  if $H=F(H)$ in the above form of $H$, i.e., for every $s \in
\FF$ and $\gamma \in I$, we have $H \cap s
G_\gamma s^{-1} =\{ 1 \}$. Let ${\rm r} (F)$ denote the rank of a (finitely generated)
 free group $F$.   Since a factor-free subgroup $H$ of
$\FF$  is free,  the  reduced
rank $\bar{ {\rm r}} (H) := \max ( {\rm r} (H)-1,0)$ of $H$, where ${\rm r} (H)$ is the rank of $H$,  is well defined.
\smallskip

Let $q^*= q^*(G_\al, \al \in I)$ denote the minimum of orders $>2$ of finite subgroups of groups $G_\al$, $\al
\in I$, and  let $q^* := \infty$ if there are no such subgroups.  It is clear that either $q^*$ is an odd prime or $q^* \in \{ 4,  \infty \}$.
If  $q^* = \infty$, define
$ \frac{q^*}{q^*-2} := 1$.
Dicks and the author \cite{DIv} proved that if
$H_1$ and $H_2$ are finitely generated factor-free subgroups of  $\FF$,  then
\begin{equation}\label{di}
\bar \rr(H_1\cap H_2)  \le  2\tfrac{q^*}{q^*-2}
 \bar \rr(H_1) \bar \rr(H_2)  .
\end{equation}
Dicks and the author \cite{DIv} conjectured that if groups $G_\al$, $\al \in I$,  contain no involutions, then the coefficient 2 could be left out  and
 \begin{equation}\label{conj}
 \bar \rr(H_1\cap H_2)  \le  \tfrac{q^*}{q^*-2}
 \bar \rr(H_1) \bar \rr(H_2)  .
 \end{equation}
This conjecture  can be regarded as a far reaching generalization of the Hanna Neumann conjecture  \cite{N1} on rank of the intersection of subgroups in free groups.  Recall that the  Hanna Neumann conjecture  \cite{N1}   claims that  if  $H_1$, $H_2$
are finitely generated subgroups of a free group, then $\bar { \rm{r} } (H_1 \cap H_2) \le \bar{ {\rm r}} (H_1) \bar { \rm r}  (H_2)$.
For more discussion, partial results and proofs of this conjecture the reader is referred to  \cite{D}, \cite{D2}, \cite{Fr}, \cite{Iv12}, \cite{Min},  \cite{N2}, \cite{St}, \cite{T}.
\smallskip

The conjecture \eqref{conj} is established by Dicks and the author \cite{DIv2} in the case when $\FF$  is the free product of two groups of order 3 in which case $q^* = 3$ and \eqref{conj} turns into
$$
\bar \rr(H_1\cap H_2)  \le  \tfrac{q^*}{q^*-2}  \bar \rr(H_1) \bar \rr(H_2) = 3  \bar \rr(H_1) \bar \rr(H_2) .
$$
Another special case in which the conjecture \eqref{conj} is known to be true is the case when $\FF$ is the free product of infinite cyclic groups, i.e., $\FF$ is a free group, as follows from Friedman's \cite{Fr}  and Mineyev's \cite{Min} proofs  of the Hanna Neumann conjecture, see also Dicks's proof \cite{D2}.  In this case  $q^* =\infty$ and the inequality \eqref{conj} turns into
\begin{equation}\label{HNC}
\bar \rr(H_1\cap H_2)  \le   \bar \rr(H_1) \bar \rr(H_2) .
 \end{equation}
More generally, the inequality \eqref{HNC}  also holds in the case when $\FF$ is the free product  of right  orderable groups as follows from results of  Antol\'in, Martino, and Schwabrow \cite{ABC}, see also \cite{Iv12}.
We mention that it follows from results of \cite{DIv}  that the conjectured inequality  \eqref{conj}  is sharp and may not be improved.
\smallskip

In an attempt to improve on the bound  \eqref{di} in a special  case, Dicks and the author  \cite{DIv2}  showed  that
   \begin{equation}\label{imrn}
  \bar \rr(H_1\cap H_2)  \le   \left(2 - \tfrac{(4+2\sqrt{3})p}{(2p-3+\sqrt{3})^2} \right) \cdot \tfrac{p}{p-2}  \bar \rr(H_1) \bar \rr(H_2)
 \end{equation}
for finitely generated factor-free subgroups $H_1$, $H_2$ of the free product $C_p * C_p$ of two cyclic  groups of prime order $p >2$.
\smallskip

Note that for $p=3$ the inequality  \eqref{imrn}  yields the conjectured inequality  \eqref{conj}.  However,  for prime $p \ge 5$, the problem whether the inequality \eqref{conj} holds for  the free product of two cyclic groups of order $p$ remains open and seems to be the most basic and appealing case of the conjecture \eqref{conj}  for groups with torsion. In this connection, we remark that the ideas of articles \cite{ABC},  \cite{D2},  \cite{Fr},  \cite{Min} do not look to be applicable to the case of free products with torsion and shed no light on the conjecture \eqref{conj} for free products of groups with torsion, especially, for free products of finite groups.
\medskip

In this article, however, we will not attempt to prove or improve on any upper bounds. Instead, we will look at generalized intersections of finitely generated factor-free subgroups in
free products of groups from a disparate standpoint and prove results of quite different flavor by utilizing the method of linear programming.
\smallskip

First we recall a stronger version of the conjecture \eqref{conj} that generalizes the  strengthened Hanna Neumann
conjecture which was  put forward by
Walter Neumann  \cite{N2}  for subgroups of free groups.  Let  $H_1$ and $H_2$ be finitely generated factor-free subgroups of an arbitrary free product $\FF = \prod_{\alpha \in I}^* G_\alpha$ of groups $G_\alpha$, $\alpha \in I$, let  the number $\frac{q^*}{q^*-2}$ be defined for $\FF$ as above,   and let $S(H_1, H_2)$ denote a set of  representatives of those double cosets $H_1 t H_2$ of $\FF$, $t \in \FF$,   that have the property $H_1 \cap t H_2 t^{-1} \ne  \{ 1 \}$.  Then the  strengthened
version of  the conjecture \eqref{conj} claims that
\begin{equation}\label{conjs}
\bar \rr(H_1, H_2) :=   \sum_{s \in S(H_1, H_2)} \bar \rr(H_1\cap s H_2 s^{-1})
\le \tfrac{q^*}{q^*-2}   \bar \rr(H_1) \bar \rr(H_2)  ,
\end{equation}
where $\bar \rr(H_1, H_2) $ is the reduced rank of the generalized intersection of $H_1$ and $H_2$ consisting of subgroups $H_1\cap s H_2 s^{-1} $, $s \in S(H_1, H_2)$.
\smallskip

Let $\KK$  denote the set of all finitely generated noncyclic factor-free subgroups of the free product  $\FF$.
Pick a subgroup $H_1 \in \KK$.  We will say that a real number $\sigma(H_1) >0 $ is the {\em Walter Neumann coefficient} for   $H_1$, or,  briefly, the WN-coefficient for $H_1$,   if, for every  $H_2 \in \KK$, we have
\begin{equation}\label{dfy0}
\brr (H_1 , H_2) \le \sigma(H_1)  \brr(H_1) \brr(H_2)
\end{equation}
and  $\sigma(H_1)$ is minimal  with this property. Clearly,
$$
\sigma(H_1)  = \sup_{H_2} \bigg\{ \frac{\brr (H_1,  H_2)}{\brr (H_1) \brr (H_2) }  \bigg\}
$$
over all  subgroups $H_2 \in \KK$.
\smallskip

For every integer $d \ge 3$, we also define the number
\begin{equation}\label{dfy0d}
\sigma_d(H_1)  := \sup_{H_2} \bigg\{ \frac{\brr (H_1 ,  H_2)}{\brr (H_1) \brr (H_2) } \bigg\}
\end{equation}
over all  subgroups $H_2 \in \KKd$, where $\KKd$ is a subset of $ \KK $ consisting of those subgroups
whose  irreducible core graphs have all of its vertices of degree at most $d$, see Section~2 for definitions.
This number  $\sigma_d(H_1)$ is called the  WN${}_d$-{\em coefficient} for    $H_1$.
 Since  $\KKd \subseteq \mathcal K_{\textsf{ff}}(\FF, d+1)$, it follows from the definitions that  $\sigma_d(H_1) \le \sigma_{d+1}(H_1) \le \sigma(H_1)$ and $\sup_{d}\{ \sigma_d(H_1) \} = \sigma(H_1)$.
\smallskip

For example, it follows from results of   \cite{DIv}, \cite{DIv2} mentioned above  that if  $\FF = C_p * C_p$ is the  free product of two cyclic groups of  prime order $p >2$ and  $H_1 \in \KK$, then
$$
\tfrac{p}{p-2} \le   \sigma_d(H_1)\le \sigma(H_1)  \le   \left(2 - \tfrac{(4+2\sqrt{3})p}{(2p-3+\sqrt{3})^2} \right) \cdot \tfrac{p}{p-2} .
$$

The main technical result of this article is the following.

\begin{thm}\label{th1} Suppose that $\FF =G_1 * G_2$ is the free product of two nontrivial groups $G_1,  G_2$
and  $H_1$ is a  finitely generated factor-free noncyclic subgroup of $\FF$. Then the following  are true.

\begin{enumerate}

\item[(a)]  For every integer $d \ge 3$,  there exists a linear programming problem (LP-problem)
\begin{equation}\label{lpa}
\PP(H_1, d) = \max\{ c(d)x(d) \mid A(d)x(d) \le b(d)  \}
\end{equation}
with integer coefficients whose solution is equal to $-\sigma_d(H_1) \brr (H_1)$.
\smallskip

\item[(b)]  There is a finitely generated factor-free subgroup $H_2^* $ of $\FF$,
$H_2^*= H_2^*(H_1)$,   such that  $H_2^* $   corresponds to  a vertex solution of the dual problem
$$
\PP^*(H_1, d) = \min \{ b(d)^{\top}  y(d)  \mid A(d)^{\top}y(d) = c(d)^{\top} , \, y(d) \ge 0  \}
$$
of the primal LP-problem  \eqref{lpa} of part (a) and
$$
\bar \rr(H_1, H_2^*)  =  \sigma_d(H_1)  \bar \rr(H_1) \bar \rr( H_2^*) .
$$
In particular,  the WN${}_d$-coefficient $\sigma_d(H_1)$ of $H_1$ is rational.

Furthermore, if $\Psi(H_1)$ and  $\Psi(H_2^*)$ denote irreducible core graphs representing subgroups $H_1$ and $H_2^*$, resp.,
and $| E \Psi |$ is the number of oriented edges in the graph $\Psi$, then
$$
| E  \Psi(H_2^*) | < 2^{  2^{| E  \Psi(H_1) |/4 + \log_2 \log_2 (4d)  } } .
$$

\item[(c)]  There exists a linear semi-infinite programming problem (LSIP-problem)
$\PP(H_1) = \sup \{ cx \mid Ax \le b  \}$ with finitely many variables in $x$ and with countably  many constraints in the system $Ax \le b$ whose dual problem
$$
\PP^*(H_1)  = \inf \{ b^{\top} y \mid A^{\top} y = c^{\top} , \, y \ge 0  \}
$$
has a solution equal to  $-\sigma(H_1) \brr (H_1)$.
\smallskip

\item[(d)]  Let the word problem for both groups $G_1, G_2$ be solvable
and let an  irreducible core graph $\Psi(H_1)$  of $H_1$ be given. Then the
LP-problem \eqref{lpa}  of part (a)   can be algorithmically written down and the
WN${}_d$-coefficient $\sigma_d(H_1)$ for $H_1$  can be computed.
In addition,  an irreducible core graph $\Psi(H_2^*)$ of the subgroup $H_2^*$ of part (b) can be algorithmically constructed.
\smallskip

\item[(e)]  Let both groups $G_1$ and $G_2$ be finite,  let $d_{m} := \max( |G_1|, |G_2|) \ge 3$, and
 let an irreducible  core graph $\Psi(H_1)$  of $H_1$ be given.
 Then the LP-problem \eqref{lpa} of part (a)  for $d = d_{m}$ coincides with the
LSIP-problem $\PP(H_1)$ of part (c) and the WN-coefficient $\sigma(H_1)$ for $H_1$ is
rational and computable.
\end{enumerate}
 \end{thm}

It is worthwhile to mention that  the correspondence  between subgroups $H_2 \in \KKd$ and vectors  of the feasible polyhedron  $\{ y(d)  \mid A(d)^{\top}y(d) = c(d)^{\top} , \, y(d) \ge 0  \}$  of the dual problem $\PP^*(H_1, d)$,   mentioned in part (b) of Theorem~\ref{th1},  plays an important role in proofs and is reminiscent of the correspondence between (almost) normal surfaces in 3-dimensional manifolds and their (resp. almost)  normal vectors in the Haken theory of normal surfaces and its generalizations, see \cite{Haken}, \cite{HLP}, \cite{Hemion}, \cite{Iv08s}, \cite{JT}.  In particular, the idea of a vertex solution works equally well both in the context of almost normal surfaces \cite{Iv08s}, see also \cite{HLP}, \cite{JT}, and in the context of factor-free subgroups, providing in either  situation  both the connectedness of the underlying object associated with a vertex solution and an upper bound on the size of the underlying object.
\medskip

Relying on the linear programming approach of Theorem~\ref{th1}, in the following Theorem~\ref{th2}, we look at the computational complexity  of the problem to compute the WN-coefficient $\sigma(H_1)$ for a factor-free subgroup $H_1$  of the free product of two finite groups and at other relevant questions.

\begin{thm}\label{th2}  Suppose that $\FF =G_1 * G_2$ is the free product of two nontrivial finite groups $G_1,  G_2$ and $H_1$ is a subgroup of $\FF$ given by a finite generating set $\Ss$ of words over the alphabet
$ G_1 \cup G_2$. Then the following are true.
\begin{enumerate}

\item[(a)]  In deterministic polynomial time {}in the size of $\Ss$, one can
detect whether $H_1$ is factor-free and noncyclic and, if so, one can construct an  irreducible graph $\Psi_o(H_1)$  of $H_1$.
\smallskip

\item[(b)]  If  $H_1$ is factor-free and noncyclic, then, in deterministic exponential time {}in the size of $\Ss$, one can write down and solve an LP-problem $\PP = \max\{ cx \mid Ax \le b  \}$ whose solution is equal to $-\sigma(H_1) \brr (H_1)$. In particular, the $WN$-coefficient $\sigma(H_1)$ of $H_1$ is computable in exponential time {}in the size of $\Ss$.
\smallskip

\item[(c)]  If  $H_1$ is factor-free and noncyclic, then  there exists a finitely generated  factor-free subgroup
$H_2^* = H_2^*(H_1)$ of $\FF$  such that $$\bar \rr(H_1, H_2^*) =  \sigma(H_1)  \bar \rr(H_1) \bar \rr( H_2^*) $$ and the size of an irreducible core graph $\Psi(H_2^*)$ of $H_2^*$  is at most doubly exponential {}in the size of $\Psi(H_1)$. Specifically,
$$
| E \Psi(H_2^*) |  <  2^{ 2^{ | E \Psi(H_1) |/4 + \log_2 \log_2 (4d_m)  } } ,
$$
where   $\Psi(H_1)$ is an  irreducible core graph of $H_1$,   $| E\Psi |$ denotes the number of oriented edges of the graph $\Psi$, and  $d_m := \max( |G_1|, |G_2|)$.

In addition, an irreducible core graph $\Psi(H_2^*)$ of $H_2^*$  can be constructed
in deterministic exponential time {}in  the size of  $\Ss$ or  $\Psi(H_1)$.
\end{enumerate}
\end{thm}

It is of interest to observe that our construction of the  graph $\Psi(H_2^*)$ is somewhat succinct (cf. the definition of succinct representations of  graphs in \cite{PCC}) in the sense that,
despite the fact that the size of $\Psi(H_2^*)$ could be doubly exponential, we are able to give a description of $\Psi(H_2^*)$  in exponential time.  In particular, vertices of  $\Psi(H_2^*)$    are represented by exponentially long bit strings and edges of $\Psi(H_2^*)$  are  drawn in blocks. As a result,
we can find out in exponential time  whether two given vertices of  $\Psi(H_2^*)$  are connected by an edge.
\smallskip

The situation with free products of more than two factors is more difficult to study and we will make additional efforts to  obtain the following results.

\begin{thm}\label{th3}
Suppose  that $\FF = \prod_{\alpha \in I}^* G_\alpha$ is
the  free product  of  nontrivial groups   $G_\al$, $\al \in I$, and $H_1$ is a
finitely generated factor-free noncyclic subgroup of $\FF$.  Then there are two  disjoint finite subsets
$I_1, I_2$ of the index set $I$ such that if \
$\wht G_1 := \prod_{\alpha \in I_1}^* G_\alpha$,  \   $\wht G_2 := \prod_{\alpha \in I_2}^* G_\alpha$,
 and      $\wht \FF := \wht G_1 * \wht G_2$,
then there exists a finitely generated factor-free subgroup $\wht H_1 $ of $\wht \FF$  with the following properties.

\begin{enumerate}
\item[(a)]  $\brr (\wht H_1  ) = \brr (H_1  ) $,
$\sigma_d(\wht H_1) \ge \sigma_d(H_1)$ for every $d \ge 3$,  and
$\sigma(\wht H_1) \ge \sigma(H_1)$. In particular, if the conjecture
\eqref{conjs} fails for $H_1$ then the conjecture \eqref{conjs} also fails for $\wht H_1$.
\smallskip

\item[(b)]  If the word problem for every group $G_\al$, where $\al \in I_1 \cup I_2$,
is solvable and  a finite irreducible graph  of $H_1$ is given, then the
LP-problem $\PP(\wht H_1, d)$  for  $\wht H_1$   of part (a) of Theorem~\ref{th1}  can be algorithmically written down and the WN${}_d$-coefficient $\sigma_d(\wht H_1)$ for $\wht H_1$  can be computed.
\smallskip

\item[(c)]  Let every group $G_\al$, where $\al \in I_1 \cup I_2$,
be finite, let $H_1$ be  given either by
a finite irreducible graph or by a finite generating set, and let
$$
{d_{M}} := \max \Big\{ |I_1 \cup I_2| , \max\{ |G_\al |  \mid   \al \in I_1 \cup I_2 \} \Big\} .
$$
Then $\sigma_{d_{M}}(\wht H_1) \ge \sigma(H_1)$
and there is an algorithm that decides whether the conjecture  \eqref{conjs} holds for $H_1$.
\end{enumerate}
\end{thm}

We remark that the proofs of Theorems~\ref{th2}--\ref{th3} provide a practical deterministic algorithm
(with exponential running time, though) to compute the WN-coefficient $\sigma(H_1) $ for a finitely generated factor-free subgroup $H_1$ of the free product of two finite groups and to determine whether a certain finitely generated factor-free subgroup of a free product of  finite groups satisfies the  conjecture \eqref{conjs}. It would be of interest to implement this algorithm and experiment with it.
\smallskip

The article is structured as follows.
In Section~2,  we define basic notions and recall geometric ideas that are used to study finitely generated  factor-free subgroups and their intersections in a free product $\FF$. In particular, we define a finite labeled graph $\Psi(H)$ associated with such a subgroup $H$ of $\FF$.  In Section~3, we consider the free product  $\FF = G_1 * G_2$  of two nontrivial groups $G_1, G_2$ and introduce certain linear inequalities associated with the groups $ G_1, G_2 $ and with the graph $\Psi(H_1)$ of $H_1$, where $H_1$ is a finitely generated  factor-free noncyclic subgroup of $\FF$.
Informally, these inequalities are used for construction of cores of potential fiber product graphs $\Psi(H_1) \times \Psi(H_2)$, where $H_2$ is another finitely generated  factor-free subgroup of $\FF$, and for subsequent translation to linear programming problems. More formally,  these inequalities enable us to define an  LP-problem
$
\max\{ c(d)x(d) \mid A(d) x(d) \le b(d) \}
$,
corresponding to $\Psi(H_1)$ and to an integer $d \ge 3$, and to define an LSIP-problem $\sup \{ cx \mid A x \le b \}$, corresponding to $\Psi(H_1)$. We also consider and make use of the dual problems of the primal problems $$
\max\{ c(d)x(d) \mid A(d) x(d) \le b(d) \} , \quad  \sup \{ cx \mid A x \le b \} .
 $$
Basic results and terminology of linear programming are discussed in Section~4.
These LP-,  LSIP-problems and their dual problems are investigated in Sections~3--4. In Section~5, we look at the  case of  free products of more than two groups and prove a few more technical lemmas. Proofs of  Theorems~\ref{th1}--\ref{th3}  are given in Section~6.

\section{Preliminaries}

Let $G_\al$, $\al \in I$,  be  nontrivial
groups,  let $\FF = \prod_{\alpha \in I}^* G_\alpha$ be their  free product, and let $H$ be  a finitely generated factor-free subgroup of $\FF$, $H \ne \{ 1\}$.
Consider the alphabet $\A = \bigcup_{\al \in I} G_\al$, where $G_{\al} \cap G_{\al'} =\{ 1\}$ if $\al \ne \al'$.
\smallskip

Analogously to the graph-theoretic approach of articles \cite{DIv}, \cite{DIv2}, \cite{Iv99}, \cite{Iv01}, \cite{Iv08}, \cite{Iv10}, \cite{Iv12},
we first define a labeled $\A$-graph $\Psi(H)$ which  geometrically represents $H$ in a manner similar  to the way Stallings  graphs represent  subgroups of a free group, see \cite{St}.
\smallskip

If $\Gamma$ is a graph, $V \Gamma $ denotes the vertex set of  $\Gamma $ and  $E \Gamma $ denotes the  set of oriented edges of $\Gamma$. For $e \in E \Gamma $ let  $e_-$, $e_+$ denote the initial, terminal, resp., vertices of  $e$ and let $e^{-1}$ be the edge with the opposite orientation, where $e^{-1} \ne  e$ for every $e \in E\Gamma$,  $(e^{-1})_- = e_+$, $(e^{-1})_+ = e_-$.
\smallskip

A  {\em path} $p = e_1 \dots e_k$ in $\Gamma$ is a sequence of edges $e_1, \dots, e_k$ such that $(e_{i})_+ = (e_{i+1})_-$,  $i=1, \dots, k-1$.
Define $p_- := (e_1)_-$, $p_+ := (e_k)_+$, and $|p| := k$, where $|p| $ is the {\em length} of $p$.
We  allow the possibility  that $p = \{ p_- \} = \{ p_+ \}$ and $|p| = 0$. A path $p$ is {\em closed} if $p_- =p_+$. A path $p  $ is called {\em reduced} if  $p$ contains no subpaths of the form $e e^{-1}$, $e \in E\Gamma$.
A closed path $p = e_1 \dots e_k$  is {\em cyclically reduced} if $|p| >0$ and both $p$ and the cyclic permutation
$e_2 \dots e_k e_1$ of $p$ are reduced paths. The  {\em core} of a graph $\Gamma$,  denoted  $\mbox{core}(\Gamma)$,
is the minimal subgraph of $\Gamma$ that contains every  edge $e$ which
can be included into a cyclically reduced path in $\Gamma$.
\smallskip

Let $\Psi$ be a  graph whose vertex set  $V \Psi$  consists of two disjoint nonempty parts  $V_P \Psi,  V_S \Psi$, so $V \Psi = V_P \Psi \cup V_S \Psi$. Vertices in  $ V_P \Psi$ are called {\em primary} and vertices in  $ V_S \Psi$ are called {\em secondary}.
Every edge $e \in E \Psi$ connects  primary and secondary vertices, hence, $\Psi$ is a bipartite graph.
\smallskip

$\Psi$ is called a  {\em labeled $\A$-graph}, or briefly {\em $\A$-graph},  if  $\Psi$ is equipped with functions
 \begin{equation*}
\ph : E\Psi  \to \A , \qquad \btt :  V_S \Psi \to I
\end{equation*}
such that, for every edge $e \in E\Psi$, it is true that
$$
\ph(e) \in \A = \bigcup_{\al \in I} G_\al , \quad   \ph(e^{-1}) =  \ph(e)^{-1} ,
$$
and,  if $e_+ \in V_S \Psi$, then $\ph(e) \in G_\al$, where $\al = \btt(e_+)$.

If  $e_+ \in V_S \Psi$, define
$$
\btt(e) :=\btt(e_+),  \qquad   \btt(e^{-1}) :=\btt(e_+)
$$
and call $\btt(e_+)$, $\btt(e)$    the {\em type} of  a vertex $e_+ \in V_S \Psi$ and of an edge $e \in E\Psi$. Thus, for every   $e \in E\Psi$, we have defined  an element
$\ph(e) \in \A$,  called the {\em label} of $e$, and an element $\btt(e)   \in I$, called the {\em type} of $e$.
\smallskip

The reader familiar with  van Kampen diagrams over a free product of groups,
as defined in \cite{LS}, will recognize that our labeling function $\ph : E\Psi \to \A $  is defined in the way analogous to labeling functions on van Kampen diagrams over free products of groups. Recall that van Kampen diagrams are planar 2-complexes whereas  graphs are 1-complexes, however, apart from this, the ideas of cancellations and edge folding  work equally well for both diagrams and  graphs.
\smallskip

An $\A$-graph $\Psi$ is called {\em irreducible} if  the following properties (P1)--(P3) hold true.
\begin{enumerate}
\item[(P1)] If $e, f \in E \Psi$,  $e_- = f_- \in  V_P \Psi$, and $e_+ \ne f_+$, then   $\btt(e) \ne \btt(f)$.
\item[(P2)] If $e, f \in E \Psi$,  $e \ne f$, and $e_+= f_+ \in V_S \Psi$, then  $\ph(e) \ne  \ph(f)$ in $G_{\btt(e)}$.
\item[(P3)] $\Psi$ has no multiple edges, $\mbox{deg}_\Psi v >0$ for every $v \in V\Psi$,
  and there is at most one vertex of degree 1 in $\Psi$ which, if exists, is primary.
    \end{enumerate}

Suppose $\Psi$ is a connected  finite
irreducible $\A$-graph and a primary vertex $o \in V_P\Psi$ is distinguished so that $\deg_\Psi o =1$
if $\Psi$ happens to have a vertex of degree 1. Then  $o$ is called the {\em base} vertex of $\Psi = \Psi_o$.
\medskip

As usual, elements of the free product  $\FF = \prod_{\alpha \in I}^* G_\alpha$
are regarded  as words over the alphabet $\A = \bigcup_{\al \in I} G_\al$, where
$G_{\al} \cap G_{\al'} =\{ 1\}$ if $\al \ne \al'$.
A {\em syllable} of a word $W$ over $\A$ is a maximal nonempty subword of $W$
all of whose letters belong to the same factor $G_\al$. The {\em syllable length} $\| W \|$ of $W$
is the number of syllables
of $W$, while the {\em length} $|W|$ of $W$ is the number of all
letters in $W$. For example, if $a_1, a_2 \in G_\al$, then $| a_1 1 a_2 | = 3$, $\| a_1 1 a_2 \| = 1$, and
$| 1| = \| 1 \| = 1$.
\smallskip

A nonempty word $W$ over   $\A $ is called {\em reduced} if every
syllable of $W$ consists of a single letter.
Clearly, $|W| = \|
W \|$ if $W$ is reduced. Note that an arbitrary nontrivial element of the
free product $\FF$  can  be uniquely
written as a reduced word. A word $W$  is called {\em cyclically reduced} if $W^2$ is reduced.
We write $U \overset 0 = W$ if words $U$, $W$
are equal as elements of $\FF$. The
literal (or letter-by-letter) equality of words $U$, $W$ is denoted $U \equiv W$.
\smallskip

If $p = e_1 \dots e_k$ is a path in an  $\A$-graph $\Psi$ and $e_1, \dots, e_k$ are edges of $\Psi$, then the {\em label} $\ph(p)$ of $p$ is the word $\ph(p) := \ph(e_1) \dots \ph(e_k)$.
\smallskip

The significance of irreducible $\A$-graphs for geometric interpretation of  factor-free subgroups $H$ of $\FF$ is given in the following lemma.

\begin{lem}\label{Lm1} Suppose $H$ is a finitely generated factor-free subgroup
of the free product $\FF =  \prod_{\alpha \in I}^* G_\alpha$, $H \ne \{ 1 \}$. Then there exists a finite connected
irreducible $\A$-graph $\Psi = \Psi_o(H)$, with a  base vertex $o$,  such that a reduced word $W$ over the alphabet $\A$ belongs to $H$ if and only if there is a reduced path $p$ in $\Psi_o(H)$  such that $p_- =p_+
=o$,  $\ph(p) \overset 0  = W$ in   $\FF$, and $| p | = 2|W|$.
\smallskip

In addition, assume that all factors $G_\alpha$, $\alpha \in I$, are finite and  $V_1, \dots, V_k$  are  words over $\A$. Then there is a deterministic algorithm which, in  polynomial time depending on the sum $|V_1|+ \dots+ |V_k|$, decides whether the subgroup  $H_V = \langle V_1, \dots, V_k \rangle$, generated by $V_1, \dots, V_k$, is factor-free and, if so, constructs an irreducible $\A$-graph $\Psi_o(H_V)$ for $H_V$.
\end{lem}

\begin{proof} The proof is based on Stallings's folding techniques and is somewhat analogous to the proof of van Kampen lemma for diagrams over free products of groups, see \cite{LS} (in fact, it is simpler because foldings need not preserve the property of being planar for  diagrams).
\smallskip

Let  $H_V = \langle V_1, \dots, V_k \rangle$ be a subgroup of $\FF$,
generated by some words $V_1, \dots, V_k$ over  $\A$. Without loss of generality we can assume that  $V_1, \dots, V_k$  are  reduced  words. Consider a graph $\wtl \Psi$ which consists of $k$ closed paths $p_1, \dots, p_k$ such that they have a single common vertex $o = (p_i)_-$, and $|p_i| = 2|V_i|$, $i=1, \dots, k$.
We distinguish  $o$ as the base vertex of $\wtl \Psi$ and call $o$  primary, the vertices adjacent to $o$ are called secondary vertices and so on.  Denote $V \equiv a_{i,1} \dots a_{i,\ell_i}$, where $a_{i,j} \in \A$ are letters, $i =1, \dots, k$, and let
$p_i =  e_{i,1} f_{i,1} \dots e_{i,\ell_i}f_{i,\ell_i}$, where $e_{i,j}, f_{i,j}$ are edges of the path $p_i$. The labeling functions $\ph, \btt$ on the path $p_i$ are defined so that if  $a_{i,j} \in G_{\al(i,j)}$, then
\begin{align*}
\btt(e_{i,j}) & := \al(i,j),  \hskip-2.3cm  &  \btt(f_{i,j}) & := \al(i,j) ,  \\
 \ph(e_{i,j}) & := a_{i,j}b_{i,j}^{-1},  \hskip-2.3cm  &  \ph(f_{i,j}) & := b_{i,j} ,
\end{align*}
where $b_{i,j}$ is an element of the group  $G_{\al(i,j)}$.
\smallskip

Clearly,  $\ph(p_i) \overset 0 = V_i$ in $\FF$ for all $i=1, \dots, k$.
\smallskip

It is also clear that $\wtl \Psi=\wtl \Psi_o$ is a finite connected
$\A$-graph with the base vertex $o$ that has the following  property.
\begin{itemize}
\item[(Q)] A word $W  \in \FF$ belongs to $H$  if and only if there is a path $p$ in $\wtl \Psi_o$ such that $p_- =p_+
=o$ and  $\ph(p) \overset 0 = W$.
  \end{itemize}

However, $\wtl \Psi_o$ need not be irreducible  and we will do
foldings of edges in $\wtl \Psi_o$ which  preserve  property (Q) and which are aimed  to achieve properties (P1)--(P2).
\smallskip

Assume that property (P1) fails for edges $e, f$ with  $e_- = f_- \in V_P\wtl \Psi_o$ so that  $e_+ \ne f_+$ and  $\btt(e) =\btt(f)$.
Let us redefine the labels  of all edges  $e'$ with $e'_+ = e_+$ so that $\ph(e') \ph(e)^{-1}$ does not change
and $\ph(e) = \ph(f)$ in $G_{\btt(e)}$. This can be done by multiplication of $\ph$-labels on the right by $\ph(e)^{-1} \ph(f)$.
Since $\ph(e) = \ph(f)$ and $\btt(e) = \btt(f)$, we may now identify the edges $e$, $f$ and vertices $e_+$, $f_+$.  Observe that this folding preserves  property (Q)  ((P2) might fail) and decreases the total edge number $|E \wtl \Psi_o|$. This operation changes the labels of edges and can be done in time polynomial in $|V_1|+ \dots+ |V_k|$ if
all factors $G_\al$, $\al \in I$, are finite. Note that if $G_\al$ were not finite, then there would be a problem with increasing space needed to store   $\ph$-labels of edges and subsequent computations with larger labels.
\smallskip

If property (P2) fails for edges $e, f$ and  $\ph(e) = \ph(f)$ in $G_{\btt(e)}$, then we  identify the edges $e, f$.  Note property  (Q) still holds ((P1) might fail) and the number $|E \wtl \Psi_o|$ decreases.
\smallskip

Suppose property (P3) fails and there are two distinct edges  $e, f$ in $\wtl \Psi_o$ such that $e_- = f_-$ and $e_+ = f_+ \in V_S \wtl \Psi_o$.
If $\ph(e) \ne \ph(f)$ in $G_{\btt(e)}$, then a conjugate of  $\ph(e) \ph(f)^{-1} \in G_{\btt(e)}$ is in $H_V$, hence we conclude that
 $H_V$ is not factor-free.  So we may assume that $\ph(e) = \ph(f)$ in $G_{\btt(e)}$.
Then we identify the edges $e, f$, thus preserving  property  (Q) and decreasing the number $|E \wtl \Psi_o|$.  If property (P3) fails so that there is a vertex $v$ of degree 1, different from $o$, then we delete $v$ along with the incident edge. Clearly, property  (Q)
still holds and the number $|E \wtl \Psi_o|$ decreases.
\smallskip

Thus,  by induction on $|E \wtl \Psi_o|$ in polynomially many
(relative to $\sum_{i=1}^{k} |V_i|$) steps as described above, we either establish that the subgroup $H_V$ is not factor-free or construct an
irreducible  $\A$-graph $\Psi_o$ with property (Q).
\smallskip

It follows from the definitions and from property (Q) of the graph $\Psi_o$ that $H_V$ is factor-free (see also Lemma~\ref{Lm2}). Other stated properties of $\Psi_o$ are straightforward.
\smallskip

Finally, we observe that if  all factors $G_\alpha$, $\alpha \in I$, are finite, then the  space required to store the $\ph$-label of an edge of intermediate graphs is constant and multiplication (or inversion) of  $\ph$-labels would require  time bounded by a constant. Therefore, the above procedure implies the existence of  a polynomial algorithm for finding out whether a subgroup
 $H_V= \langle V_1, \dots, V_k \rangle $ of $\FF$ is factor-free and for construction of a  finite  irreducible $\A$-graph $\Psi_o$  for $H_V$.
\end{proof}

The following lemma further elaborates on the correspondence between finitely generated
factor-free subgroups of the free product  $\FF$ and finite
irreducible $\A$-graphs.

\begin{lem}\label{Lm2}  Let  $\Psi_o$ be a finite connected
irreducible  $\A$-graph with the base vertex  $o$ and let
$H= H(\Psi_o)$ be a subgroup of the free product
$\FF$ that consists of all words $\varphi(p)$, where $p$ is
a  path in $\Psi_o$ such that $p_- = p_+ = o$. Then $H$ is a
factor-free subgroup of  $\FF$   and
$\bar \rr(H)=- \chi(\Psi_o)$, where
$$
\chi(\Psi_o) =  |V \Psi_o | - \tfrac 12|E \Psi_o |
$$
is  the Euler characteristic of~$\Psi_o$.
\end{lem}

\begin{proof} This follows from  the facts that the fundamental group $\pi_1(\Psi_o, o)$ of
$\Psi_o$ at $o$ is free of rank $- \chi(\Psi_o)+1$ and that the homomorphism $\pi_1(\Psi_o, o) \to \FF$, given by $p \to \ph(p) $, where $p$ is a path with $p_-=p_+= o$, has the trivial kernel in view of properties (P1)--(P2).
\end{proof}

Suppose $H$ is a nontrivial finitely generated factor-free subgroup
of a free product $\FF = \prod_{\alpha \in I}^* G_\alpha$, and $\Psi_o = \Psi_o(H)$ is a  finite irreducible  $\A$-graph for $H$ as in Lemma~\ref{Lm1}. We  say that  $\Psi_o(H)$ is an {\em irreducible graph} of $H$.
\smallskip

Let $\Psi(H) := \core(\Psi_o(H))$ denote the core of an irreducible graph $\Psi_o(H)$ of $H$.
Clearly, $\Psi(H)$
has no vertices of degree $\le 1$ and $\Psi(H)$ is also an irreducible $\A$-graph.
We  say that  $\Psi(H)$ is an {\em irreducible core graph} of $H$.
\smallskip

It is easy to see that an irreducible  graph  $\Psi_o(H)$ of $H$ can be obtained back from an irreducible  core graph $\Psi(H)$  of $H$ by attaching a suitable path $p$ to  $\Psi(H)$ so that $p$ starts at a primary vertex $o$, ends in  $p_+ \in V_P\Psi(H)$, and then by doing foldings of edges as in the proof of Lemma~\ref{Lm1}, see Figure~2.1.
\begin{center}
\usetikzlibrary{arrows}
\begin{tikzpicture}[scale=.92]
\node at (-2,0) {$\Psi_o(H)$};
\node at (0,0.8) {$\Psi(H)=$};
\node at (0,0.3) {$ \mbox{core}\Psi_o(H)$};
\draw  (0,-0.7)[fill = black]circle (0.05);
\draw  (0,.5) ellipse (1.02 and 1.2);
\draw  plot[smooth, tension=.8] coordinates
{(0,-0.7)(-0.5,-1.5) (-1,-1) (-1.5,-1.5) (-2,-1.2) (-2.5,-1.5) (-2.8,-1.5)};
\draw [-latex]  (-1.269,-1.26) -- (-1.21,-1.16);
\draw  (-2.8,-1.5)[fill = black]circle (0.05);
\node at (-1.5,-1.1) {$p$};
\node at (-2.8,-1.1) {$o$};
\node at (.0,-2.2) {Figure~2.1};
\end{tikzpicture}
\end{center}

Now suppose  $H_1$, $H_2$  are  nontrivial finitely generated factor-free subgroups of $\FF$.
Consider a set $S(H_1, H_2)$ of representatives of those double cosets
$H_1 t H_2$ of $\FF$, $t \in \FF$,   that have the property $H_1 \cap t H_2 t^{-1} \ne  \{ 1 \}$.
For every $s \in S(H_1, H_2)$,  define the subgroup $K_s := H_1 \cap s H_2 s^{-1}$.
Similarly to articles \cite{Iv99}, \cite{Iv01},  \cite{Iv08}, \cite{Iv10}, \cite{Iv12} and analogously to  the case of free groups, see  \cite{D}, \cite{N2}, we now construct a finite irreducible  $\A$-graph $\Psi(H_1, H_2)$, also denoted $\core(\Psi(H_1) \times \Psi(H_2))$,  whose connected components are irreducible  core graphs $\Psi(K_s)$, $s \in S(H_1, H_2)$.
\smallskip

First we define an $\A$-graph $\Psi_o'(H_1, H_2)$. The set  of primary vertices of
$\Psi_o'(H_1, H_2)$ is
$V_P \Psi_o'(H_1, H_2)   := V_P \Psi_{o_1}(H_1)\times V_P \Psi_{o_2}(H_2)$. Let
$$
\tau_i : V_P \Psi_o'(H_1, H_2) \to V_P \Psi_{o_i}(H_i)
$$
denote the projection map, $\tau_i((v_1, v_2)) = v_i$, $i=1,2$.
\smallskip

The set of secondary vertices $V_S \Psi_o'(H_1, H_2)$ of
$\Psi_o'(H_1, H_2)$ consists of equivalence classes $[u]_\al$, where  $u \in V_P \Psi_o'(H_1, H_2)$, $\al \in I$, with respect to  the minimal  equivalence relation generated by   the following relation   $\overset \al \sim$ on the set  $V_P \Psi_o'(H_1, H_2)$.  Define $v \overset \al \sim w$ if and only if there are edges $e_i, f_i \in E \Psi_{o_i}(H_i)$ such that $$
(e_i)_- = \tau_i(v) , \  (f_i)_- = \tau_i(w) ,  \ (e_i)_+ = (f_i)_+
$$
for each  $i=1,2$, the edges  $e_i, f_i$ have type $\al$, and $\ph(e_1) \ph(f_1)^{-1} = \ph(e_2) \ph(f_2)^{-1}$ in $G_\al$. It is easy to see that the relation  $\overset \al \sim$  is symmetric and transitive on pairs and triples of distinct elements (but it could lack the reflexive property).
\smallskip

The edges in $\Psi_o'(H_1, H_2)$  are defined so that the vertices
$$
u \in  V_P \Psi_o'(H_1, H_2) \quad  \mbox{ and}  \quad
[v]_\al \in V_S \Psi_o'(H_1, H_2)
$$
are connected by an edge if and only if $u  \in [v]_\al$.
\smallskip

The type $\btt([v]_\al)$  of a vertex $[v]_\al \in V_S \Psi_o'(H_1, H_2)$  is $\al$ and if
$$
e \in E\Psi_o'(H_1, H_2) , \quad    e_- =u ,  \quad   e_+ = [v]_\al ,
$$
then $\ph(e) :=\ph(e_1)$, where $e_1  \in E\Psi_{o_1}(H_1)$
is an edge of type $\al$ with $(e_1)_- = \tau_1(u)$, when such an $e_1$ exists, and $\ph(e_1) :=g_\al $, where $g_\al \in G_\al$, $g_\al \ne 1$,    otherwise.
\smallskip

It follows from the definitions and properties (P1)--(P2) of $\Psi_{o_i}(H_i)$, $i=1,2$, that $\Psi_o'(H_1, H_2)$   is an  $\A$-graph with properties (P1)--(P2). Hence, taking the core of
$\Psi_o'(H_1, H_2)$, we obtain a  finite  irreducible  $\A$-graph which we  denote by $\Psi(H_1, H_2)$ or by $\core(\Psi(H_1) \times \Psi(H_2))$.
\smallskip

It is not difficult to see that, when taking the connected component
$$\Psi_o'(H_1, H_2, o)$$
of $\Psi_o'(H_1, H_2)$  that contains the vertex $o = (o_1, o_2)$ and inductively removing  from $\Psi_o'(H_1, H_2, o)$  the vertices of degree 1 different from $o$, we will obtain an
irreducible  $\A$-graph $\Psi_o(H_1\cap H_2)$ with the base vertex $o$ that  corresponds to the intersection $H_1\cap H_2$ as in Lemma \ref{Lm1}.
\smallskip

It follows from the definitions and property (P1) for  $\Psi(H_i)$, $i=1,2$,   that, for every edge $e \in E \Psi(H_1, H_2)$ with $e_- \in V_P \Psi(H_1, H_2)$, there are unique edges
$e_i  \in E\Psi(H_i)$ such that  $\tau_i(e_-) =(e_i)_-$, $i=1,2$. Hence, by setting  $\tau_i(e) =e_i$, $\tau_i(e_+) =(e_i)_+$, $i=1,2$,   we extend  $\tau_i$ to the graph map
\begin{gather}\label{tau1}
\tau_i :  \Psi(H_1, H_2) \to  \Psi(H_i)  , \quad i=1,2 \, .
\end{gather}
It follows from the definitions that $\tau_i$ is locally injective and $\tau_i$ preserves  syllables of the word $\ph(p)$ for every path $p$ with primary vertices $p_-, p_+$.

\begin{lem}\label{Lm3} Suppose $H_1$, $H_2$ are  finitely generated factor-free
subgroups of the free product $\FF$ and
$S(H_1, H_2) \ne  \varnothing$.  Then the connected components of the graph  $\Psi(H_1, H_2)$ are core graphs  $\Psi(H_1 \cap s H_2 s^{-1})$ of subgroups $H_1 \cap s H_2 s^{-1}$, $s \in S(H_1, H_2)$. In particular,
$$
\bar \rr(H_1, H_2) :=
\sum_{s \in S(H_1, H_2) } \bar \rr(H_1 \cap s  H_2 s^{-1}) =
-\chi(\Psi(H_1, H_2) ) .
$$
\end{lem}
\begin{proof} This is straightforward, details can be found in \cite{Iv12}.
\end{proof}

\section{The System of Linear Inequalities $\SLI[Y_1]$}

In this Section, we  let $\FF_2 = G_1 * G_2$ be the free product of two nontrivial  groups  $G_1, G_2$, let
$\A := G_1 \cup G_2$ be the alphabet,  $G_1 \cap G_2 = \{ 1\}$,  and let $Y_1$ be a finite connected
irreducible  $\A$-graph such that  $\core(Y_1) = Y_1$ and $\brr(Y_1) := -\chi(Y_1) >0$.
In particular, $Y_1$ has no vertices of degree 1 and $Y_1$ contains a vertex of degree $>2$.
\smallskip

Let $S_2(G_\al)$, where $\al =1,2$, denote the set of all finite subsets of $G_\al$ of cardinality $\ge 2$ and let $S_1(V_P Y_1)$ denote the set of all nonempty subsets of $V_P Y_1$.
For a set  $T \in S_2(G_\al)$, consider  a function
$$
\om_T : T \to S_1(V_P Y_1) .
$$

We also consider a relation  ${\sim}_{\om_T}$  on the  set of all pairs $(a, u)$, where $a \in T$ and $u \in \om_T(a)$, defined as follows.  Two pairs $(a, u)$, $(b, v)$ are related by ${\sim}_{\om_T}$,  written $(a, u) {\sim}_{\om_T} (b, v)$, if and only if the following holds. Either $(a, u) = (b, v)$ or, otherwise,
there exist edges $e, f \in E Y_1$ with the properties that $e_- = u$,  $f_- = v$, the secondary vertex $e_+ = f_+$ has type $\al$, and $\ph(e)\ph(f)^{-1} = ab^{-1}$  in $G_\al$, see an example depicted in Figure~3.1.
It is easy to see that the relation  ${\sim}_{\om_T}$  is an equivalence one.

\begin{center}
\begin{tikzpicture}[scale=.62]
\draw  (-2,2) [fill = black]circle (0.05);
\draw  (-2,3) node (v1) {} [fill = black]circle (0.05);
\draw  (-2,4) [fill = black]circle (0.05);
\draw  (-2.3,3) ellipse (1.1 and 1.8);

\draw  (0,6) [fill = black]circle (0.05);
\draw  (1,6.) node (v5) {} [fill = black]circle (0.05);
\draw  (2,6) [fill = black]circle (0.05);
\draw  (1,6.3) ellipse (1.8 and 1.1);

\draw  (4,2) [fill = black]circle (0.05);
\draw  (4,3) node (v3) {} [fill = black]circle (0.05);
\draw  (4,4) [fill = black]circle (0.05);
\draw  (4.3,3) ellipse (1.1 and 1.8);

\draw  (0,0) [fill = black]circle (0.05);
\draw  (1,0) node (v4) {} [fill = black]circle (0.05);
\draw  (2,0) [fill = black]circle (0.055);
\draw  (1,-.3) ellipse (1.8 and 1.1);

\draw  (1,3) node (v2) {} circle (0.08);
\draw [-latex]  (-.6,3) -- (-.4,3);
\draw [-latex]  (2.6,3) -- (2.4,3);
\draw [-latex]  (1,1.4) -- (1,1.6);
\draw [-latex]  (1,4.6) -- (1,4.4);
\draw  (v1) edge (v2);
\draw  (v2) edge (v3);
\draw  (v2) edge (v4);
\draw  (v5) edge (v2);
\node at (-2.5,2) {$u_2$};
\node at (-2.5,3) {$u_5$};
\node at (-2.5,4) {$u_8$};
\node at (0,6.5) {$u_1$};
\node at (1,6.5) {$u_2$};
\node at (2,6.5) {$u_3$};
\node at (4.5,4) {$u_3$};
\node at (4.5,3) {$u_4$};
\node at (4.5,2) {$u_5$};
\node at (2,-0.5) {$u_1$};
\node at (1,-0.5) {$u_6$};
\node at (0,-0.5) {$u_7$};
\node at (1.5,4.5) {$e_1$};
\node at (2.5,2.5) {$e_2$};
\node at (0.5,1.5) {$e_3$};
\node at (-0.5,3.5) {$e_4$};
\node at (-9,6) {$T  = \{ g_1, g_2, g_3, g_4 \} \subseteq G_\alpha,$ };
\node at (-9,5) {$ \Omega_T(g_1) = \{ u_1, u_2, u_3 \},  $ };
\node at (-9,4) {$\Omega_T(g_2) = \{ u_3, u_4, u_5 \}$, };
\node at (-9,3) {$\Omega_T(g_3) = \{ u_1, u_6, u_7 \} $,  };
\node at (-9,2) {$\Omega_T(g_4) = \{ u_2, u_5, u_8 \} $,  };
\node at (-8.4,1) {$\Omega_T(g_i)  \subseteq V_PY_1, \ \varphi(e_i) = g_i,\  i =1,2,3,4$,  };
\node at (-8,0) {$(g_1, u_2) \sim_{\Omega_T} (g_2, u_4) , \  \
   (g_3, u_6) \sim_{\Omega_T} (g_4, u_5) $.  };

\node at (4.2,6.4) {$ \Omega_T(g_1) $};
\node at (4.5,0.5) { $\Omega_T(g_2)$ };
\node at (3.7,-1.24) {$ \Omega_T(g_3) $};
\node at (-2.3,5.5) {$ \Omega_T(g_4) $};
\node at (-3.5,-2.2) {Figure~3.1};
\end{tikzpicture}
\end{center}

Let
$[(a, u)]_{ {\sim}_{\om_T} }$ denote the equivalence class of $(a, u)$ and let
$$
| [(a, u)]_{ {\sim}_{\om_T} } |
$$
denote the cardinality of $[(a, u)]_{ {\sim}_{\om_T} }$.  It follows from the definitions that
\begin{gather}\label{cd1}
1 \le  | [(a, u)]_{ {\sim}_{\om_T} } | \le |T|   .
\end{gather}
We will say that  the equivalence class  $[(a, u)]_{ {\sim}_{\om_T} }$ is {\em associated} with a secondary vertex $w \in V_S Y_1$  of type $\al$ if $w = e_+$, where  $e \in E Y_1$ and  $e_- = u$. It is easy to see that the definition of the secondary vertex $w$ is independent of the primary vertex  $u$ in  $[(a, u)]_{ {\sim}_{\om_T} }$.
\smallskip

A function $\om_T : T \to S_1(V_P Y_1)$, $T \in S_2(G_\al)$, is called
{\em $\al$-admissible} if
$$
|[ (a, u)  ]_{{\sim}_{\om_T}} | \ge 2
$$
for every equivalence class $[(a, u)]_{{\sim}_{\om_T}}$, where $a \in T$, $u \in \om_T(a)$. The set of all $\al$-admissible functions is denoted $\Omega(Y_1, \al)$, $\al =1,2$.
\smallskip

Let $\om_T \in  \Omega(Y_1, \al)$ be an $\al$-admissible  function, $T \in S_2(G_\al)$,  and let
\begin{equation}\label{Nal}
N_\al ( \om_T ) := \sum ( |  [  (a, u) ]_{{\sim}_{\om_T}} | -2 )
\end{equation}
denote the sum of cardinalities minus two over all equivalence classes $[(a, u) ]_{{\sim}_{\om_T}}$, where  $a \in T$ and $u \in \om_T(a)$, of the equivalence relation ${\sim}_{\om_T}$.
\smallskip

Let $r$ be the number of all equivalence classes $[(a, u)]_{{\sim}_{\om_T}}$ of the  equivalence  relation ${\sim}_{\om_T}$,  where $a \in T$ and $u \in \om_T(a)$. If $r \ge |V_P Y_1 |$, then it follows from \eqref{Nal} and definitions that
\begin{equation*}
N_\al ( \om_T ) = \sum ( |  [  (a, u) ]_{{\sim}_{\om_T}} | -2 )  \le
     |T|\cdot |V_P Y_1 | - 2r \le (|T| - 2) |V_P Y_1 | .
\end{equation*}
On the other hand, if $r \le |V_P Y_1 |$, then it follows from \eqref{cd1} and \eqref{Nal} that
\begin{equation*}
N_\al ( \om_T ) = \sum ( |  [  (a, u) ]_{{\sim}_{\om_T}} | -2 )  \le
   (|T| -2 ) r   \le (|T| - 2) |V_P Y_1 |  .
\end{equation*}
Thus, in any case, it is shown that
\begin{equation}\label{cd2}
N_\al ( \om_T )\le (|T| - 2) |V_P Y_1 |  .
\end{equation}

For every set $A \in S_1(V_P Y_1)$,
we consider a variable $x_A$. We also introduce a special variable $x_s$. Now we will define a system of linear inequalities in these variables.
\smallskip

For every $\al$-admissible function $\Omega_T$,   where $ T \in S_2(G_\al)$ and $\al =1,2$,  we denote $T = \{ b_1, \dots, b_{k} \}$ and we set $A_i := \om_T(b_i)$, $i=1, \dots, k$.
\smallskip

If $\al =1$, then the inequality,  corresponding to the  $\al$-admissible  function $\Omega_T$, is defined as follows.
\begin{equation}\label{inqa}
- x_{A_1} -\dots  - x_{A_{k}}  - (k-2)x_s \le   - N_1(\Omega_T)   .
\end{equation}

If $\al =2$, then the inequality  corresponding to the  $\al$-admissible  function $\Omega_T$ is defined as follows.
\begin{equation}\label{inqb}
 x_{A_1} +\dots  + x_{A_{k}}  - (k-2)x_s \le   - N_2(\Omega_T)   .
\end{equation}

Let
\begin{equation}\label{slio}
\SLI[Y_1]
\end{equation}
denote  the system of linear inequalities \eqref{inqa}--\eqref{inqb} over all $\al$-admissible functions $\Omega_T$, where $\Omega_T \in \Omega(Y_1, \al)$ and  $\al =1,2$. Since the set $S_2(G_\al)$ is in general infinite (unless $G_\al$ is finite) and the set $S_1(V_P Y_1)$  is finite  (because $Y_1$ is finite),  it follows that  $\SLI[Y_1]$ is an infinite system  of linear inequalities with integer coefficients over a finite set of variables $x_A, A \in S_1(V_P Y_1)$, $x_s$.
\smallskip

Let $d \ge 3$ be an integer and   let
\begin{equation}\label{slid}
\SLI_d[Y_1]
\end{equation}
denote the subsystem of the system     \eqref{slio}  whose  linear inequalities \eqref{inqa}--\eqref{inqb} are defined for all $\al$-admissible functions $\Omega_T$,  where $\Omega_T \in \Omega(Y_1, \al)$ and  $\al =1,2$, such that $|T| \le d$.
\smallskip

If $q$ is an inequality of $\SLI_d[Y_1]$
then the coefficient of $x_s$ in the  left hand side of $q$ is the integer $-k+2$, where
$$
2 \le k = k(q) = |T| \le d ,
$$
and the right hand side of $q$ is the  integer $-N_\al(\om_T)$, where
$$
0 \le N_\al(\om_T) \le (d-2) | V_P Y_1 | ,
$$
as follows from inequality \eqref{cd2}. Also, the number of subsets $A \subseteq V_P Y_1$ that  index variables $\pm  x_A$  in the  left hand side of $q$  is finite and the total number of occurrences of such variables $\pm  x_A$
in $q$ is $k = |T| \le d$.
Therefore, $\SLI_d[Y_1]$ is a finite system of linear inequalities and
$$
\SLI[Y_1] = \bigcup_{d=3}^{\infty}  \SLI_d[Y_1] .
$$

Consider the following property of a graph $Y_2$ (which need not be connected).

\begin{enumerate}
\item[(B)] \  $Y_2$ is a finite irreducible   $\A$-graph, the map $\tau_2  :  \core (Y_1  \times  Y_2) \to Y_2$ is surjective, $\core(Y_2) = Y_2$,      and  $ \brr(Y_2) := - \chi(Y_2) > 0$.
     \end{enumerate}

For example, $Y_1$ has property (B).
\smallskip

If $\Gamma$ is a finite graph, let $\deg \Gamma$ denote the maximum
of degrees of vertices of $\Gamma$. Recall that the degree of a vertex $v \in V \Gamma$ is the number of edges $e \in E \Gamma$ such that  $e_+ = v$.  For later references, we  introduce one more  property of a graph $Y_2$.

\begin{enumerate}
\item[(Bd)] \  $Y_2$ has property (B) and $\deg Y_2 \le d$, where $d \ge 3$ is an integer.
\end{enumerate}

Suppose $Y_2$ is a graph with property (B).
For a secondary vertex $u \in V_S Y_2$ of type $\al$,  we consider
all edges $e_1, \dots, e_\ell$, where $\deg u = \ell$, such that
$$
u = (e_1)_+=  \dots = (e_\ell)_+
$$
and denote   $v_j := (e_j)_-$, $j =1, \dots, \ell$.
Define
$$
T_u := \{ \ph(e_1), \dots , \ph(e_\ell) \} .
$$
Clearly,  $T_u \subseteq S_2(G_\al)$.
For every $j = 1, \dots, \ell$,  let $\tau_2^{-1}(v_j)$ denote the full preimage of the vertex $v_j$ in
$\core( Y_1 \times Y_2)$.
Define  the sets
\begin{equation}\label{defA}
A_j(u) := \tau_1  \tau_2^{-1}(v_j) \subseteq V_P Y_1
\end{equation}
for $j=1, \dots, \ell$ and consider the function
\begin{equation}\label{STu}
\om_{T_u} : T_u \to S_1(V_P Y_1)
\end{equation}
so that $\om_{T_u}(\ph(e_j)) := A_j(u)$.
\smallskip

It is easy to check that  $\Omega_{T_u}$ is $\al$-admissible.
Since every  $\al$-admissible function $\Omega \in \Omega(Y_1, \al)$
gives rise to  an inequality \eqref{inqa}  if $\al =1$ or  to an inequality \eqref{inqb}  if $\al=2 $ and every
secondary vertex $u \in V_S Y_2$  of type $\al$  defines, as indicated above, an $\al$-admissible  function  $\Omega_{T_u}$, it follows that every  $u \in V_S Y_2$ is mapped to a certain inequality
of the system  $\SLI[Y_1]$, denoted $\inq_S(u)$.  Thus we obtain a function
\begin{equation}\label{inqSu}
\inq_S : V_S Y_2 \to \SLI[Y_1]
\end{equation}
defined from the set $V_S Y_2$ of secondary vertices of a finite irreducible $\A$-graph  $Y_2$ with property (B)
to the set of inequalities of $\SLI[Y_1]$.
\smallskip

If $q$ is an inequality of the system $\SLI[Y_1]$, denoted  $q \in \SLI[Y_1]$, we let $q^L$ denote the left hand side of $q$, let  $q^R$ denote the integer in the right hand side of $q$ and  let $k(q) \ge 2$ denote the parameter $k$ for $q$,  see the definition of inequalities \eqref{inqa}--\eqref{inqb}.

\begin{lem}\label{lem1}
Suppose  $Y_2$ is a finite irreducible   $\A$-graph such that  the map $\tau_2  :  \core (Y_1  \times  Y_2) \to Y_2$ is surjective, $\core(Y_2) = Y_2$ and $\deg Y_2 \le d$.  Then $\inq_S(V_S Y_2) \subseteq  \SLI_d[Y_1]$. Furthermore,
 \begin{align*}
\sum_{u \in V_S Y_2}  \inq_S(u)^L & = -2 \brr (Y_2) x_s ,  \\
\sum_{u \in  V_S Y_2}  \inq_S(u)^R & = -2 \brr ( \core(Y_1 \times  Y_2) )   .
\end{align*}
\end{lem}

\begin{proof}  The inclusion  $\inq_S(V_S Y_2) \subseteq  \SLI_d[Y_1]$   is evident from the definitions.
\smallskip

Suppose $v \in V_P Y_2$ and  let
$e_1$, $e_2$ be the edges such that
$(e_1)_- =  (e_2)_- =  v$  and  $u_\al := (e_\al)_+$, $\al =1,2$, is a secondary vertex of type $\al$ in  $Y_2$.
Clearly, $\ph(e_\al)  \in G_\al$ for $\al =1,2$.
Denote $A_v := \tau_1 \tau_2^{-1}(v)$. It follows from the definitions that  $A_v \in S_1( V_P Y_1)$ and that the variables   $-x_{A_v}$, $x_{A_v}$  occur in
$\inq_S(u_1)^L$, $\inq_S(u_2)^L$, resp., and will cancel out in the sum $\inq_S(u_1)^L + \inq_S(u_2)^L$. It is easy to see that  all occurrences of variables $\pm x_A$, $A \in S_1( V_P Y_1)$, in the formal sum
\begin{equation}\label{fsum2}
\sum_{u \in V_S Y_2 } \inq_S(u)^L ,
\end{equation}
before any cancellations are made, can be  paired down by using primary vertices  of $Y_2$ as indicated above.  Since every secondary vertex $u$ of $Y_2$  contributes $-(\deg u -2)$ to the coefficient
of $x_s$ in the sum  \eqref{fsum}  and
\begin{equation}\label{add1a}
\sum_{u \in V_S Y_2}  (\deg u -2) = 2 \brr (Y_2)  ,
\end{equation}
it follows that the first equality of Lemma~\ref{lem1} is true.  The second equality follows from
the analogous to \eqref{add1a} formula
\begin{equation*}
\sum_{u \in V_S (\core(Y_1 \times  Y_2))}  (\deg u -2) = 2\brr (\core(Y_1 \times  Y_2)) ,
\end{equation*}
and from the definition \eqref{Nal} of numbers $N_\al(\Omega_{T_u})$, $u \in V_S Y_2$. Here the function
$\Omega_{T_u}$ is defined for $u$ as in  \eqref{STu}.
\end{proof}

Let $A$ be a finite set. A {\em combination with repetitions} $B$ of $A$, which we denote
$$
B = [[ b_1, \dots, b_\ell ]]  \sqsubseteq  A ,
$$
is a finite unordered collection of multiple copies of elements of $A$. Hence,  $b_i \in A$ and $b_i = b_j$ is possible when $i \ne j$. If $B = [[ b_1, \dots, b_\ell ]]$ is a combination with repetitions
then the cardinality $|B|$ of  $B$ is $|B| := \ell$.
\medskip

Observe that a finite irreducible $\A$-graph $Y_2$ with  property (Bd)  can be used to construct a combination  with repetitions, denoted
$$
\inq_d(V_S Y_2),
$$
of the system   $\SLI_d[Y_1]$, whose elements are individual inequalities of $\SLI_d[Y_1]$
so that every inequality $q = \inq_S(u)$ of  $\SLI_d[Y_1]$, see \eqref{inqSu},  occurs in $\inq_d(V_S Y_2)$ as many times as the number
of preimages of $q$ in $V_SY_2$ under $\inq_S$.  Note that, in general, $\inq_d(V_S Y_2) \ne \inq_S(V_S Y_2)$ because
$\inq_S(V_S Y_2)$ is a subset of $\SLI_d[Y_1]$ while   $\inq_d(V_S Y_2)$ is a combination  with repetitions of $\SLI_d[Y_1]$.
\smallskip

It follows from Lemma~\ref{lem1} that if $\inq_d(V_S  Y_2)  = [[ q_1, \dots, q_\ell ]]$ is a combination  of $\SLI_d[Y_1]$, then
$$
\sum_{q \in  \inq_d (V_S  Y_2) }^{}q^L := \sum_{j=1}^{m} q_j^L = - 2 \brr(Y_2)  x_s  .
$$

In the opposite direction, we will prove the following.

\begin{lem}\label{lem2} Suppose $Q$ is a nonempty combination with repetitions of
$\SLI_d[Y_1]$  and
\begin{equation}\label{l2i11}
\sum_{q \in  Q}^{}q^L = -C(Q) x_s  ,
\end{equation}
where $C(Q) > 0$ is an integer. Then there exists a  finite irreducible $\A$-graph $Y_{2, Q}$ with property (Bd) such that, letting
$ \wtl Q = \inq_d(V_SY_{2, Q})$, one has  $|\wtl Q | = |Q|$ and
\begin{align}\label{l2i2}
\sum_{q \in  Q}^{}q^L & = \sum_{q \in \wtl Q}^{}q^L =     - 2\brr(Y_{2, Q}) x_s , \\  \label{l2i2a}
 \sum_{q \in  Q}^{}q^R & \ge \sum_{q \in \wtl Q}^{}q^R  =  - 2\brr(  \core(Y_1  \times  Y_{2, Q})  )    .
\end{align}
\end{lem}

\begin{proof} We will construct an $\A$-graph $Y_{2, Q}$ whose secondary vertices $u_j$  are in bijective correspondence
$$
u_j \mapsto q_j ,
$$
where $j=1, \dots, {|Q|}$, with elements of the combination
$$
Q = [[ q_1, \dots, q_{|Q|} ]]   \sqsubseteq \SLI_d[Y_1]
$$
so that the secondary vertices of type $\al=1$ in  $ Y_{2, Q}$ correspond to the inequalities of type \eqref{inqa} in $Q$,
and   the secondary vertices of type $\al=2$ in  $Y_{2, Q}$ correspond to the inequalities of type \eqref{inqb} in $Q$.
\smallskip

To fix the notation, we
let the inequality  $q_j$ of $Q$ be defined by means of an $\al_j$-admissible function
$$
\om_{T_j} : T_j \to S_1( V_P Y_1) ,
$$
where $T_j \in S_2(G_{\al_j})$ and $T_j = \{ b_{1,j}, \dots, b_{k_j, j}\}$,
 $2 \le  k_j \le d$, $b_{i, j} \in  G_{\al_j}$. Here
$k_j = k(q_j)$ denotes the parameter $k$ for $q_j$,  see  \eqref{inqa}--\eqref{inqb}.
\smallskip

Consider a secondary vertex $u_j$ of type $\al_j$ and $k_j$ edges $e_{1,j}, \dots, e_{k_j, j}$ whose terminal vertex is $u_j$ and whose $\ph$-labels are
$$
\ph(e_{1,j}) = b_{1,j}, \ \dots , \ \ph(e_{k_j, j}) = b_{k_j, j} .
$$
This is the local structure of the graph $Y_{2, Q}$ around its secondary vertices.
\smallskip

Now we will identify in pairs the initial vertices
of the edges  $e_{1,j}, \dots, e_{k_j, j}$,  $j=1, \dots, {|Q|}$,   which will form the set of primary vertices $V_P Y_{2, Q}$ of the graph $Y_{2, Q}$.  In the notation introduced above, it follows from the definitions \eqref{inqa}--\eqref{inqb} that a typical term  $\pm x_A$ of $q_j^L$ has the form  $(-1)^{\al_j}x_A$, where $A= { \om_{T_j}(b_{i,j})}$ for some $i = 1, \dots, k_j$.
\smallskip

It follows from the equality \eqref{l2i11} that there is an involution  $\iota$ on the set of all terms $\pm x_A$  of the formal sum
\begin{equation}\label{fsum}
\sum_{j=1}^{|Q|} q_j^L
\end{equation}
such that $\iota$ takes every term  $\pm x_A$ of $q_{j_1}^L$ to a term  $\mp x_A$ of $q_{j_2}^L$, $j_1 \ne j_2$, and $\iota^2 = \mbox{id}$. Therefore, if
$$
(-1)^{\al_{j_1} }x_{ \om_{T_{j_1}} (b_{i_1,{j_1}})} \quad \text{and} \ \quad  (-1)^{\al_{j_2}} x_{ \om_{T_{j_2}}(b_{i_2,j_2})}
$$
are two terms of the formal sum \eqref{fsum} which are $\iota$-images of each other,  then
$$\{ \al_{j_1},  \al_{j_2} \} = \{ 1,2 \}     \quad \text{and} \ \quad     \om_{T_{j_1}} (b_{i_1,j_1}) = \om_{T_{j_2}}(b_{i_2,j_2}) .
$$
We identify  the initial vertices of the edges $e_{ i_1,j_1}$,  $e_{i_2, j_2}$ so that the vertex
$$
(e_{  i_1,j_1 })_- = (e_{ i_2, j_2 })_-
$$
becomes a primary vertex of  $Y_{2, Q}$. We do this identification of the initial vertices of all pairs of edges, corresponding as described above to all pairs of terms $\pm x_A$, $\mp x_A$ in \eqref{fsum}  that are $\iota$-images of each other.
As a result, we obtain an $\A$-graph $Y_{2, Q}$.
It is clear from the definitions that $Y_{2, Q}$ is a finite irreducible $\A$-graph such that the degree of any secondary vertex $u_j$ of $Y_{2, Q}$
is $k_j$ such that
$$
2 \le  k_j = |T_j|  \le d ,
$$
and the degree of every primary  vertex of $Y_{2, Q}$ is 2.
\smallskip

Looking at the coefficients of $-x_s$ in  \eqref{fsum}, we can see from  \eqref{l2i11}, \eqref{inqa}--\eqref{inqb} that
$$
C(Q) =  \sum_{j=1}^{|Q|} (k_j-2) >0 .
$$
Hence, the graph $Y_{2, Q}$ has a vertex of degree at least $3$.
\smallskip

Therefore, $Y_{2, Q}$ is a finite irreducible $\A$-graph such that $\core(Y_{2, Q}) = Y_{2, Q}$ and $\brr(Y_{2, Q}) >0$. Note that $Y_{2, Q}$ is not uniquely determined by $Q$ (because there are many choices to define the involution $\iota$, i.e., to do cancellations in the left hand side of \eqref{l2i11}).
\smallskip

Consider the graph $\core( Y_1 \times Y_{2, Q})$ and the associated graph maps
$$
\al_1 : \core( Y_1 \times Y_{2, Q}) \to Y_1 , \quad  \al_2 : \core( Y_1 \times Y_{2, Q})  \to Y_{2, Q}  .
$$
It follows from the definitions, in particular, from the $\al$-admissibility of functions
$\Omega_{T_{1}}, \ldots, \Omega_{T_{|Q|}}$,   that $\al_2$ is surjective. Hence,  $Y_{2, Q}$ has property (Bd).
\smallskip

Let  the sets  $A_1(u_j), \ldots, A_{k_j}(u_j)$ be defined for a secondary vertex  $u_j$ of $Y_{2, Q}$ as in \eqref{defA}  so that   $A_i(u_j)$ is defined by means of the primary vertex $(e_{i,j})_-$, where $ i = 1, \ldots, k_j$.
It is not difficult to see from the definitions that
$$
\Omega_{T_{j}}( b_{1,j}) \subseteq   A_1(u_j), \  \ldots , \   \Omega_{T_{j}}( b_{k_j,j}) \subseteq   A_{k_j}(u_j) .
$$
This observation means that if
$\wtl Q := \inq_d(V_SY_{2, Q})$  then $|\wtl Q | = |Q|$ and $Y_{2, \wtl Q} = Y_{2, Q} $ for a suitable involution $\wtl \iota =\wtl \iota(\wtl Q)$.
\smallskip

Hence, if $q_j$  has the form \eqref{inqa}, where $ k_j = k(q_j)$ as before, then we have
$$
\inq_S(u_j)^L = -x_{A_1(u_j)} - \cdots  -x_{A_{k_j}(u_j)} - (k_j-2)x_s ,
$$
where $\Omega_{T_{j}}( b_{i,j}) \subseteq   A_i(u_j)$ for $i = 1,  \ldots , k_j$, and
$$
N_1(\Omega_{T_{j}} )  \le   N_1( \wtl  \Omega_{T_{j}} ) = -  \inq_S(u_j)^R ,
$$
here $\wtl  \Omega_{T_{j}}$ is the function,
$$
\wtl  \Omega_{T_{j}}  : \{ b_{1,j},  \dots , b_{k_j, j} \} \to S_1(V_p Y_1) ,
$$
defined by  $\wtl  \Omega_{T_{j}} (b_{i,j}) := A_i(u_j) $ for $i = 1,  \ldots , k_j$.
\smallskip

Analogously, if $q_j$  has the form \eqref{inqb}, where $ k_j = k(q_j)$, then we have
$$
\inq_S(u_j)^L = x_{A_1(u_j)} + \cdots  + x_{A_{k_j}(u_j)} - (k_j-2)x_s ,
$$
where $\Omega_{T_{j}}( b_{i,j}) \subseteq   A_i(u_j)$ for $i = 1,  \ldots , k_j$, and
$$
N_1(\Omega_{T_{j}} )  \le   N_1(   \wtl  \Omega_{T_{j}}   ) = -  \inq_S(u_j)^R ,
$$
here $\wtl  \Omega_{T_{j}} $ is the function,
$$
\wtl  \Omega_{T_{j}}  : \{ b_{1,j},  \dots , b_{k_j, j} \} \to S_1(V_p Y_1) ,
$$
defined by  $\wtl  \Omega_{T_{j}} (b_{i,j}) := A_i(u_j) $ for $i = 1, \ldots, k_j$.
\smallskip

Therefore,
$$
\sum_{q \in  \inq_d(V_S Y_{2, Q}) } q^R  =  \sum_{q \in   \wtl Q}^{}q^R    \le \sum_{q \in Q}^{}q^R .
$$
Now both  \eqref{l2i2}--\eqref{l2i2a}  follow from Lemma~\ref{lem1}.
\end{proof}

We  summarize Lemmas~\ref{lem1}--\ref{lem2} in the following.

\begin{lem}\label{lem3}  The function
\begin{equation*}
\inq_d : Y_2 \mapsto \inq_d(V_S Y_2) = Q
\end{equation*}
from the set of finite irreducible $\A$-graphs $Y_2$ with property (Bd) to the set of combinations $Q$ with repetitions of the system $\SLI_d[Y_1]$ with the property
$\sum_{q \in Q}^{}q^L = -C(Q) x_s $,
where $C(Q)>0$ is an integer, is such that
\begin{equation*}
\sum_{q \in \inq_d(V_S Y_2) }^{}q^L    =   -2\brr(Y_2) x_s  \quad   \mbox{and}   \quad
\sum_{q \in   \inq_d(V_S Y_2) }^{}q^R  =   -2\brr(\core(Y_1 \times  Y_2)) .
\end{equation*}

In addition, for every $Q$ in the codomain of the function $\inq_d$,  there exists a graph $Y_{2, Q}$ in the domain of $\inq_d$  such that, letting
 $\wtl Q = \inq_d(V_S Y_{2, Q}) $, one has  $|\wtl Q | = |Q|$ and
\begin{align*}
\sum_{q \in  Q}^{}q^L & = \sum_{q \in \wtl Q}^{}q^L =     - 2\brr(Y_{2, Q}) x_s , \\
 \sum_{q \in  Q}^{}q^R & \ge \sum_{q \in \wtl Q}^{}q^R  =  - 2\brr(  \core(Y_1  \times  Y_{2, Q})  )    .
\end{align*}
\end{lem}

\begin{proof} This is straightforward  from Lemmas~\ref{lem1}--\ref{lem2} and their proofs.
\end{proof}

\section{Utilizing Linear and Linear Semi-Infinite Programming}

First we briefly review relevant results from the theory of linear programming (LP) over the field $\mathbb Q$ of rational numbers. Following the notation of Schrijver's monograph \cite{S86}, let $A \in \mathbb Q^{m'\times n'}$ be an $m'\times n'$-matrix, let $b \in \mathbb Q^{m'\times 1} = \mathbb Q^{m'} $ be a column vector,  let  $c \in \mathbb Q^{1\times n'} $ be a row vector, $c=(c_1, \ldots, c_{n'})$,
and let $x$ be a column vector consisting of variables $x_1, \dots, x_{n'}$, so $x = (x_1, \dots, x_{n'})^{\top}$, where $M^{\top}$ means the transpose of a matrix $M$.
The inequality $x \ge 0$ means that $x_i \ge 0$ for every $i$.
\smallskip

A typical LP-problem asks about the maximal value of the objective linear function
$$
cx= c_1x_1+\dots +c_{n'}x_{n'}
$$
over all $x \in \mathbb Q^{n'}$ subject to a finite system of linear inequalities $Ax \le b$.
This value (and often the LP-problem itself) is denoted
$$
\max\{ cx \mid Ax \le b  \} .
$$

We write $\max\{ cx \mid Ax \le b    \} = -\infty$ if the set $\{ cx \mid Ax \le b    \}$ is empty.
We write $\max\{ cx \mid Ax \le b    \} = +\infty$ if the set $\{ cx \mid Ax \le b    \}$ is unbounded from above and say that
 $\max\{ cx \mid Ax \le b   \} $ is finite if the set $\{ cx \mid Ax \le b    \}$ is nonempty and bounded from above. The notation and terminology for an LP-problem
 $$
 \min\{ cx \mid Ax \le b    \} = -  \max\{ -cx \mid Ax \le b    \}
 $$
 is analogous with
 $-\infty$ and $+\infty$ interchanged.
\smallskip

 If $\max\{ cx \mid Ax \le b   \} $ is an LP-problem as defined above, then the problem
 $$
 \min\{ b^{\top}y \mid A^{\top}y = c^{\top}, y\ge 0  \} ,
 $$
where  $y = (y_1, \dots, y_m)^{\top}$, is called the {\em dual} problem of the {\em primal}
LP-problem  $\max\{ cx \mid Ax \le b  \}$.
\smallskip

The (weak) duality theorem of linear programming can be stated as follows, see \cite[Section 7.4]{S86}.

\begin{TA} Let  $\max\{ cx \mid Ax \le b   \} $ be an LP-problem and let
$\min\{ b^{\top}y \mid A^{\top}y = c^{\top}, y\ge 0  \} $ be its dual LP-problem.
  Then
 for every $x \in  \mathbb Q^{n'}$ such that  $ Ax \le b$  and for every $y \in  \mathbb Q^{m'}$ such that  $A^{\top}y = c^{\top}$,  \ $y\ge 0$,  one has \ $c x  = y^{\top} Ax \le  b^{\top} y$ and
 \begin{equation}\label{dt}
\max\{ cx \mid Ax \le b    \}  =  \min\{ b^{\top}y \mid A^{\top}y = c^{\top}, y\ge 0  \}
\end{equation}
provided both polyhedra  $\{ x \mid Ax \le b \}$ and  $\{  y \mid A^{\top}y = c^{\top}, y\ge 0  \}$ are not empty.  In addition,
the minimum, whenever it is finite,  is attained at a vector $\wht y$  which is a vertex of
the polyhedron $\{ y  \mid A^{\top}y = c^{\top}, y\ge 0 \}$.
\end{TA}

Since  the system of inequalities $\SLI[Y_1]$, as defined in Section~3, is infinite in general, we also recall basic terminology and results regarding   duality in linear semi-infinite programming (LSIP), see \cite{LSIP1}, \cite{LSIP2}, \cite{LSIP3}.  Consider a generalized LP-problem $\max\{ cx \mid Ax \le b    \}$ that has countably many linear inequalities in the system $Ax \le b$ while the number of variables in $x$ is still finite. Hence, in this setting,    $A$ is a matrix with countably many rows and $n'$ columns, or $A  \in \mathbb Q^{\infty \times n'}$ is an $\infty \times n'$-matrix, $b \in \mathbb Q^{\infty \times 1} = \mathbb Q^{\infty }$, or $b$ is an infinite column vector,  $c \in \mathbb Q^{1\times n'} $ is a row vector, and  $x = (x_1, \dots, x_{n'})^{\top}$.
\smallskip

A typical LSIP-problem over $\mathbb Q$ asks about the supremum  of the objective linear functional $cx$  over all $x \in \mathbb Q^{n'}$ subject to $Ax \le b$.  This number  and the problem itself is denoted $\sup \{ cx \mid Ax \le b \}$. As above, we write $\sup \{ cx \mid Ax \le b \} = -\infty$ if the set $\{ cx \mid Ax \le b    \}$ is empty,  $\sup \{ cx \mid Ax \le b    \} = +\infty$ if the set $\{ cx \mid Ax \le b    \}$ is not bounded from above and say that
 $\sup \{ cx \mid Ax \le b    \} $ is finite if the set $\{ cx \mid Ax \le b    \}$ is nonempty and bounded from above.  The notation and the terminology for an LSIP-problem
 $\inf \{ cx \mid Ax \le b \} = -  \sup \{ -cx \mid Ax \le b    \} $ is analogous with $-\infty$ and $+\infty$ interchanged.  Let $A_i$ denote the submatrix of $A$ of size $i \times n'$ whose first $i$ rows are those of $A$ and $b_i$ is the starting subcolumn of $b$ of  length $i$. Then
 $\max\{ cx \mid A_i x \le b_i \} $ is an LP-problem which is called the  {\em $i$-approximate} of the LSIP-problem  $\sup \{ cx \mid Ax \le b  \}$.
\smallskip

 Let  $M_i = \max\{ cx \mid A_i x \le b_i \}$ denote the optimal value of the $i$-approximate LP-problem  $\max\{ cx \mid A_i x \le b_i \}$ and $M$ is the number $\sup \{ cx \mid Ax \le b    \}$. Clearly, for every $i$, $M_i \ge M_{i+1}  \ge M$.    Note that in general  $\lim_{i \to \infty} M_i \ne M$, see  \cite{LSIP1}, \cite{LSIP2}.

Similarly  to  \cite{LSIP1}, \cite{LSIP2},  \cite{LSIP3},  we say that
if $\sup \{ cx \mid Ax \le b   \} $ is an LSIP-problem as above, then the problem
$$
\inf \{ b^{\top}y \mid A^{\top}y = c^{\top}, \ y\ge 0  \} ,
$$
where  $y = (y_1, y_2, \dots )^{\top}$ is an infinite vector whose set of nonzero components is finite, is called the {\em dual problem} of  $\sup \{ cx \mid Ax \le b  \}$.
\smallskip

For later references, we state the analogue of Theorem~A for linear semi-infinite programming which, in fact, is an easy corollary of Theorem~A.

\begin{TB}  Suppose that $\sup \{ cx \mid Ax \le b    \}$ is an LSIP-problem
whose set $\{ cx \mid Ax \le b    \}$ is nonempty and bounded from above and whose dual problem is
$\inf \{ b^{\top}y \mid A^{\top}y = c^{\top}, \ y\ge 0  \} $.  Then
 \begin{equation}\label{dtt}
\sup \{ cx \mid Ax \le b    \}  \le   \inf \{ b^{\top}y \mid A^{\top}y = c^{\top}, \ y\ge 0  \}
\end{equation}
and the equality holds  if and only if
$\sup \{ cx \mid Ax \le b    \}$ is equal to  $\lim_{i \to \infty} M_i$,
where $M_i := \max\{ cx \mid A_i x \le b_i \}$ is the optimal solution of  the
$i$-approximate LP-problem $\max\{ cx \mid A_i x \le b_i \}$
of the primal LSIP-problem  $\sup \{ cx \mid Ax \le b  \}$.
\end{TB}

In the situation when the inequality \eqref{dtt} is strict,  the difference
$$
\inf \{ b^{\top}y \mid A^{\top}y = c^{\top}, \ y\ge 0  \} - \sup \{ cx \mid Ax \le b    \} >0
$$
is called the {\em duality gap} of the LSIP-problem $\sup \{ cx \mid Ax \le b \}$.
\medskip

We now consider the problem of maximizing the objective linear function
$$
cx :=  -x_s
$$
over all rational vectors $x$, $x \in \mathbb Q^{n'}$ for a suitable $n'$, subject to the system of linear inequalities $\SLI[Y_1]$, see  \eqref{slio}, as an LSIP-problem $\sup \{ cx \mid Ax \le b  \}$.
\medskip

We also consider a subsequence of $m_{\inq, d}$-approximate LP-problems
$$
\max \{ cx \mid A_{m_{\inq, d}} x \le b_{m_{\inq, d}}  \}
$$
of the LSIP-problem $\sup \{ cx \mid Ax \le b  \}$
whose systems $A_{m_{\inq, d}}x  \le b_{m_{\inq, d}}$ of inequalities are finite subsystems  $\SLI_d[Y_1]$ of  $\SLI[Y_1]$,
 where  $d = 3,4, \dots$, as defined in \eqref{slid}.
\smallskip

It is straightforward to verify that the dual problem
\begin{equation*}
\inf \{ b^{\top} y  \mid  A^{\top} y = c^{\top},  \ y\ge 0 \}
\end{equation*}
of this LSIP-problem $\sup \{ cx \mid Ax \le b  \}$   can be equivalently stated as follows.
\begin{equation}\label{dlpp}
\sum_{j=1}^{\infty } y_j q_j^R \to \inf  \quad  \mbox{subject to} \quad   y \ge 0 , \  \
\sum_{j=1}^{\infty } y_j q_j^L = - x_s ,
\end{equation}
where almost all  $y_j$, $j=1,2,\dots$,  are zeros.  We rewrite
\eqref{dlpp} in the form
\begin{equation}\label{dlp}
\inf \bigg\{ \sum_{j=1}^{\infty} y_j q_j^R  \  \Big|  \     y \ge 0 ,   \sum_{j=1}^{\infty} y_j q_j^L = - x_s  \bigg\} .
\end{equation}

Analogously, the dual problem  of the  $m_{\inq, d}$-approximate LP-problem
$$
\max \{ cx \mid A_{m_{\inq, d}} x \le b_{m_{\inq, d}}  \}
$$
can be stated in the form
\begin{equation*}
\sum_{j=1}^{m_{\inq, d} } y_j q_j^R \to \min  \quad   \mbox{subject to} \quad  y \ge 0,  \
\sum_{j=1}^{m_{\inq, d}} y_j q_j^L = - x_s
\end{equation*}
which  we write as follows.
\begin{equation}\label{dlpd}
\min  \bigg\{ \sum_{j=1}^{m_{\inq, d}} y_j q_j^R  \  \Big| \   y \ge 0 ,    \sum_{j=1}^{m_{\inq, d}} y_j q_j^L = - x_s   \bigg\} .
\end{equation}

In Lemma~\ref{lem3},  we established the existence of a function
$$
\inq_d  :   Y_2 \mapsto  \inq_d(Y_2)
$$
from the set of finite irreducible $\A$-graphs $Y_2$ with property (Bd) to a certain set of combinations with  repetitions of $\SLI_d[Y_1]$.
Now we will relate these combinations
with repetitions of $\SLI_d[Y_1]$ to solutions of the dual LP-problem \eqref{dlpd}.
\medskip

Consider a  combination
 with repetitions $Q$ of $\SLI_d[Y_1]$ that has the property that
 \begin{equation}\label{st1}
\sum_{q \in Q} q^L = -C(Q) x_s     ,
\end{equation}
where $C(Q) >0$ is an integer. As above in \eqref{dlpd}, let the inequalities of $\SLI_d[Y_1]$ be indexed and let
\begin{equation*}
\SLI_d[Y_1] = \{ q_1, \dots, q_{m_{\inq, d}}   \} .
\end{equation*}
 Let  $\eta_j(Q)$ denote the number of times that $q_j$ occurs in $Q$, and let $\kappa_j$ be the coefficient of $x_s$ in $q_j$.  Then it follows from the definitions and \eqref{st1} that
\begin{equation}\label{st2}
     \sum_{q \in Q} q^L =  \sum_{j=1 }^{m_{\inq, d}} \kappa_j \eta_j(Q) x_s =  -C(Q) x_s    .
\end{equation}

Consider the map
\begin{equation}\label{dfyy}
\sol_d :  Q \mapsto y_Q = ( y_{Q,1}, \dots,  y_{Q, m_{\inq, d}} )^{\top}   ,
\end{equation}
where
$y_{Q,j} :=  \frac{\eta_j(Q)}{C(Q)}$ for $j =1, \dots, m_{\inq, d}$.
It follows from the definitions that $y_Q$ is a rational vector, $y_Q \ge 0$ and, by \eqref{st2}, $y_Q$ satisfies the condition
$$
\sum_{j=1}^{m_{\inq, d}} y_{{Q,j}} q_j^L  = - x_s    .
$$
Hence, $y_Q$ is a vector in the feasible polyhedron
\begin{equation}\label{fpdp}
\bigg\{  y  \  \Big| \   y \ge 0, \sum_{j=1}^{m_{\inq, d}} y_{j} q_j^L   = - x_s  \bigg\}
\end{equation}
 of the dual  LP-problem \eqref{dlpd}.
\smallskip

Note that, in place of \eqref{dfyy},  we could also write
\begin{equation}\label{solQ}
\sol_d :  Q \mapsto  C(Q)^{-1}  \eta(Q)^{\top} ,
\end{equation}
where  $\eta(Q) = (  \eta_1(Q) , \dots, \eta_{m_{\inq, d}}(Q)   )$,  as  $y_Q =   C(Q)^{-1}  \eta(Q)^{\top}$.

\smallskip

Conversely, let $z = (z_1, \dots, z_{m_{\inq, d}})^{\top}$ be a vector of the feasible polyhedron
\eqref{fpdp} of the dual  LP-problem \eqref{dlpd}. Let $C >0 $ be a common multiple of positive denominators
of the rational numbers
$z_1, \dots, z_{m_{\inq, d}}$. Consider a combination with repetitions
$Q(z)$ of  $\SLI_d[Y_1]$ such that every $q_j$ of  $\SLI_d[Y_1]$ occurs
in $Q(z)$ exactly  $C z_{j} = n_j$ many times. Then it follows from the definitions that
\begin{equation}\label{49new}
\sum_{q \in Q(z)} q^L =   \sum_{j=1}^{m_{\inq, d}} n_j q_j^L = \sum_{j=1}^{m_{\inq, d}} C z_{j}  q_j^L  =  C \sum_{j=1}^{m_{\inq, d}} z_{j}  q_j^L  = - C x_s  .
\end{equation}
Now we can see from
\begin{equation}\label{c4x}
\frac {\eta_j(Q(z)) }{C} = \frac {C z_j}{C} = z_j ,
\end{equation}
where $j = 1, \dots, m_{\inq, d}$, that the vector $y_{Q(z)} = \sol_d(Q(z))$,  defined by \eqref{dfyy} for $Q(z)$,
is equal to $z$.

\begin{lem}\label{lem4}
The map
$$
\sol_d : Q \mapsto y_Q
$$
defined by \eqref{dfyy} is a surjective function from the set of combinations $Q$ with
repetitions of $\SLI_d[Y_1]$ that
satisfy the  equation $\sum_{q \in Q} q^L = -C(Q) x_s$, where $C(Q)  >0$ is an integer,
to  the feasible polyhedron  \eqref{fpdp}   of the dual  LP-problem
\eqref{dlpd}. Furthermore, the composition of the maps
$\inq_d$ and $\sol_d$,
$$
\sol_d  \circ \inq_d : Y_2 \mapsto  \sol_d  ( \inq_d (Y_2) )  = y_{Y_2} ,
$$
is a function from the set of graphs with property (Bd) to the feasible
polyhedron \eqref{fpdp} of  the dual  LP-problem
\eqref{dlpd}. Under this map, the value of the objective function
$
 \sum_{j=1}^{m_{\inq, d}}   y_{Y_2, j}  q_j^R
$
of the  dual  LP-problem  \eqref{dlpd}  at $y_{Y_2}$ satisfies the equality
\begin{equation}\label{ps4}
\sum_{j=1 }^{m_{\inq, d}}     y_{Y_2,  j}  q_j^R    =   -\frac{\brr (\core(Y_1 \times  Y_2))}{\brr (Y_2) }     .
\end{equation}

In addition, for every $z$  in the polyhedron \eqref{fpdp},
there is a vector $ \wtl z$ in  \eqref{fpdp}   such that
$\wtl z = \sol_d  ( \inq_d (Y_2) ) $ for some graph $Y_2$ with property (Bd) and
$$
\sum_{j=1}^{m_{\inq, d}} \wtl z_{j}  q_j^R \le  \sum_{j=1}^{m_{\inq, d}} z_{ j}  q_j^R .
$$
\end{lem}

\begin{proof} As was observed above, see computations \eqref{49new}--\eqref{c4x},  $\sol_d$ is a surjective function.
\smallskip

Consider a finite irreducible $\A$-graph  $Y_2$ with property (Bd) and define
$$
Q := \inq_d(Y_2) ,  \quad y_{Y_2} := \sol_d(Q) .
$$
It follows from Lemma~\ref{lem3} that
\begin{equation}\label{eqq1}
\sum_{q \in Q} q^L = - 2\brr (Y_2) x_s  \quad \mbox{ and} \quad  \sum_{q \in Q} q^R = - 2\brr (\core(Y_1 \times  Y_2) ) .
\end{equation}
It follows from  \eqref{st2} and  \eqref{eqq1} that   $C(Q ) = 2\brr(Y_2)$. Hence, using the definition \eqref{dfyy} and  equalities  \eqref{eqq1}, we obtain
\begin{equation*}
\sum_{j=1}^{m_{\inq, d}}   y_{Y_2,  j}   q_j^R  =  \frac { \sum_{q \in Q } q^R  } {
C(Q) } = -\frac {\brr( \core(Y_1 \times  Y_2) ) }{ \brr(Y_2 ) }   ,
 \end{equation*}
as required in \eqref{ps4}.
\smallskip

To prove the additional statement, consider a vector $z$ in the polyhedron \eqref{fpdp}.

Since  $\sol_d$ is  surjective, there is a combination with repetitions $Q$ such that $\sol_d(Q) = z$. By Lemma~\ref{lem3}  for this $Q$,
there is a graph $Y_{2, Q}$ such that if $\inq_d(V_S Y_{2, Q}) = \wtl Q$ then
$|\wtl Q | = |Q|$ and
\begin{align}\label{411a}
\sum_{q \in  Q}^{}q^L & = \sum_{q \in \wtl Q}^{}q^L =     - 2\brr(Y_{2, Q}) x_s = - C(Q)x_s = - C(\wtl Q)x_s  , \\ \label{411b}
\sum_{q \in  Q}^{}q^R & \ge \sum_{q \in \wtl Q}^{}q^R  =  - 2\brr(  \core(Y_1  \times  Y_{2, Q})  )  .
\end{align}

Let $\wtl z := \sol_d( \wtl Q)$. Then, in view of \eqref{411a}--\eqref{411b}, we obtain
\begin{equation*}
\sum_{j=1}^{m_{\inq, d}}  \wtl z_j q_j^R = \frac { \sum_{q \in \wtl Q } q^R }
{ C(\wtl Q) } \le  \frac { \sum_{q \in  Q } q^R }{ C(Q) } =  \sum_{j=1}^{m_{\inq, d}}   z_j q_j^R  ,
 \end{equation*}
as required.
\end{proof}

We will say that a real number  $\sigma(Y_1)\ge 0$ is the WN-{\em coefficient}
for $Y_1$ if
\begin{equation*}
\brr ( \core(Y_1 \times  Y_2) )  \le  \sigma(Y_1)  \brr(Y_1) \brr(Y_2)
\end{equation*}
for every finite irreducible $\A$-graph $Y_2$ with property (B) and  $\sigma(Y_1)$ is minimal with this property.
\smallskip

We also consider the  WN$_d$-coefficient   $\sigma_d(Y_1)$, where $d \ge 3$ is an integer, for  $Y_1$ defined so that
\begin{equation*}
\brr ( \core(Y_1 \times  Y_2) )  \le  \sigma_d(Y_1)  \brr(Y_1) \brr(Y_2)
\end{equation*}
for every finite irreducible $\A$-graph $Y_2$ with property (Bd) and $\sigma_d(Y_1)$ is minimal with this property.
\smallskip

It is clear from the definitions that
$$
\sigma_d(Y_1) \le \sigma_{d+1}(Y_1) \le \sigma(Y_1)
$$
for every $d =3,4,\dots$  and
\begin{align}\label{supsd}
\sup_d \{ \sigma_d(Y_1)  \}= \sigma(Y_1) .
\end{align}

Observe that
$$
 \sigma(Y_1)  = \sup_{Y_2} \bigg\{  \frac {\brr(  \core(Y_1 \times  Y_2) ) }{\brr(Y_1 ) \brr(Y_2 ) } \bigg\}
$$
 over all  finite  irreducible $\A$-graphs $Y_2$ with property (B). Similarly,
\begin{align}\label{sigmd}
 \sigma_d(Y_1)   = \sup_{Y'_2}  \bigg\{ \frac {\brr(\core(Y_1 \times  Y'_2) ) }{\brr(Y_1 ) \brr(Y'_2 ) }  \bigg\}
 \end{align}
 over all  finite irreducible $\A$-graphs $Y'_2$  with property (Bd).

\begin{lem}\label{lem5} Both optima
$$
\max \{ - x_s \mid \SLI_d[Y_1]\} \quad \mbox{and} \quad \min   \bigg\{  \sum_{j=1}^{m_{\inq, d}}   y_j q_j^R   \  \Big|   \  y \ge 0 , \
\sum_{j=1}^{m_{\inq, d}}   y_j q_j^L  = - x_s \bigg\}
$$
are finite and satisfy the following inequalities and equalities
\begin{align}\label{st3}
\begin{split}
- 2 \tfrac{q^*}{q^*-2} \brr(Y_1) & \le \sup \{ - x_s \mid \SLI [Y_1]\}   \\
 & \le   \max \{ - x_s \mid \SLI_d[Y_1]\}   \\
 & = \min \bigg\{ \sum_{j=1}^{m_{\inq, d}}  y_j q_j^R  \ \Big| \  y  \ge 0 , \ \sum_{j=1}^{m_{\inq, d}}   y_j q_j^L  = - x_s \bigg\} \\
  & =   - \sigma_d(Y_1) \brr(Y_1)  .
\end{split}
\end{align}

Furthermore, the minimum is attained at a vector $\wtl y_V = \wtl y_V(d)$ of the feasible polyhedron  \eqref{fpdp} of the dual LP-problem \eqref{dlpd} such that
there is a graph $Y_{2, Q_V}$  that  has property (Bd), $\wtl y_V  = \sol_d(\inq_d(Y_{2, Q_V}))$ and the following hold
\begin{align}
\begin{split}\label{st3a}
\inf  & \bigg\{  \sum_{j=1}^{\infty}   y_j q_j^R  \ \Big| \  y \ge 0 , \
\sum_{j=1}^{\infty}   y_j q_j^L  = - x_s \bigg\} \\
 &  =  - \sigma(Y_1) \brr(Y_1) \\
& \le \min   \bigg\{  \sum_{j=1}^{m_{\inq, d}}   y_j q_j^R    \ \Big| \    y \ge 0 , \
\sum_{j=1}^{m_{\inq, d}}   y_j q_j^L  = - x_s \bigg\} \\
& =  - \sigma_d(Y_1)  \brr(Y_1) .
\end{split}
\end{align}

In particular,
$\sigma_d(Y_1) \le \sigma(Y_1) \le 2 \frac{q^*}{q^*-2}$.
\end{lem}

\begin{proof} Recall that every primary vertex of $Y_1$ has degree 2 and $d \ge 3$. Hence,
if the graph $Y_1$ contains a vertex  $u$  of degree $> d$, then $u$ is secondary and we may take some edges out of $Y_1$ to get a subgraph $\wht Y_1$  of $Y_1$ such that $| E \wht  Y_1 | < | E Y_1 |$, $\brr (\wht Y_1) >0$ and
$\core(\wht Y_1) = \wht Y_1$.  It is clear that the natural projection
$$
\tau_2 : \core(Y_1 \times \wht Y_1) \to \wht  Y_1
$$
is surjective. Hence, either the graph $\wht Y_1$ has property (Bd) or, otherwise,  $\wht  Y_1$ has a vertex of degree greater than $ d$.
Iterating this argument,  we can prove that $Y_1$ contains a subgraph $Y_{1, d}$  with property (Bd).
\smallskip

Setting $Y_2 :=  Y_{1, d}$,  we obtain, by Lemma~\ref{lem4},  a solution $\wht y =  \sol_d ( \inq_d  (Y_2))$ to the equalities and inequalities
that define the  feasible  polyhedron \eqref{fpdp}  of \eqref{dlpd}.  Hence,  both sets
$$
\bigg\{  y  \ \Big| \  y \ge 0 , \sum_{j=1}^{m_{\inq, d}}  y_j q_j^L  = - x_s \bigg\} ,   \quad
\bigg\{  y  \ \Big| \  y \ge 0 , \sum_{j=1}^{\infty}  y_j q_j^L  = - x_s \bigg\}
$$
are not empty.
\smallskip

To see that the sets  $ \{ x \mid \SLI[Y_1]\}$,
$\{  x \mid \SLI_d[Y_1]\}$  are not empty either,  we will show that the vector  $\wht x$, whose components are
 $\wht x_A := 0$ for every nonempty  $A \subseteq V_P Y_1$ and $\wht  x_{s} := 2\frac{q^*}{q^*-2} \brr(Y_1)$, is a solution both to $\SLI_d[Y_1]$ and  to $\SLI[Y_1]$.  To do this,  we will check that every inequality of $\SLI[Y_1]$ is satisfied with these values of variables, that is,
 \begin{align}\label{dd1}
 -( k-2) \cdot \tfrac{2q^*}{q^*-2} \brr(Y_1) \le -   N_\al(\Omega_T)
\end{align}
 for every $\al$-admissible function
$$
  \Omega_T : T \to S_1( V_P Y_1 ) ,
$$
where $T \in S_2(G_\al)$ and $|T| = k$.
\smallskip

Let  $T = \{ a_1, \dots, a_k\}$, $k \ge 2$, $a_i \in G_\al$,
  and $\Omega_T(a_i) = A_i$, $i =1,\dots, k$.
\smallskip

Consider a secondary vertex $u$ of $Y_1$, suppose $\deg u = \ell$ and let
$e_1, \dots, e_{\ell}$ be all edges of $Y_1$ such that
 $u = (e_1)_+ = \dots = (e_\ell)_+$.
Denote
 $$
 B := \{   \ph(e_1),  \dots,  \ph(e_\ell) \} .
 $$
 It is not difficult to see from the  definition  \eqref{Nal} of the number $N_\al(\om_T)$ that  the contribution to the sum $N_\al(\Omega_T)$,  made by those equivalence classes  that are associated with the vertex  $u \in V_S Y_1$,  does not exceed
$$
\sum_{g \in G_\al}  \max( | T \cap  B g | -2 ,  0) .
$$
Hence, it follows from the definition of the number $\frac{q^*}{q^*-2}$, see  \eqref{di}, and  from the results of Dicks and the author
\cite[Corollary 3.5]{DIv}  that
\begin{align}\label{dd2}
\begin{split}
\sum_{g \in G_\al} \max(|T \cap Bg | -2, 0)  & \le \tfrac{q^*}{q^*-2}(|T|-2)(|B|-2) \\
  & = \tfrac{q^*}{q^*-2}(k-2)(\ell-2) .
\end{split}
\end{align}
Therefore, summing up  inequalities \eqref{dd2} over all $u \in V_S Y_1$, we obtain
\begin{align*}
N_\al(\Omega)  & \le \tfrac{q^*}{q^*-2} (k-2) \cdot 2 \brr_\al (Y_1) \\
 & \le \tfrac{q^*}{q^*-2} ( k-2 ) \cdot 2 \brr (Y_1)  ,
\end{align*}
where  $2 \brr_\al (Y_1)$ is the sum $\sum_u (\deg u -2)$ over all  secondary vertices $u \in V_S Y_1$ of type $\al$.  This proves \eqref{dd1} and also shows that
\begin{align}\label{add2a}
  -\tfrac{2q^*}{q^*-2}  \brr (Y_1)  \le   \sup \{ -x_s \mid \SLI[Y_1]\}
\end{align}
because $\wht x$ with $\wht x_s =  \frac{2q^*}{q^*-2}  \brr (Y_1) $ is a solution to  $\SLI[Y_1]$.

Therefore, both  sets $\{ x \mid \SLI[Y_1]\}$ and
$\{  x \mid \SLI_d[Y_1]\}$  are not empty as required.
\smallskip

According to Theorem~A, the maximum and minimum in \eqref{st3}  are finite and equal.
The first inequality in \eqref{st3}  is shown in  \eqref{add2a} and the second one follows from the definitions.
\smallskip

It follows from the definition   \eqref{sigmd} and  Lemma~\ref{lem4} that the supremum
\begin{equation*}
  \sup_{Y_2} \bigg\{ \frac {\brr (  \core(Y_1 \times Y_2) ) }{ \brr(Y_2 ) }  \Big\}   = \sigma_d(Y_1 )  \brr(Y_1 )  =
- \inf_{Y_2}  \bigg\{ - \frac {\brr(  \core(Y_1 \times Y_2)  ) }{ \brr(Y_2 ) }   \bigg\}
\end{equation*}
over all  graphs $Y_2$ with property (Bd) is equal to
\begin{align*}
  \sigma_d(Y_1 )  \brr(Y_1 ) &  =  -\inf \bigg\{  \sum_{j=1}^{m_{\inq, d}}   y_j q_j^R   \   \Big| \  y \ge 0 , \ \sum_{j=1}^{m_{\inq, d}}
 y_j q_j^L = - x_s \bigg\}   \\
  & =  - \min \bigg\{  \sum_{j=1}^{m_{\inq, d}}   y_j q_j^R   \ \Big| \  y \ge 0 , \ \sum_{j=1}^{m_{\inq, d}}
 y_j q_j^L = - x_s  \bigg\} ,
\end{align*}
as stated in the last  equality of \eqref{st3}.

\smallskip
The inequalities and equalities of \eqref{st3}   are now proven.
\medskip

By Theorem~A,  the minimum in  \eqref{st3a}   of the LP-problem  \eqref{dlpd}
is attained at a vertex  $y_V = y_V(d)$  of  the feasible polyhedron \eqref{fpdp}.
\smallskip

It follows from Lemma~\ref{lem4} that, for the vertex $y_V$, there exists a vector $\wtl y_V$ in the polyhedron \eqref{fpdp}   such that
$$
\sum_{j=1}^{m_{\inq, d}} \wtl y_{V, j}  q_j^R \le \sum_{j=1}^{m_{\inq, d}} y_{V, j}  q_j^R .
$$
and $\wtl  y_V = \sol_d  ( \inq_d (Y_{2, Q_V}) ) $ for some graph $Y_{2, Q_V}$ with property (Bd).
Hence, the minimum in   \eqref{st3a} is also attained at  $\wtl  y_V$.
\medskip

In view of the last  equality of \eqref{st3} and  \eqref{supsd}, we obtain
\begin{align}\label{st3a1}
\begin{split}
\inf_d  &  \bigg\{ \min   \bigg\{  \sum_{j=1}^{m_{\inq, d}}   y_j q_j^R   \ \Big| \
y \ge 0 , \  \sum_{j=1}^{m_{\inq, d}}   y_j q_j^L  = - x_s \bigg\} \bigg\} \\
& = \inf_d \{  - \sigma_d(Y_1)  \brr(Y_1)   \}  \\
 &  =  - \sigma(Y_1) \brr(Y_1) \\
& \le \min   \bigg\{  \sum_{j=1}^{m_{\inq, d}}   y_j q_j^R    \ \Big| \    y \ge 0 , \
\sum_{j=1}^{m_{\inq, d}}   y_j q_j^L  = - x_s \bigg\} \\
& =  - \sigma_d(Y_1)  \brr(Y_1) .
\end{split}
\end{align}

On the other hand,  it is clear that
\begin{align}\label{st3a2}
\begin{split}
 \inf &  \bigg\{ \sum_{j=1}^{\infty}   y_j q_j^R   \ \Big| \  y \ge 0 , \
\sum_{j=1}^{\infty}   y_j q_j^L  = - x_s \bigg\} \\  & = \inf_d \bigg\{  \min   \bigg\{  \sum_{j=1}^{m_{\inq, d}}   y_j q_j^R   \ \Big| \
y \ge 0 , \  \sum_{j=1}^{m_{\inq, d}}   y_j q_j^L  = - x_s \bigg\} \bigg\} .
\end{split}
\end{align}
Now the equalities and  inequalities \eqref{st3a} follow from \eqref{st3a1}--\eqref{st3a2}

The inequalities $\sigma_d(Y_1) \le \sigma(Y_1)  \le  2\frac{q^*}{q^*-2}$ follow from  \eqref{st3} and \eqref{supsd}.
\end{proof}

\begin{lem}\label{lem6}
There exists a finite irreducible $\A$-graph $Y_{2, Q_V} = Y_{2, Q_V}(Y_1)$ with property (Bd) such that
$$
\brr(  \core(Y_1 \times Y_{2, Q_V})  )   =   \sigma_d(Y_1 ) \brr(Y_1 ) \brr(Y_{2, Q_V} ) ,
$$
$Y_{2, Q_V}$ is connected, and
$$
| E Y_{2, Q_V} |  < 2^{2^{  |E Y_1|/4 + \log_2\log_2(4d)}} .
$$
\end{lem}

\begin{proof}  According to Lemma~\ref{lem5} and to Theorem~A,  we may assume that
the minimum of the dual LP-problem \eqref{dlpd}
is attained at a vertex $y_V$  of  the feasible polyhedron \eqref{fpdp}
 of \eqref{dlpd}.
\smallskip

It is convenient to switch back to the general LP and LSIP notation as was introduced
in the beginning of this Section.
In particular,  let  $A_{m_{\inq, d}}x \le b_{m_{\inq, d}}$ be the matrix form  of the system
\eqref{slid}.
Since $y_V$ is a vertex solution of the LP-problem   \eqref{dlpd} and \eqref{dlpd}  is stated in the form
$$
\min \{ b_{m_{\inq, d}}^{\top} y \mid A_{m_{\inq, d}}^{\top} y = c^{\top}, y \ge 0 \} ,
$$
it follows that the vertex solution $y_V$ will satisfy $m_{\inq, d}$  equalities among
$$
 A_{m_{\inq, d}}^{\top} y = c^{\top} ,  \quad y_j = 0,  \ \ j = 1, \ldots, {m_{\inq, d}} ,
 $$
 whose left hand side parts  are linearly independent (as formal linear combinations in variables  $y_1, \ldots, y_{m_{\inq, d}}$).
We call these ${m_{\inq, d}}$   equalities  {\em distinguished}.
\smallskip

The foregoing observation implies that there are $r$, $r \le {m_{\inq, d}}$, distinguished equalities in the system
$A^{\top}_{m_{\inq, d}} y = c^{\top}$ such that the submatrix $A_{m_{\inq, d}, r}^{\top}$ of $A_{m_{\inq, d}}^{\top}$,
consisting of the rows of $A_{m_{\inq, d}}^{\top}$ that correspond to the $r$ distinguished equalities, has the following property.
The rank of $A_{m_{\inq, d}, r}^{\top}$ is $r$ and deletion of the columns  of $A_{m_{\inq, d}, r}^{\top}$, that correspond to the variables $y_j$ that in turn correspond to
the distinguished equalities $y_j = 0$,  produces an $r\times r$  matrix  $A_{m_{\inq, d}, r \times r}^{\top}$ with
$\det A_{m_{\inq, d}, r \times r}^{\top} \ne 0$. Reordering the equalities in the system $A_{m_{\inq, d}}^{\top} y = c^{\top}$ and variables $y_j$  if necessary, we may assume that $A_{m_{\inq, d}, r }^{\top}$ consists of the first $r$ rows of $A_{m_{\inq, d}}^{\top}$ and
$A_{m_{\inq, d}, r \times r}^{\top}$ is an upper left submatrix of $A_{m_{\inq, d}}^{\top}$.
\smallskip

Let
$$
\bar y_V = (y_{V,1}, \dots,  y_{V,r})
$$
be the truncated version of $y_V$ consisting of the first $r$ components. It follows from the definitions that
$\bar y_V$ contains all nonzero components of $y_V$ and
$$
A_{m_{\inq, d}, r \times r}^{\top}  \bar y_V  =  \bar c^{\top} = (c_1, \dots,  c_r)^{\top} .
$$

Since $\sum_{j=1}^{{m_{\inq, d}}} y_{V,j} q_j^L  = - x_s$, it follows that $c_{i} =0$ if $c_{i}$ corresponds to  a variable $x_{B}$ and $c_i =-1$ if $c_i$ corresponds to the variable  $x_{s}$.
Since $y_V  \ne 0$ following from the definition of the LP-problem \eqref{dlpd},  we conclude that $\bar c^{\top} \ne 0$, i.e.,  one of $c_i$ is $-1$ and all other entries in  $\bar c^{\top}$ are equal to $0$.
\smallskip

Note that  every row of   $A_{m_{\inq, d}, r \times r}$  contains at most $d+1$ nonzero entries such that one is $-(k-2)$,
where $2 \le k \le d$ (this is the coefficient of $x_s$ that could be zero),  and the other nonzero entries have the same sign and their sum is at least $-d$ and at most $d$, see the definitions \eqref{inqa}--\eqref{inqb}.
Hence,  the standard Euclidian norm  of any row of $ A_{m_{\inq, d},  r \times r} $ is at most
$$
(d^2 + (d-2)^2)^{1/2} < 2d
$$
as $d \ge 3$. Hence, by the Hadamard's inequality, we have that
\begin{equation}\label{cr1}
| \det A_{m_{\inq, d}, r \times r} | <    (2d)^{r}  .
\end{equation}

Invoking the Cramer's rule, we further obtain that
\begin{equation}\label{cr2}
y_{V,j} = \frac{\det A_{m_{\inq, d}, r \times r, j}^{\top}(\bar c^{\top})}{  \det A_{m_{\inq, d}, r \times r}}   ,
\end{equation}
where  $ A_{m_{\inq, d}, r \times r, j}^{\top}(\bar c^{\top})$ is the matrix obtained from $A^{\top}_{m_{\inq, d}, r \times r}$ by replacing the $j$th column with $\bar c^{\top}$, $j=1, \ldots, r$.  Since $\bar c^{\top}$ has a unique nonzero entry which is
$-1$, we have  from the Hadamard's inequality,   similarly to \eqref{cr1}, that
 \begin{equation}\label{cr3}
| \det A_{m_{\inq, d}, r \times r, j}(\bar c^{\top}) | < (2d)^{r-1} .
\end{equation}

In view of \eqref{cr1}--\eqref{cr3},  we can see that there is a common denominator
$C >0$ for the rational numbers $y_{V,1}, \dots,  y_{V,r}$  that satisfies $C < (2d)^{r}$ and that the nonnegative integers
$C y_{V,1}, \ldots,  C  y_{V,r}$  are less than $(2d)^{r-1}$.
\smallskip

Hence, it follows from the definition of the function $\sol_d$, see also Lemma~\ref{lem4}, that if $Q_V$ is a combination such that $y_{V} = \sol_d(Q_V)$
and $| Q_V |$ is minimal with this property, i.e.,  the entries of $\eta(Q_V)$  are coprime,
then
\begin{equation}\label{cr3a}
| Q_V | < r (2d)^{r-1} .
\end{equation}
Recall that the cardinality $| Q |$ of a combination with repetitions $Q$ is
defined so that every $q \in Q$ is counted as many times as it occurs in $Q$.
\medskip

We now construct a graph $Y_{2, Q_V}$ from $Q_V$ as described in the proof of Lemma~\ref{lem2}.
Recall that if $ \inq_d(V_SY_{2, Q_V}) =  \wtl Q_V $ then $\wtl Q_V$ could be different from $Q_V$ but $|\wtl Q_V | = |Q_V|$ and $Y_{2, Q_V}$  could also be constructed by means of $\wtl Q_V $.
\smallskip

It follows from  the definitions and Lemmas~\ref{lem4},~\ref{lem5}
that if
$$
\wtl y_V :=  \sol_d(\inq_d(V_S Y_{2, Q_V}))
$$
then the minimum of the dual LP-problem \eqref{dlpd}
is also attained at $\wtl y_V$ and this minimum is equal to $-\sigma_d(Y_1) \brr(Y_1)$. Hence,
\begin{equation*}
\brr(\core( Y_1 \times Y_{2, Q_V})) = \sigma_d(Y_1) \brr(Y_1) \brr(Y_{2, Q_V}) .
\end{equation*}

Since $| V_S  Y_{2, Q_V} | = |Q_V|$ and the degree of every secondary vertex of $Y_{2, Q_V}$ is at most
$d$,  it follows from \eqref{cr3a} that
 \begin{equation}\label{cr4}
| E Y_{2, Q_V} | \le  2d | V_S Y_{2, Q_V} | =  2d  | Q_V |  < r (2d)^{r}  .
\end{equation}

Note that $r$ does not exceed  the total number $n_{\inq}$ of variables of $\SLI[Y_1]$.
Since every primary vertex of $Y_1$ has degree 2 and edges of $Y_1$ are oriented,  we have  $| E Y_1 | = 4 | V_P Y_{1} |$.
Since each variable $x_B$  of $\SLI[Y_1]$, different from $x_s$, is indexed with a nonempty set
$B \subseteq V_P Y_1$,  it follows that
\begin{align} \label{cr5}  
r  \le   n_{\inq}  \le   ( 2^{ | V_P Y_1 |  }-1) + 1  =  ( 2^{ | E Y_1 |/4 }-1) + 1  = 2^{ | E Y_1 |/4 }  .
\end{align}

Finally, we obtain from \eqref{cr4}--\eqref{cr5} that
 \begin{align}
 \begin{split}\label{cr5a}
| E Y_{2, Q_V} |  & <  r (2d)^{r}  \\
                & \le 2^{ | E Y_1 |/4 } \cdot  (2d)^{ 2^{ | E Y_1 |/4 }} \\
 & =   2^{ | E Y_1| /4  } \cdot  2^{  (\log_2 (2d) )   \cdot  2^{ | E Y_1 |/4  }  } \\
 & <   2^{  (\log_2 (2d) +1) \cdot  2^{ | E Y_1 |/4}  } \\
 & =  2^{ 2^{| E Y_1 |/4  + \log_2 \log_2 (4d)  } }   ,
 \end{split}
\end{align}
as desired.
\smallskip

It remains to show that the graph $Y_{2, Q_V}$ is connected.

Arguing on the contrary, assume that  the graph   $Y_{2,Q_V}$   is the disjoint union of its two subgraphs   $Y_3$ and $ Y_4$. First we assume that
\begin{equation}\label{case1}
\brr(Y_3) >0  \quad   \text{and}  \quad   \brr(Y_4) >0 .
\end{equation}

Clearly, $Y_3$ and $Y_4$ are  graphs with property (Bd).
Recall that the secondary vertices of the graph   $Y_{2,Q_V}$ bijectively correspond to the inequalities of the combination $Q_V$,
see the proof of Lemma~\ref{lem2}.  In particular, we can consider  the combinations $Q_{3}$ and $Q_{4}$, whose inequalities bijectively correspond to the  secondary  vertices of $Y_3$ and $Y_4$, resp.  It is clear that $Q_V$ is the union of the combinations $Q_{3}$ and $Q_{4}$ and
\begin{equation}\label{etaQ}
\eta(Q_V) = \eta( Q_{3} ) +  \eta(  Q_{4}) .
\end{equation}
We specify that by the union $B_1  \sqcup B_2$  of two combinations  $B_1, B_2$ we mean the combination whose elements are all elements of both $B_1  $ and $ B_2$, in particular, $|B_1  \sqcup B_2| =  |B_1| + |B_2|$.
\smallskip

Furthermore, the graphs  $Y_3$ and $Y_4$ could be
constructed from $Q_{3}$ and $Q_{4}$, resp., in  the same manner as $Y_{2,Q_V}$ was constructed from $Q_{V}$.
In particular,  the combinations  $Q_{3}$ and $Q_{4}$ belong to the domain of  the function $\sol_d$.
\smallskip

Invoking  Lemma~\ref{lem4}, denote $y_V(j)  := \sol_d( Q_{j} )$,  $j=3,4$.
We also denote
$$
\sum_{q \in Q_V}q^L =  - C(Q_V) x_s ,  \quad  \sum_{q \in Q_{j} } q^L =  - C(Q_{j}) x_s ,
$$
where $j = 3,4$.
\smallskip

Since $ Q_V =  Q_{3} \sqcup Q_{4}$, it follows that  $ C(Q_V) =  C(Q_{3}) + C(Q_{4})$.
According to the definition  \eqref{dfyy} of the function $\sol_d$, we have
\begin{equation}\label{etayX}
y_{V,i}  = \frac{\eta_i(Q_{V} ) }{ C(Q_{V}) } ,   \qquad   y_{V,i}(j)  = \frac{\eta_i(Q_{j}) }{ C(Q_{j})  }
\end{equation}
for all suitable $i, j$.
Hence, in view of  \eqref{etaQ}, for every $i = 1, \dots, m_{\inq, d}$, we obtain
\begin{align}\label{eqQC}
\begin{split}
y_{V,i} & = \frac{\eta_i(Q_{V} ) }{ C(Q_{V}) }   = \frac{\eta_i(Q_{3} )  +\eta_i(Q_{4} ) } { C(Q_{V}) }  \\
 & = \frac{C ( Q_{3} )  } { C(Q_{V}) } \cdot  \frac{\eta_i(Q_{3} )  } { C(Q_{3}) }      +
     \frac{C ( Q_{4} )  } { C(Q_{V}) }  \cdot   \frac{ \eta_i(Q_{4} ) } { C(Q_{4}) } \\
 & =  \lambda_3  y_{V, i}(3) +  \lambda_4  y_{V, i}(4)  ,
\end{split}
\end{align}
where $\lambda_3 =  \frac{C ( Q_{3} )  } { C(Q_{V}) } $ and  $\lambda_4 =  \frac{C ( Q_{4} )  } { C(Q_{V}) } $ are positive rational numbers that satisfy $\lambda_3+ \lambda_4 =1$.
\smallskip

The equalities  \eqref{eqQC} imply that
\begin{equation}\label{ylam}
y_V = \lambda_3  y_V(3) +  \lambda_4  y_V(4)  .
\end{equation}

Since $y_V$ is a vertex of the polyhedron  \eqref{fpdp},  $y_V(3)$ and $ y_V(4) $ are vectors in  \eqref{fpdp},    and $0 < \lambda_3, \lambda_4 < 1$,  $\lambda_3+ \lambda_4 =1$, it follows  from  \eqref{ylam}  that
$$
y_V(3) =  y_V(4) =  y_V .
$$
Hence, in view of \eqref{etayX}, the tuples $\eta(Q_V)$, $\eta( Q_{3} )$, $ \eta(  Q_{4})$ that have integer entries  are rational multiples of each other.
Referring to \eqref{etaQ}, we conclude that the entries  of $\eta(Q_V)$  are not coprime, contrary to the definition of the combination $Q_V$.  This contradiction completes the case  \eqref{case1}.
\medskip

We now assume that  the graph   $Y_{2,Q_V}$   is the disjoint union of its two subgraphs   $Y_3$ and $ Y_4$ such that
\begin{equation}\label{case2}
\brr(Y_3) >0  \quad   \text{and}  \quad   \brr(Y_4) = 0 .
\end{equation}
\smallskip

Let $2Q_V$ denote  the combination such that  $\eta(2Q_V) = 2 \eta(Q_V)$, i.e., to get $2Q_V$  from $Q_V$ we double the number of occurrences of each inequality in $Q_V$. Using this combination $2Q_V$, we can construct, as in the proof of Lemma~\ref{lem2}, a  graph $Y_{2, 2Q_V}$ which consists of two disjoint copies of $Y_{2,Q_V}$, denoted $\bar Y_{2,Q_V}$ and $\wht  Y_{2,Q_V}$. Since  $Y_{2,Q_V} = Y_3 \cup Y_4$, we can represent
the graph $Y_{2, 2Q_V}$ in the form
$$
Y_{2, 2Q_V} = Y_5    \cup Y_6,
$$
 where $ Y_5 :=  \bar Y_3  \cup \bar Y_4  \cup \wht Y_4$ and  $ Y_6 :=  \wht Y_3$
 \smallskip

Clearly, $\brr(Y_5) >0 $,  $\brr(Y_6) >0 $,   and  both  $Y_5,  Y_6$ have  property (Bd).
As above, we remark that the  secondary vertices of  $Y_{2, 2Q_V}$ are in bijective correspondence with the inequalities of  $2Q_V$. Hence,   the combination $2Q_V$ is the  union of the combinations $Q_{5}$ and $Q_{6}$ that consist of those inequalities that correspond to the  secondary vertices of   $Y_5$ and $Y_6$, resp., and that can be used to construct the graphs   $Y_5$ and $Y_6$ in the same manner as $Y_{2,Q_V}$ was constructed from $Q_V$.
\smallskip

As above, we can write
\begin{equation}\label{etaQa}
\eta(2Q_V) =    \eta( Q_{5} ) +  \eta(  Q_{6}) .
\end{equation}

Note that  the combinations  $Q_{5}$ and $Q_{6}$ belong to the domain of  the function $\sol_d$.
Using  Lemma~\ref{lem4}, denote $y_V(j)  := \sol_d( Q_{j} )$,  $j=5,6$.
As above, denote
$$
\sum_{q \in 2Q_V}q^L =  - C(2Q_V) x_s ,  \quad  \sum_{q \in Q_{j} } q^L =  - C(Q_{j}) x_s ,
$$
where $j = 5,6$.
\smallskip

Since $ 2Q_V =  Q_{5} \sqcup Q_{6}$, it follows that  $ C(2Q_V) =  C(Q_{5}) + C(Q_{6})$.
According to the definition  \eqref{dfyy} of the function $\sol_d$, we have
\begin{equation}\label{etaya}
y_{V,i}  = \frac{\eta_i(Q_{V} ) }{ C(Q_{V}) } = \frac{\eta_i(2Q_{V} ) }{ C(2Q_{V}) } ,  \qquad   y_{V,i}(j)  = \frac{\eta_i(Q_{j}) }{ C(Q_{j})  }
\end{equation}
for all suitable $i, j$.
Hence, in view of  \eqref{etaQa}, for every $i = 1, \dots, m_{\inq, d}$, we obtain
\begin{align}\label{eqQCa}
\begin{split}
y_{V,i} & = \frac{\eta_i(2Q_{V} ) }{ C(2Q_{V}) }   = \frac{\eta_i(Q_{5} )  +\eta_i(Q_{6} ) } { C(2Q_{V}) }  \\
 & = \frac{C ( Q_{5} )  } { C(2Q_{V}) } \cdot  \frac{\eta_i(Q_{5} )  } { C(Q_{5}) }      +
     \frac{C ( Q_{6} )  } { C(2Q_{V}) }  \cdot   \frac{ \eta_i(Q_{6} ) } { C(Q_{6}) } \\
 & =  \lambda_5  y_{V, i}(5) +  \lambda_6  y_{V, i}(6)  ,
\end{split}
\end{align}
where $\lambda_5 =  \frac{C ( Q_{5} )  } { C(2Q_{V}) } $ and  $\lambda_6 =  \frac{C ( Q_{6} )  } { C(2Q_{V})}$ are positive rational numbers that satisfy $\lambda_5+ \lambda_6 =1$.
\smallskip

The equalities  \eqref{eqQCa} imply that
\begin{equation}\label{ylama}
y_V = \lambda_5  y_V(5) +  \lambda_6  y_V(6)  .
\end{equation}

Since $y_V$ is a vertex of the polyhedron  \eqref{fpdp}, $y_V(5)$ and $ y_V(6) $ are vectors in the polyhedron  \eqref{fpdp},     and $0 < \lambda_5, \lambda_6 < 1$,  $\lambda_5+ \lambda_6 =1$, it follows  from  \eqref{ylama}  that
$$
y_V(5) =  y_V(6) =  y_V.
$$
Hence, in view of \eqref{etaya}, the tuples $\eta(2Q_V)$, $\eta( Q_{5} )$, $ \eta(  Q_{6})$ that have integer entries are rational multiples of each other.
Referring to \eqref{etaQa} and keeping in mind that   the  entries  of $\eta(Q_V) $  are  coprime, we conclude that  \begin{equation}\label{QVe}
\eta(Q_V) = \eta( Q_{5} ) = \eta(  Q_{6}) ,
\end{equation}
i.e., $Q_V = Q_{5} = Q_{6}$.   However, $Y_6 = \wht Y_3$ and $\wht  Y_3$ is a subgraph of $\wht  Y_{2,Q_V}$ that consists of several connected components of $\wht  Y_{2,Q_V}$ and  $\wht  Y_3 \ne \wht   Y_{2,Q_V}$. Hence, $ Q_{5} \ne  Q_V $.
This contradiction to \eqref{QVe} completes the second case \eqref{case2}. Thus  the graph
 $Y_{2,Q_V}$ is connected. The proof of
Lemma~\ref{lem6} is complete.   \end{proof}

\section{More Lemmas}

We now let $\FF = \prod_{\al \in I}^* G_\al$ be  an arbitrary free product of nontrivial groups  $G_\al$,  $\al \in I$, and $|I| >1$. Let $H$ be a finitely generated factor-free subgroup of  $\FF$. As in Section~2, let $\Psi_o(H)$ denote an irreducible
$\A$-graph  of $H$, where $\A =   \bigcup_{\al \in I} G_\al$,  with the base vertex $o$ and
let $\Psi(H)$ denote the core of  $\Psi_o(H)$.
\smallskip

Let $I(H)$ denote a subset of  the index set $I$ such that $\al \in I(H)$  if and only if  there is a secondary vertex $u \in V_S \Psi(H)$ of type $\al$. Since $H$ is finitely generated,  it follows that the set  $I(H)$ is finite.
\smallskip

Let us fix a finitely generated factor-free subgroup $H_1$ of  $\FF$ with positive reduced rank
$\brr(H_1) = -\chi(\Psi(H_1)) >0$.
\smallskip

We say that a  finitely generated factor-free subgroup $H_2$ of  $\FF$ has  {\em property } (B) (relative to $H_1$)  if the core graph $\Psi(H_2)$ of $H_2$ has the original property (B) in which the graphs $Y_1$ and $Y_2$ are replaced with core graphs $\Psi(H_1)$ and $\Psi(H_2)$, resp., i.e., $\brr(H_2) = -\chi(\Psi(H_2)) >0$  and  the map
$$
\tau_2 : \core (\Psi(H_1) \times \Psi(H_2)) \to  \Psi(H_2)
$$
is surjective.
\smallskip

Let $d \ge 3$ be an integer.
Analogously, we   say that a  finitely generated factor-free subgroup $H_2$ of  $\FF$ has  {\em property } (Bd)   (relative to $H_1$)  if the core graph $\Psi(H_2)$ of $H_2$ has the original property (Bd) in which the graphs $Y_1$ and $Y_2$ are replaced with core graphs $\Psi(H_1)$ and $\Psi(H_2)$, resp., i.e.,
$$
\brr(H_2) = -\chi(\Psi(H_2)) >0,   \quad   \deg \Psi(H_2)  \le d
$$
and  the map $\tau_2 : \core (\Psi(H_1) \times \Psi(H_2)) \to  \Psi(H_2)$
is surjective.
\smallskip

Recall that if   $\Gamma$ is a  finite graph then  $\deg \Gamma$ is the maximum degree of a vertex of $\Gamma$.

\begin{lem}\label{lemBd} Suppose $H_2$ is a  finitely generated factor-free subgroup of  $\FF$ such that
$\deg \Psi(H_2) \le d$,  where $d \ge 3$ is an integer or $d = \infty$,  $\brr(H_2) = -\chi(\Psi(H_2)) >0$, and the map
$$
\tau_2 :  \core (\Psi(H_1) \times \Psi(H_2)) \to  \Psi(H_2)
$$
is not surjective. Then there exists a finitely generated factor-free subgroup $H_4$ of  $\FF$ with property (Bd) if $d < \infty$ or
with property (B) if $d = \infty$  such that
\begin{equation}\label{bol}
    \frac{\brr(H_1, H_4) }{  \brr(H_4) }  >  \frac{\brr(H_1, H_2) }{  \brr(H_2) }  . \end{equation}
\end{lem}

\begin{proof} Recall that  $ \brr(  H_1, H_2 )=   \brr( \core (\Psi(H_1) \times \Psi(H_2)) )$ and
$ \brr(  H_i) = \brr(\Psi(  H_i))$, $i=1,2$.
If $\brr( \core (\Psi(H_1) \times \Psi(H_2)) ) = 0$, then we may take $H_4 = H_1$ and the inequality \eqref{bol} holds.  Assume that $\brr( \core (\Psi(H_1) \times \Psi(H_2)) ) > 0$ and that the map
$$
\tau_2 : \core (\Psi(H_1) \times \Psi(H_2)) \to  \Psi(H_2)
$$
is not surjective. Consider the subgraph
$\Gamma := \tau_2 (\core (\Psi(H_1) \times \Psi(H_2)) ) $ of $\Psi(H_2)$. It follows from the definitions and assumptions that $\brr(\Gamma) < \brr (\Psi(H_2))$ and
$$
\brr(  \core (\Psi(H_1) \times \Gamma )  ) = \brr(  \core (\Psi(H_1) \times \Psi(H_2))  ) >0,
$$
whence $ \brr( \Gamma ) > 0$. It is also clear that $\core(\Gamma  ) = \Gamma $.
Therefore,
\begin{equation}\label{e52}
\frac{\brr( \core (  \Psi(H_1) \times \Gamma   ))  }{  \brr(\Gamma ) }  > \frac{\brr( \core (  \Psi(H_1) \times \Psi(H_2)   ))  }{  \brr(\Psi(H_2) ) } .
\end{equation}

Let $\Gamma_1, \dots, \Gamma_k$  be  connected components of the graph $\Gamma $.
Since
$$
\brr( \core (  \Psi(H_1) \times \Gamma )) >0 ,
$$ it follows  that  $\brr( \Gamma ) >0$.
Note that the graph
$$
\core (  \Psi(H_1) \times \Gamma )
$$
consists of disjoint graphs $\core (  \Psi(H_1) \times \Gamma_j )$, $j=1, \dots, k$.
In particular,
$$
\brr(\Gamma) = \sum_{j=1}^k \brr(\Gamma_j) , \qquad   \brr( \core (  \Psi(H_1) \times \Gamma )) = \sum_{j=1}^k \brr( \core (  \Psi(H_1) \times \Gamma_j )) ,
$$
 hence,
\begin{equation}\label{e53}
\frac{\brr( \core (  \Psi(H_1) \times \Gamma   ))  }{  \brr(\Gamma ) } =  \frac{ \sum_{j=1}^k \brr( \core (  \Psi(H_1) \times \Gamma_j )) }{     \sum_{j=1}^k \brr(  \Gamma_j )      } .
\end{equation}

Note that if $ \brr(  \Gamma_j )  =0$ then $ \brr( \core (  \Psi(H_1) \times \Gamma_j )) =0$.
\smallskip

Let $\Gamma_{j^*}$ be chosen so that $\brr(\Gamma_{j^*}) >0$ and the ratio
$$
\frac{  \brr( \core ( \Psi(H_1) \times \Gamma_{j^*} )) } {  \brr( \Gamma_{j^*} )      }
$$
is maximal over those graphs $\Gamma_j$ with   $ \brr(  \Gamma_j )  >0$.  It follows from
$$
\brr(\Gamma) = \sum_{j=1}^k \brr(\Gamma_j) >0
$$
that such $j^*$ does exist.  It is not difficult to see that
$$
 \frac{ \sum_{j=1}^k \brr( \core (  \Psi(H_1) \times \Gamma_j )) }{     \sum_{j=1}^k \brr(  \Gamma_j )      }
  \le     \frac{  \brr( \core ( \Psi(H_1) \times \Gamma_{j^*} )) } {  \brr( \Gamma_{j^*} )      } .
  $$
This, together with \eqref{e52} and \eqref{e53},  implies that
$$
 \frac{  \brr( \core ( \Psi(H_1) \times \Gamma_{j^*} )) } {  \brr( \Gamma_{j^*} )      }  \ge
  \frac{  \brr( \core ( \Psi(H_1) \times \Gamma)) } {  \brr( \Gamma ) } >
  \frac{  \brr( \core ( \Psi(H_1) \times \Psi(H_2))) } {  \brr( \Psi(H_2) ) } .
$$

Hence, picking an arbitrary
primary vertex $v \in V_P \Gamma_{j^*}$ in $\Gamma_{j^*}$ as a base vertex, and letting $H_4 := H(\Gamma_{j^*, v})$, as in Lemma~\ref{Lm2}, we obtain a subgroup $H_4$ with the desired inequality \eqref{bol}.
\end{proof}

\begin{lem}\label{lemsup} The supremum
$$
\sup_{H_3} \bigg\{ \frac{ \brr(   H_1, H_3   )  }{  \brr(H_3 ) } \bigg\}
$$
over all finitely generated factor-free subgroups $H_3$ of  $\FF$ such that  $\brr(H_3 ) >0$ and $\deg  \Psi(H_3) \le d$,
where $d \ge 3$ is an integer or $d = \infty$,
is equal to  $\sup_{H_2} \bigg\{ \dfrac{ \brr(   H_1, H_2   )  }{  \brr(H_2 ) } \bigg\}$
over all finitely generated factor-free subgroups $H_2$ of  $\FF$ that possess  property (Bd) when $d < \infty$ or property (B) when  $d = \infty$,  and satisfy the condition $I(H_2) \subseteq I(H_1)$. In particular, we have
\begin{align*}
\sigma_d(H_1)\brr(H_1 )  &   = \sigma_d( \Psi(H_1))\brr(\Psi(H_1)) ,  \\
\sigma(H_1)\brr(H_1)     &   =  \sigma( \Psi(H_1)) \brr(\Psi(H_1)) .
\end{align*}
\end{lem}

\begin{proof}  The first claim follows from Lemma~\ref{lemBd} and the observation that
if the map
$$
\tau_2 :  \core (\Psi(H_1) \times \Psi(H_2)) \to  \Psi(H_2)
$$
is surjective  then $I(H_2) \subseteq I(H_1)$. The equalities follow from the first claim, the definitions of the numbers  $\sigma_d(H_1)$, $\sigma(H_1)$,
 $\sigma_d(\Psi(H_1))$, $\sigma(\Psi(H_1))$,  and  Lemma~\ref{lemBd}.
\end{proof}

In view of  Lemma~\ref{lemsup}, when investigating the supremum
 $$
 \sup_{H_3}  \bigg\{ \frac{ \brr(   H_1, H_3  )  }{  \brr(H_3 ) }  \bigg\}
 $$
over all finitely generated factor-free subgroups $H_3$ of  $\FF$ with
$\brr(H_3 ) >0$  and $\deg \Psi(H_3)  \le d$,  we may assume that the index set $I$
is finite, i.e., $I= I(H_1)$,  say, $I = \{ 1, \dots, m\}$,
and so  $\FF= G_1  *  G_2  * \ldots * G_{m}$.
\smallskip

Furthermore, in order to be able to make use of results of Sections~3--4,
we consider $\FF$ as the following  free product
$$
\FF_2(1) =   G_1 *  G(2, m)
$$
of two groups $G_1$ and  $G(2, m) := G_2  * \ldots * G_{m}$. Let $g_\al \in G_\al$ be some nontrivial element of  $G_\al$, $\al \in I = \{ 1, \dots, m \}$. For every  $a_\al \in G_\al$,  consider the map
\begin{equation}\label{map2}
a_\al \mapsto
  (g_{\al+1} \ldots  g_{m}  g_{1} \ldots  g_{\al })^{-1}  a_\al   g_{\al +1} \ldots  g_{m}  g_{1} \ldots   g_{\al}   ,
\end{equation}
where $g_{\al+1 }  \ldots   g_{m}  g_{1} \ldots  g_{\al }$ is a cyclic permutation of the
word $g_{1}  g_{2}  \ldots g_{m}$.
\smallskip

Recall that a subgroup $K$ of a group $G$ is called {\em antinormal}
if, for every $g \in G$, $g K g^{-1} \cap K \ne \{ 1\}$ implies $g \in K$.

\begin{lem}\label{lemmap2} Let $|I| = m \ge 3$ and let $H_1$ be a finitely generated factor-free subgroup of $\FF$.  Then the map \eqref{map2} extends to monomorphisms
 $$
 \mu : \FF \to \FF , \qquad \mu_2 : \FF \to  \FF_2(1)
 $$
 that have the following properties.

\begin{enumerate}

\item[(a)]  A word $U \in \FF$ with $|U| >1$ is cyclically reduced if and only if $\mu(U)$ is cyclically reduced.

\item[(b)]   The subgroups $\mu_2(\FF )$ and   $\mu(\FF )$ are  antinormal in $\FF_2(1)$ and $\FF$, resp.

\item[(c)]   $\mu_2(H_1)$ is a factor-free subgroup of $\FF_2(1)$ and $\mu(H_1)$ is factor-free in $\FF$.  Furthermore, $\deg \Psi(H_1) = \deg  \Psi(\mu_2(H_1))$.

\item[(d)]   If $K_1$ and $ K_2$ are  finitely generated  factor-free subgroups of $\FF$, then
$$\brr(K_1, K_2)   =    \brr(  \mu_2(K_1) , \mu_2(K_2) ) .$$

\item[(e)]    The supremum
$$
\sup_{H_2}  \bigg\{ \frac{ \brr(  H_1, H_2   )  }{  \brr(H_2 ) }  \bigg\}
$$
over all finitely generated factor-free subgroups $H_2$ of $\FF$
such that  $\brr(H_2) > 0$  and $ \deg  \Psi(H_2) \le d$,  where $d \ge 3$ is an integer,  does not exceed
the supremum
$$
\sup_{K_2} \bigg\{  \frac{ \brr(  \mu_2( H_1), K_2   )  }{  \brr(K_2 ) }  \bigg\}
$$
over all finitely generated factor-free subgroups $K_2$ of  $\FF_2(1)$ with property (Bd) relative to $\mu_2( H_1)$.
In particular,
$$
\sigma_d(H_1) \le  \sigma_d( \mu_2( H_1)) \quad   \text{ and} \quad   \sigma(H_1) \le  \sigma( \mu_2( H_1)) .
$$
\end{enumerate}
\end{lem}

\begin{proof}
It is clear that the map \eqref{map2} extends to homomorphisms
 $$
 \mu : \FF \to \FF , \quad \mu_2 :  \FF \to \FF_2(1) .
 $$
Note that if $a_1 \in G_{\al_1}$ and $a_2 \in G_{\al_2}$ are nontrivial elements and $\al_1 \ne \al_2$, then
$\mu(a_1) \mu(a_2) $ is a cyclically reduced word. This remark implies that the kernels of the maps $\mu, \mu_2$ are trivial, whence $\mu, \mu_2$  are monomorphisms.
\medskip

(a) It follows from the foregoing remark that  a word $U \in \FF$ with $|U| >1$ is cyclically reduced if and only if $\mu(U)$ is cyclically reduced.
\medskip

(b)  Let $U_1, U_2 \in \FF$ be reduced words and
$W \mu(U_1) W^{-1} = \mu(U_2)$ in $\FF$. Using induction on $|U_1| + |U_2|$, we will prove that $W \in \mu(\FF)$.
\smallskip

Suppose $U_1$ is not cyclically reduced and
$$
U_1 \equiv a_1 U_3 a_2 ,
$$
where $a_1, a_2 \in G_\al \setminus \{ 1\}$ are letters of $U_1$. Then we can replace $U_1$ with $U_1' := U_3 a_3$, where $a_3 \in G_\al$,    $a_3 = a_2 a_1$ in $G_\al$ if  $a_3 \ne 1$ or with  $U_1' := U_3$ if $a_3 = 1$, and we replace $W$ with $W' := W \mu_2(a_1)$. This way we obtain an equality
$$
W' \mu(U_1') (W')^{-1} \overset 0 =  \mu(U_2)
$$
in $\FF$ in which $|U_1'|+ |U_2| < |U_1| + |U_2|$. Hence, it follows from the induction hypothesis that $W \in \mu(\FF)$, as required. If  $U_2$ is not cyclically reduced, then, analogously to what we did above for $U_1$, we can decrease the sum  $|U_1| + |U_2|$ and use the induction hypothesis.
\smallskip

Thus we may assume that both words $U_1, U_2$ are
cyclically reduced.  By part (a), the words $\mu(U_1)$, $\mu(U_2)$ are also cyclically reduced.  Observe that if $$
W V_1 W^{-1} \overset 0 =  V_2
$$
in $\FF$, where $V_1, V_2$ are cyclically reduced and $W$ is reduced, then $V_2$ is a cyclic permutation of $V_1$. More specifically, there is a factorization
$$
V_1 \equiv V_{11}V_{12}
$$
and an integer $k$ such that if $k \ge 0$ then $W \equiv V_{12}V_{1}^k$ and if  $k \le 0$ then $W \equiv V_{11}^{-1}V_{1}^k$. In either case, $V_2 \equiv V_{12} V_{11}$. Applying this observation to the equality
$$
W \mu(U_1) W^{-1} \overset 0 = \mu(U_2)
$$
in $\FF$, we can see from \eqref{map2}, when $m \ge 3$,  that a cyclic permutation  of  $\mu(U_1)$ equal to $\mu(U_2)$ must have the form
$\mu(\bar U_1)$, where $\bar U_1$  is a cyclic permutation of $U_1$. For similar reasons, $W \equiv \mu(V)$  for some $V \in \FF$ and part (b) is proven for the subgroup $\mu(\FF)$. It now follows that $\mu_2(\FF)$  is also antinormal in $\FF_2(1)$.
\medskip

(c) Arguing on the contrary, suppose $H$ is a factor-free subgroup of $\FF$ and one of $\mu(H)$, $\mu_2(H)$ is not factor-free in $\FF$, $\FF_2(1)$, resp.  Then it follows from the definitions that  $\mu_2(H)$ is not factor-free in $\FF_2(1)$. Hence, there is a reduced word $U$
such that $U$ is not conjugate in $\FF$ to a word of length $\le 1$ and
\begin{equation}\label{e54}
    \mu(U) \overset 0 = WVW^{-1}
\end{equation}
in $\FF$, where $W$ is either empty or reduced and $V$ is either a letter of $G_1 \setminus \{ 1\}$ or $V$ is
a reduced word with no letters of $G_1$. Thus, $V$ is reduced and either $V \in G_1$ or $V \in G(2,m)$.
\smallskip

Assume that the word $U$ in \eqref{e54} is not cyclically reduced. Then
$$
U \equiv a_1 U_1 a_2 ,
$$
where $a_1, a_2 \in G_\al \setminus \{ 1\}$ are letters of $U$. If $a_1 a_2 = a_3$ in $G_\al$   and $a_3 \in G_\al \setminus \{ 1\}$, then the word $U' \equiv U_1 a_3$, similarly to $U$, is not conjugate to $\FF$ to a word of length $\le 1$ and $\mu(U')$, being conjugate to $\mu(U)$ in $\FF$, has a representation of the form \eqref{e54}, so $U$ can be replaced with
$U'$.  If $a_1 a_2 = 1$ in $G_\al$, then the word $U_1$ can be taken as $U$. Hence, by induction on $|U|$, we may assume that $U$ is cyclically reduced.
\smallskip

If the word  $WVW^{-1}$  in \eqref{e54} is not reduced, then there are words $W'$, $V'$ such that
$$
\mu(U) \overset 0 = W' V' (W')^{-1} ,
$$
$W'$, $V'$ have the foregoing properties of $W$, $V$, resp., and
$$
2|W'| + |V'| < 2|W| + |V| .
$$
Indeed, if, say  $W \equiv W_1 a_1$ and $V \equiv a_2 V_1$, where $a_1, a_2 \in G_\al \setminus \{ 1\}$, then we set $W' := W_1$ and $V'$ is a reduced word equal in $\FF$  to $a_1 a_2 V_1 a_1^{-1}$. Note that $W'$, $V'$ have the foregoing properties of $W$, $V$, resp., and
$$
|W'| =|W|-1 , \quad    |V'| \le |V|+1 ,
$$
whence  $2|W'| + |V'| < 2|W| + |V|$. Thus, by induction on
$2|W| + |V|$, we may assume that the word $W V W^{-1}$ in \eqref{e54} is  reduced.
\smallskip

Since $U$ is cyclically reduced and $|U| >1$, it follows from part (a) that $\mu(U)$ is cyclically reduced. Hence, the word $W$ is empty and  $\mu(U) \equiv V$, where $V$ is a single letter of $G_1 \setminus \{ 1\}$ or $V$ has no letters of $G_1$. However, neither situation is possible by the definition  \eqref{map2}. This contradiction completes the proof of the first statement of part (c).
\smallskip

Now we will prove the equality
$$
\deg \Psi(H_1) = \deg  \Psi(\mu_2(H_1))
$$
of part (c).
It follows from the definition  \eqref{map2} that the graph $\Psi(\mu_2(H))$
can be visualized as  a graph obtained from $\Psi(H)$ by subdivision of edges  of $\Psi(H)$ into paths in accordance with formula  \eqref{map2}  and subsequent ``mergers" of edges that have labels in $G_2 \cup \dots \cup G_{m}$. In particular, for every
vertex $v \in V \Psi(H)$ with   $\deg v >2$, there will be a unique vertex $u = u(v) \in V_S \Psi(\mu_2(H))$  of degree $\deg u = \deg v$ and this map $v \mapsto u(v)$ is bijective on  the sets of all vertices of $\Psi(H)$, $\Psi(\mu_2(H))$    of  degree $> 2$. Hence, the maximal degree of vertices of  $\Psi(\mu_2(H))$ is equal to that of $\Psi(H)$, as claimed.
\medskip

(d) By part (c), the subgroups  $\mu_2(K_1), \mu_2(K_2)$ of $\FF_2(1)$ are  factor-free
and the subgroups  $\mu(K_1)$, $\mu(K_2)$  of $\FF$ are also factor-free.
Let
$$
T(\mu_2(K_1), \mu_2(K_2) )
$$
be a  set of representatives of those double cosets
$\mu_2(K_1) U \mu_2(K_2)$ of $\FF_2(1)$, where $U \in \FF_2(1)$, that have the property
$$
\mu_2(K_1) \cap U \mu_2(K_2)U^{-1} \ne  \{ 1 \} .
$$
If $T \in T(\mu_2(K_1), \mu_2(K_2) )$, then it follows from the definition    of  the set   $T(\mu_2(K_1), \mu_2(K_2))$
that there are nontrivial
$V_i \in K_i$, $i =1,2$, such that
$$
T \mu_2(V_2) T^{-1} = \mu_2(V_1) \ne 1
$$
in $\FF_2(1)$.
By part (b), such an  equality  implies  $T \in \mu_2(\FF)$ (note $\mu_2$ could be replaced with $\mu$).   Now we can see that
there is a set $S(K_1, K_2) \subseteq \FF$ such that
$$
\mu_2(S(K_1, K_2)) = T(\mu_2(K_1), \mu_2(K_2) )
$$
and $S(K_1, K_2) $ is  a set of representatives of those double cosets $K_1 S K_2$ of $\FF$, $S \in \FF$, that have the property $K_1 \cap S K_2 S^{-1} \ne  \{ 1 \}$. Therefore,
$$
\brr(K_1, K_2)   := \sum_{S \in S(K_1, K_2)}  \brr( K_1 \cap S K_2 S^{-1}) =
\brr(  \mu_2(K_1) , \mu_2(K_2) )  ,
$$
as desired.
\medskip

(e)  This follows from  Lemma~\ref{lemsup}, parts (c)--(d) and definitions.
\end{proof}

\section{Proofs of Theorems}

For the reader's convenience, we restate Theorems~1.1--1.3 before proving them.

\begin{T1} Suppose that $\FF =G_1 * G_2$ is the free product of two nontrivial groups $G_1,  G_2$
and  $H_1$ is a  finitely generated factor-free noncyclic subgroup of $\FF$. Then the following  are true.

\begin{enumerate}
\item[(a)]  For every integer $d \ge 3$,  there exists a linear programming problem (LP-problem)
\begin{equation*}\tag{1.8}
\PP(H_1, d) = \max\{ c(d)x(d) \mid A(d)x(d) \le b(d)  \}
\end{equation*}
with integer coefficients whose solution is equal to $-\sigma_d(H_1) \brr (H_1)$.
\smallskip

\item[(b)]  There is a finitely generated factor-free subgroup $H_2^* $ of $\FF$, $H_2^*= H_2^*(H_1)$,
such that  $H_2^* $   corresponds to  a vertex solution of the dual problem
$$
\PP^*(H_1, d) = \min \{ b(d)^{\top}  y(d)  \mid A(d)^{\top}y(d) = c(d)^{\top} , \, y(d) \ge 0  \}
$$
of the primal LP-problem  \eqref{lpa} of part (a)  and
$$
\bar \rr(H_1, H_2^*)  =  \sigma_d(H_1)  \bar \rr(H_1) \bar \rr( H_2^*) .
$$
In particular,  the WN${}_d$-coefficient $\sigma_d(H_1)$ of $H_1$ is rational.

Furthermore, if $\Psi(H_1)$ and  $\Psi(H_2^*)$ denote irreducible core graphs representing subgroups $H_1$ and $H_2^*$, resp.,
and $| E \Psi |$ is the number of oriented edges in the graph $\Psi$, then
$$
| E  \Psi(H_2^*) | < 2^{  2^{| E  \Psi(H_1) |/4 + \log_2 \log_2 (4d)  } } .
$$

\item[(c)]  There exists a linear semi-infinite programming problem (LSIP-problem)
$\PP(H_1) = \sup \{ cx \mid Ax \le b  \}$ with finitely many variables in $x$ and with countably  many constraints in the system $Ax \le b$ whose dual problem
$$
\PP^*(H_1)  = \inf \{ b^{\top} y \mid A^{\top} y = c^{\top} , \, y \ge 0  \}
$$
has a solution equal to  $-\sigma(H_1) \brr (H_1)$.
\smallskip

\item[(d)]  Let the word problem for both groups $G_1, G_2$ be solvable
and let an  irreducible core graph $\Psi(H_1)$  of $H_1$ be given. Then the
LP-problem \eqref{lpa}  of part (a)   can be algorithmically written down and the
WN${}_d$-coefficient $\sigma_d(H_1)$ for $H_1$  can be computed.
In addition,  an irreducible core graph $\Psi(H_2^*)$ of the subgroup $H_2^*$ of part (b) can be algorithmically constructed.
\smallskip

\item[(e)]  Let both groups $G_1$ and $G_2$ be finite,  let $d_{m} := \max( |G_1|, |G_2|) \ge 3$, and
 let an irreducible  core graph $\Psi(H_1)$  of $H_1$ be given.
 Then the LP-problem \eqref{lpa} of part (a)   for $d = d_{m}$  coincides with the
LSIP-problem $\PP(H_1)$ of part (c) and the WN-coefficient $\sigma(H_1)$ for $H_1$ is
rational and computable.
\end{enumerate}
 \end{T1}

\begin{proof}[Proof of Theorem~\ref{th1}]
 We start with part (a).  Assume that
 $$
 I = \{ 1,2\} ,  \quad  \FF  = G_1 * G_2
 $$
 and $H_1$ is a finitely generated factor-free noncyclic subgroup of $\FF$. As in Section~2,  let $\Psi_o(H_1)$ denote a finite irreducible $\A$-graph of $H_1$ and let $\Psi(H_1)$ denote the core of   $\Psi_o(H_1)$. Conjugating $H_1$ if necessary, we may assume that
 $\Psi_o(H_1) = \Psi(H_1)$.
\smallskip

Denote $Y_1 := \Psi(H_1)$ and pick an integer $d \ge 3$.
As in Sections~3--4, consider the system of linear inequalities $\SLI_d[Y_1]$, see   \eqref{slid},  and the LP-problem
\begin{equation}\label{lppf}
    \max \{ - x_s \mid \SLI_d[Y_1] \}  .
\end{equation}

 According to  Lemma~\ref{lem5}, the maximum of the  LP-problem
\eqref{lppf} is equal to
 $$-\sigma_d(Y_1) \brr(Y_1), $$
where
$$
\sigma_d(Y_1)\brr(Y_1) = \sup_{Y_2}  \bigg\{ \frac{  \brr ( \core(Y_1 \times Y_2) )}{ \brr(Y_2) }  \bigg\}
$$
over  all finite irreducible $\A$-graphs $Y_2$ with property (Bd) relative to $Y_1$.  By Lemma~\ref{lemsup}, we have
$$
\sigma_d( Y_1 ) \brr(Y_1)   = \sigma_d(H_1)\brr(H_1), $$
as desired in part (a). Part (a) is proven.
\medskip

We will continue to use below the notation introduced in the proof of part~(a).
\smallskip

Part (b) follows from  Lemmas~\ref{lem5},~\ref{lem6} and their proofs in which the construction of the graph
$Y_{2, Q_V}$ is based on a vertex solution  $y_V$ to the dual LP-problem \eqref{dlpd}.
To define the desired subgroup $H_2^*$  of  $\FF$  for $H_1$,
we can use the graph $Y_{2, Q_V}$ of Lemma~\ref{lem6}
as an  irreducible $\A$-graph $\Psi_{o^*}(H_2^*)$.  By Lemmas~\ref{lem5},~\ref{lem6},~\ref{lemsup},
the subgroup $H_2^*$ has all of the desired properties.
Part (b) is proven.
 \medskip

To prove part (c), we note that it follows from Lemmas~\ref{lem5} and \ref{lemsup}  that the dual problem  \eqref{dlp} of the LSIP-problem $ \sup \{ - x_s \mid \SLI[Y_1] \}$, where $Y_1 = \Psi(H_1)$ as above,  has the infimum equal to
$
-\sigma(Y_1) \brr(Y_1)  = -\sigma(H_1) \brr(H_1) .
$
This proves part (c).
\medskip

Now we turn to parts (d)--(e) of Theorem~1.1.
First we discuss  how to algorithmically write down inequalities of the system $\SLI_d[Y_1]$, where $d \ge 3$ is a fixed integer. Recall that every inequality of $\SLI[Y_1]$ is written in the form \eqref{inqa}--\eqref{inqb} and there are finitely many subsets $A \subseteq S_1(V_PY_1)$ that are indices of $k$ variables $\pm x_A$ in the left hand sides of inequalities  \eqref{inqa}--\eqref{inqb}, where $2 \le k = |T| \le d$. The coefficient of $x_s$ is the integer $-(k-2)$ and the right hand side of \eqref{inqa}--\eqref{inqb} is an integer $-N(\Om_T)_{\al}$, where
$$
0 \le N(\Om_T)_{\al} \le (d-2) |   V_P Y_1| ,
$$
see \eqref{cd2}. This information is sufficient to conclude that the set of inequalities in the system $\SLI_d[Y_1]$ is finite. However, this information is not sufficient to algorithmically write down
inequalities of $\SLI_d[Y_1]$ because the set of available sets $T$ is infinite whenever the union $G_1 \cup G_2$ is infinite.
\smallskip

To algorithmically write down  the system $\SLI_d[Y_1]$, we assume that
the word problem for both groups $G_1, G_2$ is solvable and we will look more closely into the definition of inequalities \eqref{inqa}--\eqref{inqb}.
\smallskip

Recall that inequalities \eqref{inqa}--\eqref{inqb} are defined in Section~3 by using an $\al$-admissible function $\Omega_T : T \to S_1( V_P Y_1)$, where $T \in S_2(G_\al)$ and $\al \in I = \{1,2\}$.
\smallskip

We also recall that $\sim_{\Omega_T}$ denotes an equivalence relation on the set of all pairs $(a, u)$, where $a \in T$ and $u \in \Omega_T(a)$, see Section~3. Making use of the equivalence relation $\sim_{\Omega_T}$, we define a relation  $\approx$ on the set  $T$ so that $a \approx b$ if and only if  there are  $u \in \Omega_T(a)$ and $v \in \Omega_T(b)$ such that
$$(a,u) \sim_{\Om_T} (b,v) .
$$
Note that this relation $\approx $  is reflexive and symmetric. The transitive closure of  the relation $\approx $ is an equivalence relation on $T$ which we denote by  $\approx_{{\Om_T} } $. The equivalence class of $a \in T$  is denoted  $[a]_{\approx_{\Om_T}}$.   It follows from the definition of  $[a]_{\approx_{\Om_T}}$ and from the property of being $\al$-admissible for $\Om_T$ that, for every $b_1 \in
[a]_{\approx_{\Om_T}}$, there is an element $b_2 \in [a]_{\approx_{\Om_T}}$ such that $b_2 \ne b_1$ and there are edges $e_1, e_2 \in EY_1$ such that  $(e_1)_+ = (e_2)_+ \in V_S Y_1$, the vertex $(e_1)_+$ has  type $\al$, $(e_i)_- \in \Om_T (b_i)$, $i =1,2$, and
\begin{equation}\label{t3t3}
    b_1 b_2^{-1} = \ph(e_1) \ph(e_2)^{-1}
\end{equation}
in $ G_\al$.
 Note that if we connect every two such elements $b_1, b_2 \in [a]_{\approx_{\Om_T}}$  by an edge, then the graph $\Gamma(\Om_T)$,  whose vertex set is $T$, will have connected components whose vertex sets are equivalence classes  $[a]_{\approx_{\Om_T}}$ of $T$.
This connectedness of subgraphs  of $\Gamma(\Om_T)$ on vertex sets
$[a]_{\approx_{\Om_T}}$ obviously implies the following.

\begin{lem}\label{lem3t} The equations  \eqref{t3t3} can be used to determine all elements of the equivalence class  $[a]_{\approx_{\Om_T}}$ for given $a \in G_\al$.
\end{lem}

\begin{proof} This easily follows from the definitions. Recall that the word problem is solvable in $G_\al$.
\end{proof}

Let
$
C(\al, d)
$
be a subset of $G_\al$ of cardinality
$$
|C(\al, d) | = d^2 +d,
$$
where   $\al =1,2$ and $d \ge 3$ is a fixed integer.
In the arguments below, this set  $C(\al, d)$  will be held fixed. Note that if $|G_\al| < d^2 +d$, so it is not possible to choose  $d^2 +d$ distinct elements in $G_\al$, then all inequalities \eqref{inqa} if $\al =1$ or $\eqref{inqb}$ if  $\al =2$  for $k \le d$, where as before $k = |T|$,  can be written down effectively for the following reasons. The sets
$$
S_2(G_\al) \quad  \mbox{and} \quad   \{ \Om_T \mid \Om_T : T \to S_1(V_P Y_1), T \in  S_2(G_\al) \}
$$
are finite, they can be written down explicitly, and it is possible to verify whether given function
$$
\Om_T :  T  \to S_1(V_P Y_1)
$$
is $\al$-admissible.
\medskip

Clearly, the same conclusion as above holds if both $G_1, G_2$ are finite but in the arguments below we will only need the equality $|C(\al, d) | = d^2 +d$, hence we can just assume that $|G_\al| \ge d^2 +d$.
\medskip

Consider a subset $C \subset  C(\al, d)$, where $1 \le |C| \le k \le d$,  and let  $Z = \{ z_1, \dots, z_{k-|C|} \}$ be  a set of indeterminates . Note  that $| C \cup Z | = k$.
Consider a function
\begin{equation}\label{fscz}
\Om_{ C \cup Z} : C \cup Z \to S_1(V_P Y_1) .
\end{equation}

Similarly to the relation $\sim_{\Omega_T}$ defined in Section~3, we introduce a relation $\sim_{\Om_{ C \cup Z}}$ on the set of all pairs $(a, u)$, where $a \in C \cup Z$ and $u \in \Om_{ C \cup Z}(a)$, defined as follows. Two pairs
$(a,u)$ and $(b,v)$ are related by $\sim_{\Om_{ C \cup Z}}$  if and only if  either
$(a,u)=  (b,v)$ or, otherwise, there exist edges $e, f \in EY_1$ such that $e_- = u$, $f_- = v$ and the secondary vertex $e_+ = f_+$ has type $\al$.
\medskip

We also consider an analogue  $\approx_Z$  of the relation $\approx$ defined above so that $a \approx_Z b$, where $a, b \in C\cup Z$, if and only if   there are
$$
u \in \Om_{ C \cup Z}(a), \quad  v \in \Om_{ C \cup Z}(b)
$$
such that
$
(a, u) \sim_{\Om_{ C \cup Z}}  (b, v) .
$
As before,  the  relation $\approx_Z$ is reflexive and symmetric.   By taking the transitive closure of the  relation $\approx_Z$ we obtain an equivalence relation on the set $C \cup Z$  which is denoted by $\approx_{\Om_{ C \cup Z}}$.
\medskip

We will say that a function $\Om_{{ C \cup Z}}$, as in \eqref{fscz}, is {\em unacceptable} if there is an equivalence class $[(a, u)]_{\sim_{\Om_{ C \cup Z}}}$ of $\sim_{\Om_{ C \cup Z}}$ with a single element in it or there is  an  equivalence class $[a]_{\approx_{\Om_{ C \cup Z}}}$  of the relation $\approx_{\Om_{ C \cup Z}}$ that
contains no elements of $C$. Note that, when given a function $\Om_{{ C \cup Z}}$ as in \eqref{fscz}, we can algorithmically check whether or not $\Om_{{ C \cup Z}}$ is unacceptable.
\medskip

If now the function $\Om_{C \cup Z}$ is not found to be unacceptable, then  we attempt to construct a function
$$
\zeta : Z \to G_\al
$$
by using the following algorithm.
\medskip

First, we set $\zeta_0(c) := c$ if $c \in C $  and let
$$
C_0 := C,  \quad Z_0 := \varnothing .
$$

Consider the set of all triples $(a, u, \ell)$, where $a \in C \cup Z$, $u \in  \Om_{C \cup Z}(a)$, $1 \le \ell  \le d+1$, and do the following. By induction on $i \ge 0$, assume that the sets
$$
C_i \subseteq G_\al ,  \quad Z_i \subseteq Z
$$
are constructed and a bijective function
$$
\zeta_i : C_0 \cup Z_i \to C_i
$$
is  defined so that the restriction of $\zeta_i$ on $C_0$ is  $\zeta_0$. For every unordered pair
$\{ (a, u, \ell),  (b, v, \ell) \}$ of distinct triples with a fixed $\ell$ (first we use $\ell =1$, then $\ell =2$ and so on up to $\ell  = d+1$), we check whether there are edges $e , f \in EY_1$ such that
$$
e_- = u , \quad f_- = v ,  \quad e_+ = f_+,
$$
and  $e_+ = f_+ \in V_SY_1$ has type $\al$.  If there are no such edges, then we pass on to the next pair
$\{ (a, u, \ell) ,  (b, v, \ell) \}$. If there are such edges $e, f$, then we consider three Cases 1--3 below, perform the described actions and pass on to the next pair. We remark that these actions can be
algorithmically implemented as follows from the solvability of the word problem for groups $G_1, G_2$ and the availability of the graph $Y_1 = \Psi(H_1)$.
\medskip

{\em  Case 1.} \  If both $a, b \in C \cup Z_i$, then we check whether the equality
 $$
 \zeta_i(a) \zeta_i(b)^{-1} = \ph(e) \ph(f)^{-1}
 $$
 holds in $G_\al$. If this equality is false, then we conclude that the function  $\Om_{C \cup Z}$ is  unacceptable and stop. Otherwise, we set
 $$
 Z_{i+1} := Z_{i} ,  \quad C_{i+1} := C_{i}  ,  \quad   \zeta_{i+1} := \zeta_{i} .
 $$
\medskip

{\em  Case 2.} \  Suppose that exactly one of  $a, b$ is in $ C \cup Z_i$, say $b \in C \cup Z_i$. Then it is clear that
$a \in Z \setminus Z_i$ and  we can uniquely determine an element $\xi(a)$ by solving the equation
  $\xi(a) \zeta_i(b)^{-1} = \ph(e) \ph(f)^{-1} $. If $\xi(a) \in C_i$, then we conclude that the function  $\Om_{C \cup Z}$ is  unacceptable and stop. Otherwise, we set
$$
Z_{i+1} := Z_{i}\cup \{ a \}  ,  \quad   C_{i+1} := C_{i}\cup \{ \xi(a) \}
$$
and define a function $\zeta_{i+1}$ on the set $C \cup Z_{i+1}$ so that  $\zeta_{i+1}(a) := \xi(a)$ and the restriction of  $\zeta_{i+1}$ on $C \cup Z_{i}$ is $\zeta_{i}$.
\medskip

  {\em  Case 3.} \  If both $a, b \not\in C \cup Z_i$, then we  set
$$
Z_{i+1} := Z_{i}   ,  \quad    C_{i+1} := C_{i}  ,  \quad   \zeta_{i+1} := \zeta_{i} .
$$

Cases 1--3 are complete.
\medskip

Since every equivalence class $[a]_{\approx_{\Om_{C \cup Z}}}$ contains an element
of $C$,  it follows from the definitions that while this algorithm runs over all pairs for a fixed $\ell' = 1, \dots, d$, one of the following three Cases (C1)--(C3) will occur.

\begin{enumerate}
\item[(C1)] For some $i$,  $|Z_{i+1}| = |Z_{i}|+1$.

\item[(C2)] The set  $\Om_{C \cup Z}$ is found to be  unacceptable.

\item[(C3)] For the index $i$, corresponding to the last pair $\{ (a, u, \ell'),  (b, v, \ell') \}$  for parameter $\ell$ equal to $\ell'$,  one has  $Z_{i} = Z$.
\end{enumerate}

Since $|Z| \le d-1$, we can see that it is not possible for Case (C1) to occur for all
$\ell' = 1, \dots, d$. Hence, running this algorithm
consecutively for $\ell' = 1, \dots, d$, results either in conclusion that the function   $\Om_{C \cup Z}$ is  unacceptable or in construction of a bijective function
$$
\zeta = \zeta_{i} : C \cup Z \to  C_i \subseteq G_\al ,
$$
where $Z_{i} =Z$, in which case we say that the function $\Om_{C \cup Z}$ is  {\em acceptable}.
Furthermore, setting
$$
T := \zeta(C \cup Z) \quad \mbox{and}  \quad  \Om_T( \zeta(a) ) := \Om_{C \cup Z}(a)
$$
for every $a \in C \cup Z$, we obtain an $\al$-admissible function $\Om_T$ on the set $T$, $T \subseteq G_\al$.
\medskip

Observe that the set of all such   functions
$$
\Om_{C \cup Z} :  C \cup Z \to S_1(V_PY_1) ,
$$
where $C \subseteq C(\al, d)$  and  $Z = \{ z_1, \dots, z_{k-|C|} \}$, see  \eqref{fscz},  is finite  (recall the set $C(\al, d)$ is fixed)  and that all such  functions   can be written down explicitly. Moreover, using the foregoing algorithm, we can verify whether a function $\Om_{C \cup Z}$ is acceptable and, when doing so, construct a unique function
$$
\zeta : C \cup Z \to S_1(V_PY_1) ,
$$
 where $T := \zeta( C \cup Z )$, so that $\Om_T( \zeta(a) ) := \Om_{C \cup Z}(a)$ for every $a \in C \cup Z$ and $\zeta(c) = c$ if $c \in C$. Therefore, in order to establish that inequalities \eqref{inqa}--\eqref{inqb} can be algorithmically written down, it remains to prove the following.

\begin{lem}\label{lemx} For every $\al$-admissible function
$$
\Om_{T'} : T' \to S_1(V_PY_1) ,
$$
where  $T' \subset G_\al$ and $2 \le |T'|=k \le d$, there exists an acceptable function
$$
\Om_{C \cup Z} : C \cup Z \to S_1(V_PY_1) ,
$$
where $C \subset C(\al, d)$ and $Z =   \{ z_1, \dots, z_{k-|C|} \}$, with the following property.
\smallskip

Let $T := \zeta(C \cup Z)$ and let
$$
\Om_{T} : T \to S_1(V_PY_1)
$$
be the  $\al$-admissible function, defined by
 $\Om_T( \zeta(a) ) := \Om_{C \cup Z}(a)$ for every $a \in C \cup Z$ and $\zeta(c) = c$ for $c \in C$. Then the two
 inequalities \eqref{inqa}, that correspond to  $\Om_{T'}$ and to $\Om_{T}$ if $\al =1$, or the two
 inequalities \eqref{inqb}, that correspond to  $\Om_{T'}$ and to $\Om_{T}$ if $\al =2$, are identical.
\end{lem}

To prove Lemma~\ref{lemx}, we first establish an auxiliary lemma.

\begin{lem}\label{lemy} Suppose
$$
\Om_{T'} : T' \to S_1(V_PY_1)
$$
is an $\al$-admissible function, where $2 \le |T'| \le d$, and $T' = E_1 \cup \dots \cup  E_r$ is a partition of $T'$ into equivalence classes $[a]_{\approx_{\Om_{T'}} }$ of the equivalence relation $\approx_{\Om_{T'}}$.
Then there are elements $h_1, \dots, h_r \in G_\al$ such that the set $$T :=  E_1 h_1 \cup \dots \cup E_r h_r$$ has the cardinality $ |T|= |T'|$ and every set $E_i h_i$, $i = 1, \dots, r$, contains an element from the set  $C(\al, d)$.
\end{lem}

\begin{proof} By induction on $i$, where  $1 \le i \le r$, we will prove the existence of elements
$h_1, \dots, h_i \in G_\al$ with the property that the set
$E_1 h_1 \cup \dots \cup E_i h_i$ has the cardinality
$$
\sum_{j=1}^i |E_j h_j|
$$
and every set $E_j h_j$, $j = 1, \dots, i$, contains an element from $C(\al, d)$.
\medskip

If $i=1$, then we set $h_1 := b^{-1} c$, where $b \in E_1$ and $c \in C(\al, d)$.
\medskip

Making the induction hypothesis, assume that there are elements
$h_1, \dots, h_i \in G_\al$ with the desired properties.
\medskip

To make the induction step from $i$ to $i+1$, denote
$$
C_i(\al, d) :=   C(\al, d) \cap  (E_1 h_1 \cup \dots \cup E_i h_i )
$$
and let $b \in E_{i+1}$. For an element $c \in C(\al, d) \setminus C_i(\al, d)$, we consider the set
$$
R_c := E_{i+1} b^{-1} c .
$$
Clearly, $R_c $ contains an element from $C(\al, d)$ and if $R_c $ is disjoint from the set  $E_1 h_1 \cup \dots \cup E_i h_i$, then we can set
$$
h_{i+1}  :=   b^{-1} c .
$$
Therefore, we may assume that $R_c$ contains an element from $E_1 h_1 \cup \dots \cup E_i h_i$ for every $c \in C(\al, d) \setminus C_i(\al, d)$.
\medskip

Suppose that elements in $E_1 h_1 \cup \dots \cup E_i h_i$ are indexed by integers from 1 to $|E_1 h_1 \cup \dots \cup E_i h_i|$ and elements in  $R_c = E_{i+1} b^{-1} c$, where $b$ and $c$ are chosen as above,  are indexed by integers from 1 to $|E_{i+1}|$ so that, for every $e \in E_{i+1}$,  the index of $e  b^{-1} c \in R_c$ is equal to that of $e \in E_{i+1}$.  In other words, we wish to keep indices stable when multiplying $E_{i+1}$ by  $b^{-1} c$.
\medskip

Making use of these indices, we fix an element $b \in E_{i+1}$ and, for every
$$
c \in C(\al, d) \setminus C_i(\al, d) ,
$$
we consider the pair $( j_R(c),  j_E(c))$ of indices  $j_R(c),  j_E(c)$  in $R_c = E_{i+1} b^{-1} c$  and in  $E_1 h_1 \cup \dots \cup E_i h_i$, resp., of an element of the intersection
$$
R_c  \cap (E_1 h_1 \cup \dots \cup E_i h_i)
$$
which is not empty as was assumed above.
\medskip

Suppose that  $( j_R(c_1),  j_E(c_1)) =  ( j_R(c_2),  j_E(c_2))$. Then it follows from the definitions that if $e_1, e_2 \in  E_{i+1}$ are such that
\begin{align*}
e_1 b^{-1} c_1 & \in R_{c_1} \cap (E_1 h_1 \cup \dots \cup E_i h_i) , \\
 e_2 b^{-1} c_2 & \in R_{c_2}\cap (E_1 h_1 \cup \dots \cup E_i h_i) ,
\end{align*}
then $e_1 = e_2$ and $e_1 b^{-1} c_1 = e_2 b^{-1} c_2$ in $G_\al$. These equalities imply that $c_1 = c_2$. Therefore,  for distinct elements $c_1, c_2 \in C(\al, d) \setminus C_i(\al, d)$, the pairs
$$
(j_R(c_1),  j_E(c_1)) ,  \quad  (j_R(c_2),  j_E(c_2))
$$
are also distinct. However, the number of elements $c$ in   $C(\al, d) \setminus C_i(\al, d)$ is
$$
|C(\al, d)|- | C_i(\al, d)| \ge (d^2 +d)-d = d^2
$$
and the number of all such pairs  $(j_R(c),  j_E(c))$  is less than $d^2$.  This contradiction completes the induction step and Lemma~\ref{lemy} is proved.
\end{proof}

\begin{proof}[Proof of Lemma~\ref{lemx}] Utilizing the notation of Lemma~\ref{lemy}, we let
 $$
 T' = E_1 \cup \dots \cup  E_r
 $$
 and let $h_1, \dots, h_r \in G_\al$ be elements  such that the set
 $$
 T :=  E_1 h_1 \cup \dots \cup E_r h_r
 $$
 has cardinality $ |T|= |T'| = k$ and every set $E_i h_i$, $i = 1, \dots, r$, contains an element from $C(\al, d)$.
 \medskip

Define a function
$$
\wht  \Om : T \to    S_1(V_PY_1)
$$
so that
if $a \in E_i$, $i = 1, \dots, r$, then  $\wht  \Om(ah_i) :=  \Om_{T'}(a)$.
 \medskip

Define $C := C(\al, d) \cap T$ and let $C = \{ c_1, \dots, c_{|C|} \}$.
Introducing more notation, denote
$$
T = \{ c_1, \dots, c_{|C|}, b_1, \dots, b_{k-|C|} \}
$$
and $Z = \{ z_1, \dots, z_{k-|C|} \}$.
\medskip

We also define a function
$$
\Om_{C \cup Z} :  C \cup Z  \to S_1(V_PY_1)
$$
by setting $\Om_{C \cup Z}(c_i) := \wht \Om (c_i)$ and $\Om_{C \cup Z} (z_j) :=\wht  \Om (b_j)$ for all $i,j$.
In view of Lemma~\ref{lem3t}, it is not difficult to see that the function $\Om_{C \cup Z} $ is acceptable, $\zeta(C \cup Z) = T$, and if
 $$
 \Om_{T} : T \to S_1(V_PY_1)
 $$
 is the function   defined by
$\Om_{T}(\zeta(a)) := \Om_{C \cup Z}(a)$ for every $a \in  C \cup Z$,
where  $\zeta(c) = c$ for $c \in C$, then the following hold true. The   function $\Om_{T}$ is $\al$-admissible, $\Om_{T}= \wht  \Om$, and   the two inequalities \eqref{inqa} if $\al =1$   or the two  inequalities   \eqref{inqb} if $\al =2$, corresponding to  $\Om_{T'}$ and to $\Om_{T}$, are identical.
Lemma~\ref{lemx} is proved.
\end{proof}

To finish the proof of part (d) of Theorem~\ref{th1}, we remark that, by  Lemma~\ref{lemx}, the LP-problem  $\max \{ - x_s \mid \SLI_{d}[Y_1] \}$ can be algorithmically written down. Solving this  LP-problem   we obtain, by Lemma~\ref{lem5}, the number  $- \sigma_d( Y_1)  \brr (Y_1)$ which is equal to  $- \sigma_d( H_1)  \brr (H_1)$ by Lemma~\ref{lemsup}. Since the number $\brr (Y_1) = \brr (H_1)$ is readily computable off the graph $Y_1$ (recall $\brr(Y_1) = |E Y_1 | /2 - |VY_1 |$),  it follows that the coefficient $\sigma_d( Y_1)$ is also computable.
\medskip

Since the LP-problem $\max \{ - x_s \mid \SLI_d[Y_1] \} $ can be effectively written down, its dual problem
\eqref{dlpd} can also  be effectively constructed. Using the notation of the foregoing proof of parts (a)--(c), we observe that a vertex solution $y_V = y_V(d)$ to \eqref{dlpd} can be computed, see \cite{S86}.
Hence, a  combination with repetitions
$Q_V$, such that  $\sol_d(Q_V) = y_V$ and  all entries in $\eta(Q_V)$  are coprime, is also computable,  see Lemma~\ref{lem4}.
\smallskip

Now, as in the proof of Lemma~\ref{lem2},  we can construct a graph $ Y_{2, Q_V} = \Psi(H_2^*)$ from  $Q_V$
 and observe that this construction can be done algorithmically. The proof of part~(d) is complete.
\medskip

 To show part (e), we note that if both groups $G_1, G_2$ are finite
 then any irreducible finite $\A$-graph  $\Psi$ has the property that $\deg u \le \max \{ |G_1 |, |G_2| \}$
 for every secondary vertex $u \in V_S \Psi$. Hence, setting
 $$
 d_{m}:= \max \{ |G_1 |, |G_2| \} ,
 $$
 we obtain  that $\SLI[Y_1] =  \SLI_{d_{m}}[Y_1]$ and so,
 by Lemma~\ref{lem5}, $\sigma( Y_1) = \sigma_{d_{m}}( Y_1)$.  Since the coefficient
 $\sigma_{d_{m}}( Y_1) =  \sigma_{d_{m}}( H_1)$
 is rational and computable by part (d), the  number $\sigma(H_1) = \sigma( Y_1)$ is also rational and computable.
 Theorem~\ref{th1} is proved.  \end{proof}

\begin{T2}  Suppose that $\FF =G_1 * G_2$ is the free product of two nontrivial finite groups $G_1,  G_2$ and $H_1$ is a subgroup of $\FF$ given by a finite generating set $\Ss$ of words over the alphabet
$ G_1 \cup G_2$. Then the following are true.
\begin{enumerate}

\item[(a)]  In deterministic polynomial time {}in the size of $\Ss$, one can
detect whether $H_1$ is factor-free and noncyclic and, if so, one can construct an  irreducible graph $\Psi_o(H_1)$  of $H_1$.
\smallskip

\item[(b)]  If  $H_1$ is factor-free and noncyclic, then, in deterministic exponential time {}in the size of $\Ss$, one can write down and solve an LP-problem $\PP = \max\{ cx \mid Ax \le b  \}$ whose solution is equal to $-\sigma(H_1) \brr (H_1)$. In particular, the $WN$-coefficient $\sigma(H_1)$ of $H_1$ is computable in exponential time {}in the size of $\Ss$.
\smallskip

\item[(c)]  If  $H_1$ is factor-free and noncyclic, then  there exists a finitely generated  factor-free subgroup
$H_2^* = H_2^*(H_1)$ of $\FF$  such that $$\bar \rr(H_1, H_2^*) =  \sigma(H_1)  \bar \rr(H_1) \bar \rr( H_2^*) $$ and the size of an irreducible core graph $\Psi(H_2^*)$ of $H_2^*$  is at most doubly exponential {}in the size of $\Psi(H_1)$. Specifically,
$$
| E \Psi(H_2^*) |  <  2^{ 2^{ | E \Psi(H_1) |/4 + \log_2 \log_2 (4d_m)  } } ,
$$
where   $\Psi(H_1)$ is an  irreducible core graph of $H_1$,   $| E\Psi |$ denotes the number of oriented edges of the graph $\Psi$, and  $d_m := \max( |G_1|, |G_2|)$.

In addition, an irreducible core graph $\Psi(H_2^*)$ of $H_2^*$  can be constructed
in deterministic exponential time {}in  the size of  $\Ss$ or  $\Psi(H_1)$.
\end{enumerate}
\end{T2}

\begin{proof}[Proof of Theorem~1.2] Part (a) follows from Lemma~\ref{Lm1}.
\medskip

To show part (b), we first observe that, in the case when $G_1$ and $ G_2$ are finite, we can  effectively write down  the system   $\SLI_{d_{m}}[Y_1]$, where $d_{m} =\max \{ |G_1|, |G_2| \}$,  and this can be done in exponential time {}in the size of $Y_1 := \Psi(H_1)$. Indeed,  the number of all functions
$$
\Omega_T :  T \to   S_1(V_P Y_1) ,
$$
where $T \in  S_2(G_\al)$ and $|T| \le d_m$,
 is bounded above by $2^{d_{m}}  2^{ | V_P Y_1|^{d_{m}}} = 2^{d_{m}}  2^{ (| E Y_1|/4)^{d_{m}}}$. Hence, we can construct all such functions in exponential time. We can also check whether every such function is $\al$-admissible in polynomial time {}in the size of  $Y_1$.
 Note that the input is the generating set $\Ss$ while the orders of finite groups
 $G_\al$, $\al = 1,2$, and the parameter $d_{m}$ are regarded as constants. Hence, all inequalities of the system $\SLI_{d_{m}}[Y_1]$ that are defined by means of $\al$-admissible functions  $\Omega$ as above, see definitions \eqref{inqa}--\eqref{inqb}, can be computed in exponential time {}in the size of $Y_1$.
\smallskip

Furthermore, by  Lemma~\ref{Lm1}, the size of  the graph $Y_1$ is polynomial {}in the size of the generating set $\Ss$.
By Theorem~1.1(e),  $\SLI[Y_1] =   \SLI_{d_{m}}[Y_1]$. Hence, the size of the system
$\SLI[Y_1] =   \SLI_{d_{m}}[Y_1]$ is exponential {}in the size of  $\Ss$. It is clear that the size of the primal LP-problem $\max \{ - x_s \mid \SLI_{d_{m}}[Y_1] \}$ as well as the size of the dual problem \eqref{dlpd}  are also exponential {}in the size of $Y_1$ or in the size of   $\Ss$.
By Theorem~\ref{th1} and Lemma~\ref{lemsup}, an optimal solution to the  dual problem \eqref{dlpd}  is equal to
$$
- \sigma_{d_{m}}( Y_1)  \brr (Y_1) = - \sigma_{d_{m}}( H_1)  \brr (H_1)
= - \sigma( Y_1)  \brr (Y_1) = - \sigma( H_1)  \brr (H_1) .
$$
It remains to mention that an LP-problem $\max \{cx \mid Ax \le b \} $ can be solved in polynomial time  {}in the size of the problem, see \cite{S86}, and that the reduced rank $\brr(Y_1) = \brr(H_1)$ can be computed in polynomial time {}in the size of  $Y_1$.
\medskip

To prove part (c), we recall that the size of the dual LP-problem \eqref{dlpd}, similarly to the size of the primal LP-problem  $\max \{ - x_s \mid \SLI_{d_{m}}[Y_1] \} $, is  exponential ({}in the size of $Y_1$ or $\Ss$) that a vertex solution  $y_V = y_V({d_{m}})$  to  \eqref{dlpd} can be computed in polynomial time {}in  the size of the dual LP-problem \eqref{dlpd}, see \cite{S86}. Note that here and below we use the notation of the proofs of proofs of Lemmas~\ref{lem2},~\ref{lem6}.
Hence, a vertex solution $y_V$  to \eqref{dlpd}  can be computed in exponential time ({}in the size of $Y_1$ or  $\Ss$). Using the function $ \sol_{d_{m}}$, we can compute    a  combination with repetitions $Q_V$,  such that  $\sol_{d_{m}}( Q_V) =   y_V $ and entries of  $Q_V$ are coprime, in polynomial
time {}in the size of $y_V$. The size of the vertex $y_V$, as was established in the proof of Lemma~\ref{lem6}, see
\eqref{cr1}--\eqref{cr3}, \eqref{cr5},  is  exponential. Hence, the  combination $Q_V$  can also be  computed in
 exponential time.
\smallskip

The inequality
$$
| E \Psi(H_2^*) | <  2^{  2^{| E \Psi(H_1) |/4 +  \log_2 \log_2 (4d_{m})  } },
$$
where, as above, $d_{m} = \max (|G_1|, |G_2| )$,  follows from  part~(d) of Theorem~\ref{th1}.

In view of inequalities \eqref{cr3a} and \eqref{cr5a}, we obtain that
\begin{align}\label{cr6}
|  Q_V| < r (2 d_{m})^{r-1} <  2^{ 2^{| E Y_1 |/4  + \log_2 \log_2 (4d_{m})  } }  .
\end{align}

This bound, in particular, means that  every inequality $q \in \SLI_{d_{m}}(Y_1)$ occurs in $Q_V$ less than
$$
2^{ 2^{| E Y_1 |/4  + \log_2 \log_2 (4d_{m})  } }
$$
times, hence, the number ${n}_{Q_V}(q)$ of occurrences of $q$ in $Q_V$ can be written by using
at most $2^{| E Y_1 |/4  + \log_2 \log_2 (4d_{m})  }$ bits.
\smallskip

As in the proofs of Lemmas~\ref{lem2},~\ref{lem6},  we construct a graph $Y_{2, Q_V}$ whose secondary vertices are in bijective correspondence with inequalities of $Q_V$ and whose primary vertices  are defined by means of an involution $\iota_V$ on the set of terms $\pm x_{D}$ of the left hand sides $q^L$ of the inequalities $q \in Q_V$.

\begin{lem}\label{lem8}
The graph $Y_{2, Q_V}$ can be constructed in  deterministic exponential time {}in the size of $Y_1$.
\end{lem}

\begin{proof} We need to explain how to compute the  involution $\iota_V$ as above in exponential time ({}in the size of $Y_1$).
To do this, for each variable $x_{D}$ of the system $\SLI_{d_{m}}(Y_1)$, see  \eqref{slio}, we consider a  graph $\Lambda_{D}$ whose set of vertices is the subset
 $$
 R_V := \{ q \mid q \in Q_V \}
 $$
 of $\SLI_{d_{m}}(Y_1)$ formed with the inequalities of $Q_V$. If $q_1, q_2 \in R_V$ are distinct, $q_1^L$ contains the term $ x_{D}$ and $q_2^L$ contains the  term $- x_{D}$,   then we draw an edge  in $\Lambda_{D}$ that connects $q_1$ and $q_2$.  In other words, if there is a potential cancellation between terms $\pm x_D$  in the sum $q_1^L + q_2^L$ then  $\Lambda_{D}$  contains an edge  that connects $q_1$ and $q_2$.
 \smallskip

It is clear that $\Lambda_{D}$  is a bipartite graph so that every edge connects a vertex of type \eqref{inqa} and a vertex of type \eqref{inqb}.
\smallskip

Consider a weight function
\begin{align}\label{wjD}
\omega_{D} : E  \Lambda_{D} \to \mathbb Z ,
\end{align}
where $\mathbb Z$ is the set of integers, such that
$\omega_{D}(e^{-1}) = \omega_{D}(e) \ge 0$ and
$$
\sum_{e_- = q} \omega_{D}(e) = n_q(x_D)   {n}_{Q_V}(q) ,
$$
where  $n_q(x_D) $ is the number of times the term $x_D$ or $-x_D$ occurs in $q^L$ and ${n}_{Q_V}(q)$ is  the number of occurrences of $q$ in $Q_V$. Clearly, $n_q(x_D) {n}_{Q_V}(q)$ is  the total number of
occurrences of terms  $\pm x_D$  in the subsum
\begin{align*}
\underbrace{q^L + \cdots +q^L}_{    {n}_{Q_V}(q)  \   \text {times } }
\end{align*}
of  the sum  $\sum_{q' \in Q_V} (q')^L$. Note that ${n}_{Q_V}(q) = \eta_j(Q_V)$ if $q = q_j$ in the notation of
$\eqref{solQ}$.
\smallskip

Our nearest goal is to show that such a weight function $\omega_{D}$  can be computed in exponential time for every index $D$.
\smallskip

Let the edge set
$$
E \Lambda_{D}  =  \{ e_1, e_1^{-1}, e_2, e_2^{-1}, \dots, e_{ | E \Lambda_{D} |/2  }, e_{ | E \Lambda_{D} |/2  }^{-1}     \}
$$
of the graph $\Lambda_{D}$  be indexed as indicated and let $(e_i)_-$ be a vertex of type  \eqref{inqa} for every $i$.
\smallskip

We will define the numbers
$\omd(e_i)$ by induction for $i =1,2, \ldots,  | E \Lambda_{D} |/2$  by the following procedure which also assigns
intermediate weights $\omega_{D}(q)$ to vertices $q \in R_V$ of $\Lambda_{D}$.
\smallskip

Originally, we set
$$
\omega_{D}(q) :=  n_q(x_D)    {n}_{Q_V}(q)
$$
for every $q \in R_V$.  For  $i \ge 1$, if the edge $e_i$ connects $q_1$ and $q_2$ then we set
$$
\omd(e_i) := \min(\omd(q_1), \omd(q_2) )
$$
and redefine the weights of $q_1$ and $q_2$ by setting
\begin{align*}
\omd'(q_1) & := \omd(q_1) - \min(\omd(q_1), \omd(q_2) ) , \\
\omd'(q_2) & := \omd(q_2) - \min(\omd(q_1), \omd(q_2) ) ,
\end{align*}
where $\omd'(q_1) $ denotes the new weight.
\smallskip

 Note that the assignment of a nonnegative weight $\omd(e_i)$ to the edge $e_i$, connecting $q_1$ and $q_2$, can be interpreted as making $\omd(e_i)$ cancellations between terms  $\pm  x_{D}$ of the subsums
\begin{align*}
 \underbrace{q_1^L + \cdots +q_1^L}_{\text { ${n}_{Q_V}(q_1)$ times }}  \quad \text{and}  \quad
 \underbrace{ q_2^L + \cdots +q_2^L}_{\text { ${n}_{Q_V}(q_2)$ times }}
\end{align*}
 of the sum in the left hand side of the equality
\begin{align}\label{QVr}
 \sum_{q \in Q_V} q^L = -2\brr(Y_1)x_s .
 \end{align}

 Analogously,  the intermediate weight $\omd(q_1)$   of a vertex  $q_1 \in V \Lambda_{D}$  can be interpreted as the number of terms  $\pm  x_{D}$ of the subsum
 $$
 \underbrace{
 q_1^L + \cdots +q_1^L }_{\text { ${n}_{Q_V}(q_1)$ times }}
 $$
 which are still uncancelled in the left hand side of \eqref{QVr}.
 \smallskip

 Therefore, in view of the equality \eqref{QVr},   in the end of this process, we will obtain that the weights $\omd(q)$ of all vertices $q \in R_V$ are zeros, i.e.,
 cancellations of the terms $\pm x_{D}$ are complete, and the weights $\omd(e_i)$ of all edges $e_i$  have desired properties.
 \smallskip

 Clearly,  the foregoing inductive procedure makes it possible to compute such a weight function $\omd$ in polynomial time {}in the size of the graph $\Lambda_D$ and in the size of  numbers $n_{Q_V}(q)$, $q \in R_V$, written in binary.  Hence, we can compute weight functions $\omd$ for all $D$  in exponential time.
\smallskip

Now we will define the involution $\iota_V$ based on the weight functions $\omd$.
\smallskip

Let  elements of the set $R_V = \{ q_1, \ldots, q_{|R_V|} \}$ be indexed as indicated and  let  elements of the combination
\begin{align}\label{QVc}
\begin{split}
Q_V  =  [[ & q_{1,1}, q_{1,2}, \ldots, q_{1,{n}_{Q_V}(q_1)}, \\
& \ldots,  \\
& q_{i,1}, q_{i,2}, \ldots, q_{i,{n}_{Q_V}(q_i)}, \\
& \ldots,  \\
   &  q_{|R_V|,1}, q_{|R_V|,2}, \ldots, q_{|R_V|,{n}_{Q_V}(q_{|R_V|})} ]] ,
\end{split}
\end{align}
where $q_{i,j} = q_i \in R_V$ for all possible $i, j$, be double indexed as indicated according to the indices introduced on elements of $R_V$.
\smallskip

Since the secondary vertices of the graph $Y_{2, Q_V}$  are in bijective correspondence with elements of $Q_V$, see the proof of Lemma~\ref{lem2}, we can also write
$$
V_S Y_{2, Q_V} = \{ u_{i,j} \mid 1 \le i \le |R_V|, \ 1 \le j \le {n}_{Q_V}(q_i) \} ,
$$
where
\begin{align}\label{uqC}
 u_{i, j}   \mapsto    q_{i, j}
\end{align}
under this correspondence.
\smallskip

Let $q_i \in R_V$ be fixed and  let
$$
q_{m_1(i)}, \ldots, q_{m_{t_i}(i)}
$$
be all vertices of $\Lambda_{D}$, where $ m_1(i) < \cdots < m_{t_i}(i)$, that are connected to $q_i$ by edges $f_1, \ldots, f_{t_i}$, resp.,  in $\Lambda_{D}$ with positive weights $\omd(f_1), \ldots, \omd(f_{t_i})$, resp.  We assume that  $q_i$ is the terminal vertex of  the edges $f_1, \ldots, f_{t_i}$.
\smallskip

Recall that  $q_i^L$ contains  $n_{q_i}(x_D) \ge 1$ terms $\pm x_{D}$,  here  the sign is a minus if
$q_i$ has type \eqref{inqa} and  the sign is a plus if
$q_i$ has type \eqref{inqb}.
\smallskip

According to the weights  $\omd(f_1), \ldots, \omd(f_{t_i})$, we will define $( D, i, t)$-blocks of consecutive terms $\pm x_{D}$ in the sum
\begin{align}\label{seqq}
q_{i,1}^L + q_{i,2}^L + \cdots  + q_{i,{n}_{Q_V}(q_i)}^L ,
\end{align}
see \eqref{QVc}, in the following manner.  (Here and below we disregard all terms $\pm x_B$, where $B \ne D$, in \eqref{seqq}  when we talk about consecutive terms $\pm x_{D}$ in \eqref{seqq}.)
\smallskip

The  $( D, i, 1)$-block consists of the first $\omd(f_1)$  consecutive  terms $\pm x_{D}$    in the sum
\eqref{seqq}.
The $(D, i, 2)$-block consists of the next $\omd(f_2)$   consecutive  terms $\pm x_{D}$ in the sum  \eqref{seqq}  and so on. Note that the first   term $\pm x_{D}$  of the  $(D, i, 2)$-block  is $(\omd(t_1) +1)$st term  $\pm x_{D}$ in the sum  \eqref{seqq}  and the last term $\pm x_{D}$  of the  $(D, i, 2)$-block  is the $(\omd(t_1) +\omd(t_2))$th  term  $\pm x_{D}$ in the sum  \eqref{seqq}.
\smallskip

The $(D, i, t_i)$-block consists of the last $\omd(f_{t_i})$ consecutive  terms $\pm x_{D}$ in the sum   \eqref{seqq}.
Since
$$
\sum_{t=1}^{t_i} \omd(f_t) =    n_{q_i}(x_D)   {n}_{Q_V}(q_i)
$$
and $\omd(f_t) >0$ for every $t$, it follows that these  $(D, i, t)$-blocks, where $t = 1, \dots, t_i$ and
$ D, i$  are fixed, will form a partition of the sequence of    terms $\pm x_{D}$   of  the sum \eqref{seqq} into $t_i$ subsequences. Note that the terms  $\pm x_{D}$   of  the same summand $q_{i,j}^L $ of \eqref{seqq}  could be in different blocks when $n_{q_i}(x_D)   > 1$.
\smallskip

We  emphasize  that every $(D, i, t)$-block is associated with a  vertex $q_i \in V \Lambda_{D} =  R_V$ and with an edge $f_t$ of $\Lambda_{D}$
so that $f_t$ ends in $q_i$ and $\omd(f_t) >0$.  In particular, for every $(D, i, t)$-block,  associated with a vertex $q_i \in R_V$ and with  an edge $f_t$ of  $\Lambda_{D}$, we have another $(D, i', t')$-block,  associated with a vertex $q_{i'} \in R_V$ and with an edge $f'_{t'}$ of $\Lambda_{D}$, so that $q_{i'}  \ne q_i$ and $f'_{t'} = f_t^{-1}$. Here $ f'_1, \ldots, f'_{t'_{i'}}$ are the edges of $\Lambda_{D}$ defined for $q_{i'}$ in the same fashion as the edges  $f_1, \ldots, f_{t_i}$ of $\Lambda_{D}$ were defined for $q_{i}$.
Note that $i'' = i$ and $f''_{t''} = f_t$ in this notation.
\smallskip

We define the involution $\iota_V$ so that all the  terms   $\pm x_{D}$  of the $(D, i, t)$-block are mapped by $\iota_V$ to  the  terms $\mp x_{D}$ of  the $(D, i', t')$-block  in the natural increasing order of elements in the block.
\smallskip

In other words,  this definition of the involution  $\iota_V$  means that the primary  vertices of the graph
$Y_{2, Q_V}$, for details see the proof of Lemma~\ref{lem2}, that are connected by edges to the secondary vertices
\begin{align}\label{ui1}
 u_{i,  1},    u_{i,  2},   \ldots,   u_{i,   n_{Q_V}(q_i) }
\end{align}
of $Y_{2, Q_V}$,  see \eqref{uqC},  and that correspond to the terms $\pm x_D$ of the $(D,i,t)$-block, will be identified, in the increasing order, with the primary vertices that are connected by edges  to  the secondary vertices
\begin{align}\label{ui2}
 u_{i',  1},    u_{i',  2},   \ldots,   u_{i',   n_{Q_V}(q_{i'}) }
\end{align}
of $Y_{2, Q_V}$ and that correspond to the terms $\mp x_D$ of the $(D,i',t')$-block.
\smallskip

The labels to the edges of the graph $Y_{2, Q_V} $ are assigned as described in the  proof of Lemma~\ref{lem2}.  Specifically,  let $e_{1, j}, \ldots, e_{k_i, j}$ be all the edges of
$Y_{2, Q_V}$ that end in a secondary vertex $u_{i, j}$, i.e.,
$$
(e_{1,j})_+ = \cdots = (e_{k_i,j})_+ = u_{i, j} ,
$$
where $k_i = k(q_i)$. Furthermore, let $\{ b_1, \ldots, b_{k_i} \}$ denote the domain of an $\al_i$-admissible  function
$$
\Omega_{T_i} :  \{ b_1, \ldots, b_{k_i} \} \to S_1(V_P Y_1)
$$
that defines the inequality $q_i$.
Then we set
$$
\ph(e_{1, j}) := b_1 ,  \  \ldots, \   \ph(e_{k_i, j}) :=b_{k_i}  .
$$

Note that the primary vertices that are discussed above and that are connected by edges to vertices \eqref{ui1} will be precisely those  $e_{\ell, j}$, among
$(e_{1,j})_-, \ldots$,  $(e_{k_i,j})_-$ over all $j = 1, \ldots,     n_{Q_V}(q_i)$,   for which
$$
 \Omega_{T_i}(    \ph(e_{\ell, j}) ) =  \Omega_{T_i}(   b_\ell   )   = D .
$$
Similar remark can be made about the primary  vertices that are discussed  above and that are connected by edges to vertices  \eqref{ui2}.
\smallskip

It is clear that the foregoing construction of the involution  $\iota_V$ can be done in polynomial time in the total size of graphs $\Lambda_D$, weights  $\omd(e)$,    $e \in E\Lambda_D$,   and numbers $n_{Q_V}(q)$, $q \in R_V$, written in binary.
Therefore, we can compute  $\iota_V$ in exponential time {}in the size of $Y_1$ (or $\Ss$).
Thus the graph $Y_{2, Q_V}$ can also  be constructed in exponential time, as required.
The proof of Lemma~\ref{lem8} is complete.
\end{proof}

\smallskip

Since the graph $Y_{2,  Q_V}$  can be constructed in exponential time {}in the size of the generating set  $\Ss$, it follows from Lemma~\ref{lem6} that we can use $ Y_{2,  Q_V}$ as an  irreducible $\A$-graph  $\Psi(H_2^*)$  of  the subgroup $H_2^*$. Theorem~\ref{th2} is proved.  \end{proof}

\medskip
It is worthwhile to mention that our construction of the  graph $Y_{2, Q_V}$ is somewhat succinct (cf. the definition of succinct representations of  graphs in \cite{PCC}) in the sense that,
despite the fact that the size of $Y_{2, Q_V}$ could be doubly exponential, we are able to give a description of $Y_{2, Q_V}$  in exponential time ({}in the size of $Y_1$).  In particular, vertices of  $Y_{2, Q_V}$    are represented by exponentially long bit strings and edges of $Y_{2, Q_V}$  are  drawn in blocks. As a result,
we can find out in exponential time  whether two given vertices of  $Y_{2, Q_V}$  are connected by an edge labelled by  given letter $g \in G_\al$.

\begin{T3}
Suppose  that $\FF = \prod_{\alpha \in I}^* G_\alpha$ is
the  free product  of  nontrivial groups   $G_\al$, $\al \in I$, and $H_1$ is a
finitely generated factor-free noncyclic subgroup of $\FF$.  Then there are two  disjoint finite subsets
$I_1, I_2$ of the index set $I$ such that if \
$\wht G_1 := \prod_{\alpha \in I_1}^* G_\alpha$,  \   $\wht G_2 := \prod_{\alpha \in I_2}^* G_\alpha$,
 and      $\wht \FF := \wht G_1 * \wht G_2$,
then there exists a finitely generated factor-free subgroup $\wht H_1 $ of $\wht \FF$  with the following properties.

\begin{enumerate}
\item[(a)]  $\brr (\wht H_1  ) = \brr (H_1  ) $,
$\sigma_d(\wht H_1) \ge \sigma_d(H_1)$ for every $d \ge 3$,  and
$\sigma(\wht H_1) \ge \sigma(H_1)$. In particular, if the conjecture
\eqref{conjs} fails for $H_1$ then the conjecture \eqref{conjs} also fails for $\wht H_1$.
\smallskip

\item[(b)]  If the word problem for every group $G_\al$, where $\al \in I_1 \cup I_2$,
is solvable and  a finite irreducible graph  of $H_1$ is given, then the
LP-problem $\PP(\wht H_1, d)$  for  $\wht H_1$   of part (a) of Theorem~\ref{th1}  can be algorithmically written down and the WN${}_d$-coefficient $\sigma_d(\wht H_1) $ for $\wht H_1$  can be computed.
\smallskip

\item[(c)]  Let every group $G_\al$, where $\al \in I_1 \cup I_2$,
be finite, let $H_1$ be  given either by
a finite irreducible graph or by a finite generating set, and let
$$
d_{M} := \max \Big\{ |I_1 \cup I_2| , \max\{ |G_\al |  \mid   \al \in I_1 \cup I_2 \} \Big\} .
$$
Then $\sigma_{d_{M}}(\wht H_1) \ge \sigma(H_1)$
and there is an algorithm that decides whether the conjecture  \eqref{conjs} holds for $H_1$.
\end{enumerate}
\end{T3}

\begin{proof}[Proof of Theorem~1.3]  (a)  As in the proof of Theorem~\ref{th1}, we assume that the  subgroup $H_1$ is given by an irreducible $\A$-graph $\Psi(H_1)$  with $\core(\Psi(H_1)) = \Psi(H_1)$,  now the alphabet is $\A = \bigcup_{\al \in I} G_{\al}$.  Note that it is also possible to assume that $H_1$ is defined by a finite generating set $\Ss$ whose elements are words over the alphabet $\A$. In the latter case, we could apply Lemma~\ref{Lm1} which, when given
a finite generating set of a subgroup $H$ of $\FF$, verifies that $H$ is a factor-free subgroup of $\FF$ and, if so, constructs  an irreducible $\A$-graph of $H$.
\smallskip

 Making use of the graph $\Psi(H_1)$ of $H_1$, we switch from the original index set $I$ to its finite subset $I(H_1)$ and rename it by $\{ 1, \dots, m \}$.
 Here and below we use the notation introduced in Section~5. Without loss of generality, we may assume that $m \ge 3$, otherwise, we set $\wht H_1 := H_1$.
 \smallskip

 Consider the  embedding
 $$
 \mu_2 : \FF \to \FF_2(1)
 $$
 defined by means of the map \eqref{map2}, where
 $$
 \FF_2(1) = G_1 * G(2,m)  \quad  \text{ and}  \quad G(2,m) = G_2 * \dots  * G_m .
 $$
 Denote $\wht H_1 := \mu_2(H_1)$.  By Lemma~\ref{lemmap2}, $\mu_2$ is a monomorphism, hence,  $\brr(\wht H_1)=  \brr(H_1)$ and, by Lemma~\ref{lemmap2}(e),
 $$
 \sigma_d(H_1) \le \sigma_d(  \wht H_1 )
 $$
 for every $d \ge 3$.  Consequently,   $\sigma(H_1) \le \sigma(  \wht H_1 )$ as well. This proves part (a).
\medskip

 (b) Assume that the word problem is solvable in groups $G_\al$, $\al \in I(H_1)$.
 Then the word problem is also solvable in factors $G_1,  G(2,m)$ of the free product $\FF_2(1) = G_1 * G(2,m)$. Furthermore, using the graph $\Psi(H_1)$ of $H_1$ and the map \eqref{map2}, we can algorithmically construct a finite irreducible  graph $\Psi(\wht H_1)$  with $\core(\Psi(\wht H_1)) = \Psi(\wht H_1)$. By Theorem~\ref{th1}(d), the LP-problem $\PP( \wht H_1, d) =  \PP( \Psi(\wht H_1), d) $, associated with $\wht H_1$, can be effectively constructed and the coefficient  $\sigma_d(\wht H_1)$ can be computed, as claimed in part (b).
\medskip

(c) We will continue to use the notation introduced above.
Suppose that all factors $ G_\al$, where $\al \in I(H_1)= \{ 1, \dots, m \}$, are finite. We also assume that  $H_1$ is  given by an irreducible graph $\Psi(H_1)$  with $\core(\Psi(H_1)) = \Psi(H_1)$ or $H_1$   is  given   by  a finite generating set.  Note that Lemma~\ref{Lm1} reduces the latter case to the former one.
By Lemma~\ref{lemsup}, when computing the number
$$
\sigma( H_1) = \sup_{H_2}  \bigg\{ \frac {\brr(H_1, H_2)}{\brr(H_1) \brr(H_2)}  \bigg\}
$$
over all finitely generated factor-free subgroups $H_2$ with $\brr(H_2) >0$,
we may assume that the subgroup $H_2$ has property (B)   and satisfies the condition $I(H_2) \subseteq I(H_1)$.  The condition  $I(H_2) \subseteq I(H_1)$ implies that the degree of every primary vertex of $\Psi(H_2)$ does not exceed  $ | I(H_1)|$. On the other hand,  the degree of every secondary vertex of
$\Psi(H_2)$ does not exceed
$$
\max \{ | G_\al |    \mid     \al \in   I(H_1) \} .
$$
Hence, the degree $\deg v$ of every vertex $v$ of $\Psi(H_2)$ satisfies
\begin{equation}\label{eq11}
 \deg v \le {d_{M}}  := \max \{ | I(H_1)|, \, \max \{   | G_\al |   \mid   \al \in   I(H_1)  \} \, \} .
 \end{equation}
 Thus, by Lemma~\ref{lemsup}, we may conclude that
 $$
 \sigma( H_1) =  \sigma_{d_{M}}( H_1) = \sup_{H_2}  \bigg\{ \frac {\brr(H_1, H_2)}{\brr(H_1)  \brr(H_2)}  \bigg\} ,
 $$
where the supremum is taken over all subgroups $H_2$ with property (Bd) in which $d = d_{M}$.
Applying Lemma~\ref{lemmap2}(e) to $H_1$, we obtain
$$
\sigma(H_1) =  \sigma_{d_{M}}( H_1) \le \sigma_{d_{M}}(\wht H_1) .
$$
Recall that an irreducible graph $\Psi(\wht H_1)$ of  $\wht H_1 = \mu_2(H_1)$  can be algorithmically constructed from $\Psi(H_1)$ (for details see the proof of Lemma~\ref{lemmap2}(c)) and that the word  problem is solvable for factors of the free product
$$
\FF_2(1) = G_1 * G(2,m).
$$
Invoking Theorem~\ref{th1}(d), we see that the LP-problem
 $$ \PP(\wht H_1, {d_{M}}) = \PP(\Psi(\wht  H_1), {d_{M}})$$
 can be algorithmically written down and hence the coefficient $\sigma_{d_{M}}(\wht H_1)$ can be computed.   The proof of Theorem~\ref{th3} is complete.
\end{proof}

In conclusion, we mention that it is not clear whether there is a duality gap between
the LSIP-problem $ \sup \{ - x_s \mid \SLI[Y_1] \}$, introduced in Section~4, and its dual problem \eqref{dlp} and
it would be of interest to find this out. Another natural problem
is to find an algorithm that solves the dual  problem  \eqref{dlp} of the LSIP-problem $ \sup \{ - x_s \mid \SLI[Y_1] \}$
and thereby effectively  computes the WN-coefficient $\sigma(\Psi(H_1)) =  \sigma(H_1)$ for a finitely generated factor-free subgroup $ H_1$ of the free product of two groups (and, perhaps, more than two groups) which are not necessarily finite.  It would also be interesting to find an algorithm  that computes the Hanna Neumann coefficient $\bar \sigma(H_1)$ for a finitely  generated factor-free noncyclic subgroup $ H_1$ of the free product $\FF$ of two finite groups which is defined as
$$
\bar \sigma(H_1) := \sup_{H_2}  \bigg\{ \frac {\brr(H_1 \cap H_2)} {\brr(H_1) \brr(H_2)}  \bigg\}
$$
over all finitely  generated factor-free noncyclic subgroups $H_2$ of $\FF$.
\medskip

{\em Acknowledgments.}  The author is grateful to the referee for helpful remarks and suggestions.

\end{document}